\documentclass[a4paper,12pt,oneside]{article}
\pdfoutput=1

\usepackage[english]{babel}
\usepackage[utf8]{inputenc}
\usepackage[T1]{fontenc}

\usepackage{amsmath,amssymb,amsthm,mathtools,graphicx,caption}
\usepackage[dvipsnames]{xcolor}
\usepackage{tikz-cd}

\usepackage[shortlabels]{enumitem}
\setlist{nosep}

\usepackage{etoolbox}
\patchcmd{\thebibliography}{\section*{\refname}}{\section{References}}{}{}

\usepackage{titling}
\usepackage{fancyhdr}
\usepackage{hyperref}

\title{Analysis and X-ray tomography}
\author{%
Joonas Ilmavirta\\%
\scriptsize\texttt{joonas.ilmavirta@jyu.fi},
\scriptsize\texttt{joonas.ilmavirta.research@gmail.com}
}
\date{\today}

\theoremstyle{plain}
\newtheorem{theorem}{Theorem}[section]
\newtheorem{lemma}[theorem]{Lemma}

\newtheorem{proposition}[theorem]{Proposition}

\theoremstyle{definition}
\newtheorem{definition}[theorem]{Definition}

\theoremstyle{remark}

\newtheorem{exx}{Exercise}[section]
\newtheorem{bexx}{Bonus exercise}

\newcommand\Xqed[1]{\leavevmode\unskip\penalty9999 \hbox{}\nobreak\hfill\quad\hbox{#1}}
\newenvironment{ex}{\exx}{\Xqed{$\bigcirc$}\endexx}
\newenvironment{bex}{\bexx}{\Xqed{$\bigcirc$}\endbexx}
\newcommand{\qa}{\begin{ex}Do you have any questions or comments regarding section~\thesection? Was something confusing or unclear? Were there mistakes? What new figure would help?\end{ex}}

\newcommand{\T}{\mathbb{T}}
\newcommand{\R}{\mathbb{R}}
\newcommand{\Z}{\mathbb{Z}}
\newcommand{\C}{\mathbb{C}}
\newcommand{\N}{\mathbb{N}}
\newcommand{\Sphere}{\mathbb{S}}
\newcommand{\p}{\mathcal{P}}
\newcommand{\A}{\mathcal{A}}
\newcommand{\B}{\mathcal{B}}
\newcommand{\ip}[2]{\left\langle#1,#2\right\rangle}
\newcommand{\der}{\mathrm{d}}
\newcommand{\dd}{\,\der}
\newcommand{\eps}{\varepsilon}
\renewcommand{\phi}{\varphi}
\newcommand{\abs}[1]{\left\lvert #1 \right\rvert}
\newcommand{\aabs}[1]{\left\| #1 \right\|}
\DeclareMathOperator{\spt}{spt}
\newcommand{\dummy}{{\,\cdot\,}}
\newcommand{\xrt}{\mathcal{I}}
\newcommand{\ft}{\mathcal{F}}
\DeclareMathOperator{\id}{id}
\newcommand{\pud}{{\bar D^*}}
\newcommand{\rot}{\mathcal{R}}
\newcommand{\h}{\mathcal{H}}

\newcommand{\inwb}{{\partial_{\mathrm{in}}(S\Omega)}}

\newcommand{\puna}[1]{\textcolor{red}{#1}}
\newcommand{\sini}[1]{\textcolor{blue}{#1}}
\newcommand{\viher}[1]{\textcolor{Green}{#1}}

\newcommand{\todo}[1]{}

\begin{document}

\maketitle

\noindent
These are lecture notes for the course ``Analysis and X-ray tomography''.
The course is a broad overview of various tools in analysis that can be used to study X-ray tomography.
The focus is on tools and ideas, not so much on technical details and minimal assumptions.
Only very basic functional analysis is assumed as background.
Exercise problems are included.


{
\renewcommand{\baselinestretch}{-1.5}
\small
\tableofcontents
}

\clearpage

\setcounter{section}{-1}

\section{Foreword}

These lecture notes are intended to be a point of entry to X-ray tomography and inverse problems more generally for anyone with a background in analysis.
Most treatises on the mathematics of X-ray tomography pick a single tool and dig deep with it, but here the goal is the opposite: these notes instruct you to pick up several different tools and use them towards the same goal.
I hope you, the reader, have an answer of some depth to two questions upon having read these notes:
What is the X-ray transform?
How can the various tools of analysis be used together to tackle a practical problem?

Due to the diversity of the material I expect you have varying amounts of background knowledge on the tools we employ.
Some tools only appear in passing, some are often mentioned but never used in a serious way, and some are in heavy use.
What is relevant, interesting, or in any way worth one's while depends on the reader, and I expect every reader to find something boring or incomprehensible on these pages.
What I hope you to find in addition are connections;
I aim to guide you to think of the same problem in numerous different ways and see connections.

As a necessary consequence of breadth is shallowness.
If you have prior familiarity with some of the topics we touch, you can supply much context and detail that I have left unwritten.
I hope the way less familiar topics are treated can function as an invitation to study them further or refresh your memory of prior studies.
The breadth gives you room to pick a focus and ignore what feels too alien.
I will not refrain from taking detours outside the main scope, and it is up to you to identify and ignore the ones that are not for you.
Some of the discussions are left somewhat imprecise; while developing a more complete theory of many of the tools we use would be interesting, it would distract from the main objective of X-ray tomography.
These notes hopefully serve as an invitation to dig deeper into the many related subfields of mathematics.

The X-ray transform has numerous aspects one could study.
These notes focus on injectivity or uniqueness questions only.
Matters of stability, larger function spaces, regularization, and essentially all numerical and practical aspects are left out.
The aim is to keep focus on a single problem --- injecitivity of the X-ray transform --- and spread the focus across numerous different ways to solve this one problem.

The two prerequisites to gain an access to this material are introductory functional analysis and some level of mathematical maturity.
Earlier study of inverse problems, measure theory, differential geometry, distribution theory, Fourier analysis, or partial differential equations will most certainly help but is not required.

I have given courses based on this material at
University of Jyv\"askyl\"a (several times, starting Fall 2017),
Rice University (Spring 2019),
and Tampere University (Spring 2021).
I would be happy to hear of any implementations of courses making use of this material.
The material is designed to fill 15 lectures of 90 minutes, supported with exercises.

Previous feedback has been very useful and new feedback is welcome.
Whether you read these notes within or outside a course I give, an email with suggestions or other feedback will be received with joy.

\section{Introduction}
\label{sec:intro}

\subsection{Direct and inverse problems}

Consider a physical system whose behavior depends on some parameters.
Here are some examples:
\begin{enumerate}
\item X-ray images depend on how the object attenuates X-rays (described by an attenuation coefficient depending on position).
\item The way in which boundary current (current flux density) depends on boundary voltage of an electrically conducting object depends on the (position-dependent) conductivity.
\item The spectrum of oscillations of a drum depends on the shape of the drum.
\end{enumerate}

\noindent
The direct problem asks to determine the behavior, given the parameters:
\begin{enumerate}
\item Given the attenuation coefficient, find the attenuation of any X-ray.
\item Given the conductivity, find how the boundary current depends on boundary voltage.
\item Given the drum shape, find the spectrum.
\end{enumerate}

\noindent
The inverse problem asks the opposite:
\begin{enumerate}
\item Given the attenuation data for all lines, find the attenuation coefficient everywhere.
See figure~\ref{fig:xrt}.
\item Given how the boundary current depends on boundary voltage, find the conductivity everywhere inside.
See figure~\ref{fig:eit}.
\item Given the spectrum, find the shape.
\end{enumerate}

\begin{figure}[t]
    \centering
    \includegraphics[scale=0.4,trim={3cm 0 3cm 0},clip,page=2]{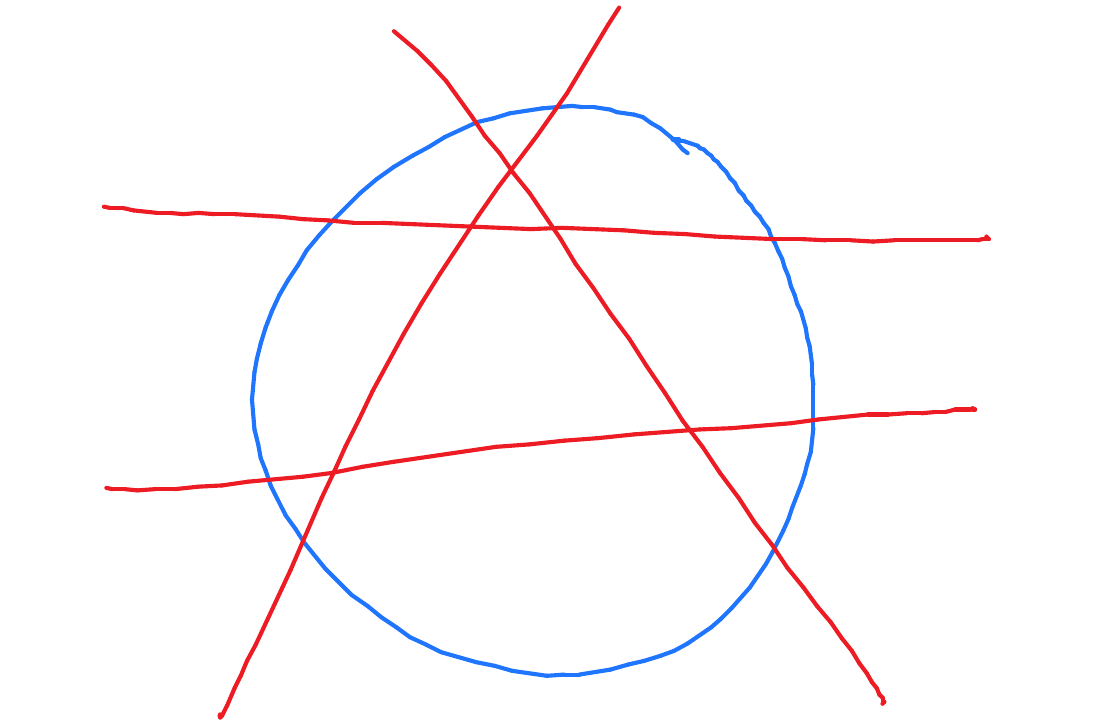}
    \caption{In X-ray tomography the aim is to reconstruct a function in \sini{a domain} from its integrals over \puna{all lines through the domain}.}
    \label{fig:xrt}
\end{figure}

\begin{figure}[t]
    \centering
    \includegraphics[scale=0.5,trim={3cm 0 3cm 0},clip,page=1]{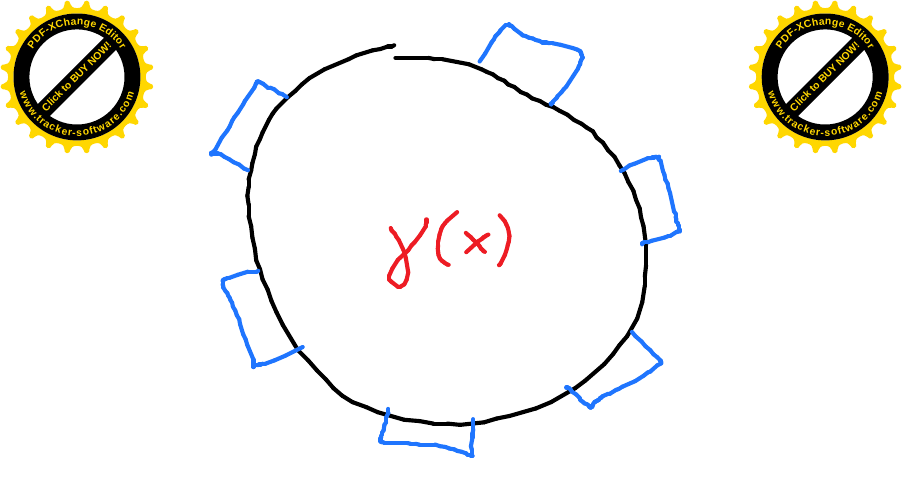}
    \caption{In electrical impedance tomography (EIT) the task is to find a \puna{conductivity~$\gamma$ that depends on position} from measurements of voltage and current at \sini{electrodes placed at the boundary}.}
    \label{fig:eit}
\end{figure}

\noindent
These inverse problems are theoretical problems in physics.
We are interested in the mathematical formulations of these problems, particularly the first one.
Solving the mathematical problem is a necessary step in solving the physical problem, but there are many more steps to take.
We will ignore numerical implementation, data acquisition, and other practical considerations, and focus on the underlying mathematics.

It is not at all unusual that a physical problem becomes a mathematical problem after some analysis.
This is done in a number of courses in physics, and physicists are well acquainted with solving mathematical problems arising from physics.
The issue with these three inverse problems is that the underlying mathematical problems are hard:

\begin{enumerate}
\item Given the integral of a continuous (or other) function $\R^n\to\R$ over each line, reconstruct the function.
\item Let $\Omega\subset\R^n$ be a nice domain and $\gamma\colon\Omega\to(0,\infty)$ with $\log(\gamma)\in L^\infty$. Given $\{(u|_{\partial\Omega},\gamma\nu\cdot\nabla u|_{\partial\Omega})
;u\in H^1(\Omega),\nabla\cdot(\gamma\nabla u)=0\}$, find~$\gamma$.
\item Given the Dirichlet spectrum of the Laplace operator on a domain $\Omega\subset\R^n$, find the domain.
\end{enumerate}

\noindent
Here and henceforth ``nice'' is not a precise term, but used when we want to avoid a precise technical definition for the sake of clarity.

The first one is the simplest, also because it is linear.
The second one is harder and there are some big open problems related to it, but it is still relatively well understood.
Our understanding of the third problem is very limited.

To give specific examples, the first problem has been solved for compactly supported distributions, the second one for $n=2$ (and $n\geq3$ if~$\gamma$ is Lipschitz), and the third one very partially (there are some counterexamples and rigidity results and very few full uniqueness results).

\begin{ex}
Let us then see how the mathematical and physical versions of the first problem are related.
Feel free to make any regularity assumptions on the function~$f$.
(You have this freedom throughout the course; if your idea needs more derivatives than you are given, go ahead and assume more.)

Consider a ray of light traveling on the real axis in the positive direction.
Let the intensity at $x\in\R$ be~$I(x)$.
If the attenuation function is $f\colon\R\to\R$ (a sufficiently regular positive function), then~$I$ satisfies the Beer--Lambert law
\begin{equation}
I'(x)=-f(x)I(x).
\end{equation}
Solve this differential equation.
Show that if $I(0)\neq0$ (if the intensity was zero, there would not be any real measurement), then the knowledge of~$I(0)$ and~$I(L)$ determines $\int_0^Lf(x)\dd x$.
\end{ex}

In physics, the attenuation coefficient is often denoted by~$\mu$.
Since it is the most important function on this course, it will be most convenient to follow the mathematical convention and call it~$f$.

\begin{ex}
Consider a bounded domain (some physical object) $\Omega\subset\R^3$.
Suppose the attenuation is described by a continuous function $f\colon\R^3\to[0,\infty)$ with $f=0$ in $\R^3\setminus\Omega$.
Consider a line segment $\gamma\colon[0,L]\to\R^3$, $\gamma(t)=x_0+tv$, and suppose that~$\gamma(0)$ and~$\gamma(L)$ are both outside~$\Omega$.
Suppose that we fire an X-ray beam along~$\gamma$ and measure the initial and final intensity.
Argue that such a measurement determines the integral of~$f$ over~$\gamma$.
\end{ex}

Attenuation is caused by two kinds of events:
photons are absorbed and scattered by the medium.
The attenuation coefficient is the sum of an absorption coefficient and a scattering coefficient.
We will not go deeper into the physics in this course.

\subsection{Goals}

All the mathematical inverse problems above are of the following form:
Consider a function $F\colon X\to Y$.
Given~$F(x)$, find~$x$.

The direct problem is finding the function~$F$ (and proving it is well defined), and this function is called the forward operator.
The function~$F$ can be complicated.
Let us see what the sets~$X$ and~$Y$ are in the three examples above:

\begin{enumerate}
\item $\{\text{continuous compactly supported functions supported in }\bar\Omega\subset\R^n\}\\{}\qquad\to\{\text{real-valued functions on the set of lines}\}$
\item $\{\gamma\colon\Omega\to(0,\infty);\log(\gamma)\in L^\infty\}\\{}\qquad\to
\p(H^{1/2}(\partial\Omega)\times H^{-1/2}(\partial\Omega))
$
\item $\{\text{smooth bounded domains in }\R^n\}\\{}\qquad\to\{\text{multisets of positive real numbers}\}$
\end{enumerate}

\noindent
Once one understands the forward operator, one can start studying the corresponding inverse problem.
So, what does it exactly mean to find~$x$?
There are several kinds of goals:

\begin{itemize}
\item Uniqueness: Show that if $F(x)=F(x')$, then $x=x'$.
\item Reconstruction: Give a formula or other method to reconstruct~$x$, given~$F(x)$. That is, find a left inverse function $G\colon Y\to X$ so that $G\circ F=\id_X$.
\item Stability: Show that if $F(x)\approx F(x')$, then $x\approx x'$. Equip the spaces~$X$ and~$Y$ with suitable norms or topologies, and prove that the left inverse~$G$ is continuous.
\end{itemize}

\noindent
A left inverse is what we use to process the data.
We have measured~$F(x)$, and we compute~$G(F(x))$ to find~$x$.
There is no need for a two-sided inverse, and there can be several ways~$G$ to analyze the data.

Ideally, we want a stable reconstruction, so that the left inverse~$G$ is continuous.
This has nothing to do with continuity of the unknown function; the operators~$F$ and~$G$ can map between any kinds of function spaces.

The study of any mathematical inverse problem starts with uniqueness, and that is what we shall focus on in this course.
That is, our sole goal is to prove that a certain function~$F$ is injective.
Some uniqueness proofs immediately give a formula for~$G$.

One important aspect we will ignore is range characterizations.
This is about finding what kind of data can really arise from real measurements --- finding the set $F(X)\subset Y$.

\subsection{The X-ray transform}

The forward operator in the X-ray tomography problem is known as the X-ray transform.
There are several different notations out there.
We will denote it by~$\xrt$.
It maps functions on~$\R^n$ into functions on the set of all lines.

In general, it is defined so that if~$f$ is a function in~$\R^n$ and~$\gamma$ is a line in~$\R^n$, then~$\xrt f(\gamma)$ is the integral of~$f$ over~$\gamma$.
This definition can be extended to various classes of functions, or even distributions.
To emphasize, let us give this as s definition:

\begin{definition}
\label{def:xrt}
Let $f\colon\R^n\to\R$ (or~$\C$) be a sufficiently regular function.
Denote by~$\Gamma$ the set of all straight lines in~$\R^n$.
The X-ray transform of~$f$ is the function $\xrt f\colon\Gamma\to\R$ (or~$\C$) defined by letting~$\xrt f(\gamma)$ be the integral of~$f$ over~$\gamma$.
\end{definition}

Let us see one example of a definition precisely.

\begin{ex}
Let $B\subset\R^n$ be the unit ball, and let~$C_B$ denote the space of continuous functions $f\colon\R^n\to\R$ with $f(x)=0$ for $x\notin\bar B$.
Show that if the space~$C_B$ is equipped with the norm $\aabs{f}=\sup_B\abs{f}$, then it is a Banach space.
\end{ex}

\begin{ex}
\label{ex:xrt-def}
Let us parametrize all lines in~$\R^n$ with $x\in\R^n$ and $v\in \Sphere^{n-1}$.
Explain why $\xrt\colon C_B\to C_b(\R^n\times \Sphere^{n-1})$ given by
\begin{equation}
\xrt f(x,v)
=
\int_{-\infty}^\infty f(x+tv)\dd t
\end{equation}
is well defined (the defining integral exists and~$\xrt f$ is continuous), linear, and continuous (the operator has bounded norm).
Here $C_b(\R^n\times \Sphere^{n-1})$, the space of continuous and bounded functions $\R^n\times \Sphere^{n-1}\to\R$, is also equipped with the supremum norm.
(It turns out that this~$\xrt$ is injective but it does not have a continuous left inverse.)
\end{ex}

We will not pursue optimal regularity in this course.
Our interest will be in ideas and tools, not proving theorems with sharp assumptions.
A reader with suitable experience in analysis is invited to consider lower regularity versions of the results presented here.

In particular, we want to show that the X-ray transform defined in exercise~\ref{ex:xrt-def} is injective.
If we need to make additional assumptions like differentiability, we will.
Since the operator is linear, we need to show that $\xrt f=0$ implies $f=0$.
We will prove this result in a number of different ways and review the necessary tools.
This is the whole plan for this course.

\begin{ex}
Physically, there is a constraint on the attenuation function~$f$.
Namely, the attenuation must be non-negative: $f\geq0$.
Recall the Beer--Lambert law and explain why this is physically reasonable.
\end{ex}

\begin{ex}
Prove that if $f\in C_B$, $f\geq0$, and $\xrt f=0$ (the integral is zero over all lines), then $f=0$.
This is far easier to prove than injectivity of the X-ray transform~$\xrt$.
Does the desired uniqueness result for non-negative attenuation functions follow from this observation?
\end{ex}

Non-continuous attenuation functions are physically relevant.
We restrict our attention to continuous functions for technical convenience.
In some exercises we will consider non-continuous functions, but they will be integrable over each geodesic.

In one dimension the problem is hopeless.
Therefore we make the standing assumption that the dimension~$n$ is at least~$2$ unless otherwise mentioned.

\begin{ex}
\label{ex:1.7}
Show that the X-ray transform $\xrt\colon C_B\to C_b$ as defined in exercise~\ref{ex:xrt-def} is not injective if $n=1$.
\end{ex}

\subsection{The Radon transform and parametrizations}
\label{sec:radon-transform}

In the X-ray transform a function is integrated over all lines.
In the Radon transform a function is integrated over all hyperplanes.
In the plane these two transforms coincide, but in higher dimensions they do not.

Let us give a more detailed description of the Radon transform.
Let~$H$ be the set of all hyperplanes in~$\R^n$.
Then the Radon transform of a, say, compactly supported continuous function $f\colon\R^n\to\R$ is a function $Rf\colon H\to\R$ given by $Rf(h)=\int_hf\dd\h^{n-1}$.
The integral over the hyperplane~$h$ is of course taken with respect to the Hausdorff measure of dimension $n-1$.
This is the same thing as identifying the hyperplane (isometrically) with~$\R^{n-1}$ and using the usual Lebesgue measure.
Or you can use any sensible way to integrate over a hyperplane; they will all give the same result when the function is continuous and compactly supported.

\begin{ex}
\label{ex:radon-xrt}
Explain how one can calculate the Radon transform of a function, given its X-ray transform.
Then explain how injectivity of the Radon transform implies injectivity of the X-ray transform.
\end{ex}

Whichever transform we study, we need to describe the lines or hyperplanes somehow.
There are various options (see figure~\ref{fig:line-parametrizations}):

\begin{itemize}
\item Consider the abstract set of all lines in~$\R^n$.
\item Parametrize a line with a point $x\in\R^n$ and a direction $v\in \Sphere^{n-1}$. The line is $x+v\R$.
\item Parametrize a line in~$\R^2$ with the closest point to the origin. (This only fails to parametrize the lines through origin.)
\end{itemize}

\begin{figure}[t]
    \centering
    \includegraphics[scale=0.4,trim={3cm 0 3cm 0},clip,page=3]{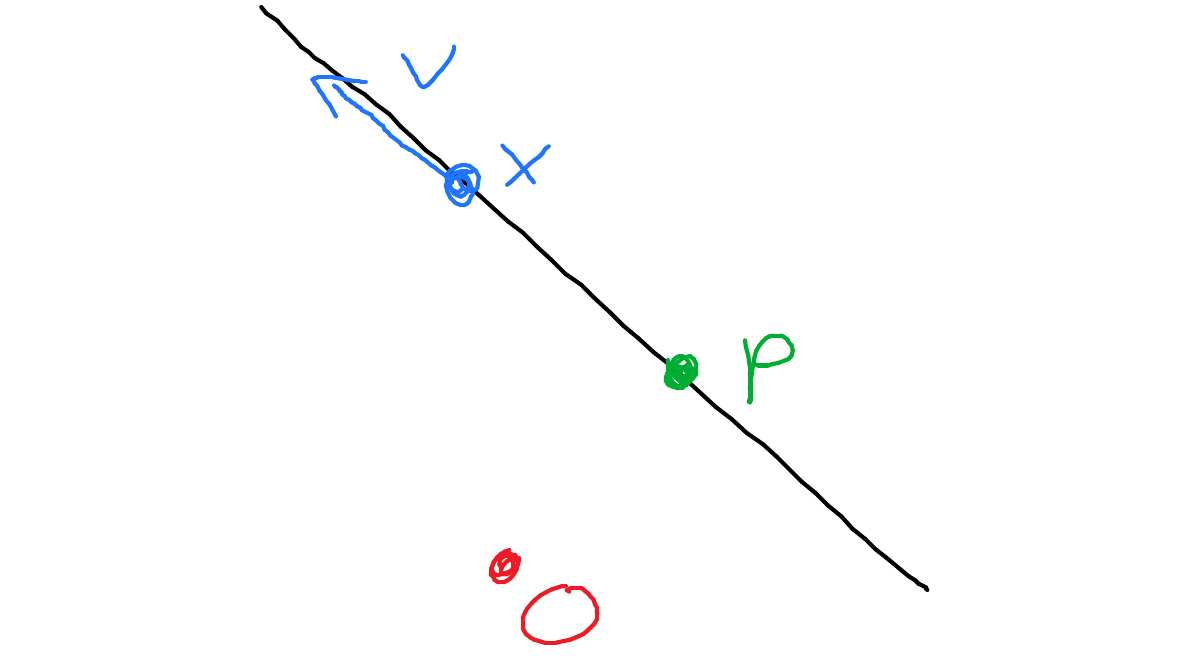}
    \caption{Two ways to parametrize a line in the plane: giving \sini{a point and a direction} or giving \viher{the closest point} to \puna{the origin}.}
    \label{fig:line-parametrizations}
\end{figure}

\noindent
This is not all.
One can also use fan beam coordinates, parallel beam coordinates, or identify a line with a direction and the boundary point of entrance.
This last parametrization will occur as~$\inwb$ in later sections.
\todo{Näistäkin kuvat?}

When the parametrization of lines is redundant, the X-ray transform should have the same value with different parameters representing the same line.
This is a simple example of a (partial) range characterization.
We will not try to characterize $\xrt(C_B)\subset C_b(\R^n\times \Sphere^{n-1})$, for example.

Hyperplanes in~$\R^n$ can also be parametrized by the closest point to the origin (with difficulties at the origin), just like one can do with lines in the plane.
This is common in the analysis of the Radon transform.

\begin{ex}
\label{ex:1.9}
Consider the characteristic function of a ball centered at the origin with unit radius.
Three ways to describe lines were listed above.
Using the first one (abstract lines), the X-ray transform of the characteristic function evaluated at a line~$\gamma$ is simply the length of the segment of~$\gamma$ that meets the ball.

Find the X-ray transform using the other two ways listed above to describe the lines.
Now you should get explicit formulas instead of abstract descriptions.
\end{ex}

\begin{ex}
Let us use the second parametrization given above, characterizing lines as $x+v\R$.
The X-ray transform of a certain function $f\colon\R^n\to\R$ is
\begin{equation}
\xrt f(x,v)
=
\begin{cases}
\sqrt{2+(v\cdot x)^2-\abs{x}^2} & \text{when }2+(v\cdot x)^2-\abs{x}^2\geq0
\\
0 & \text{otherwise}.
\end{cases}
\end{equation}
What is the function~$f$?
(Compare to exercise~\ref{ex:1.9}.)
\end{ex}

\begin{ex}
\label{ex:2d-hd}
The most typical X-ray imaging method is computerized tomography (CT), where a three-dimensional image (of the attenuation function) is reconstructed slice by slice.

If we can show that the X-ray transform is injective in two dimensions, then it follows that it will also be injective in higher dimensions.
Explain why this is so.
\end{ex}

\qa

\section{The Fourier series}
\label{sec:fs}

\subsection{Introduction}

Consider the function series
\begin{equation}
\label{eq:fs-real}
f(x)
=
b_0
+
\sum_{k=1}^\infty (b_k\cos(kx)+c_k\sin(kx)).
\end{equation}
Whether or not the series converges and in which sense depends on the sequences of coefficients~$(b_k)_{k=0}^\infty$ and~$(c_k)_{k=1}^\infty$.
It is quite obvious that if the series defines a reasonable function, then it will be periodic with period~$2\pi$.

The surprise is that every~$2\pi$-periodic function can be written as a series like this, and that the coefficient sequences are unique.
The regularity of the function and the mode of convergence depends on how fast (if at all) $b_k,c_k\to0$ as $k\to\infty$.

Having two coefficient series like above is quite awkward for a number of reasons.
It is far more convenient to study the series
\begin{equation}
\label{eq:fs-complex}
f(x)
=
\sum_{k\in\Z}a_ke^{ikx}
\end{equation}
with complex coefficients~$a_k$.
Even if the function~$f$ is real-valued, complex coefficients are needed, so the whole theory is best built over~$\C$.
The definition of the X-ray transform can be easily extended from real functions to complex ones; see definition~\ref{def:xrt}.

\begin{ex}
Compare the series in~\eqref{eq:fs-real} and~\eqref{eq:fs-complex}.
Write either $e^{it}=\cos(t)+i\sin(t)$ or $\cos(t)=\frac12(e^{it}+e^{-it})$ and $\sin(t)=\frac1{2i}(e^{it}-e^{-it})$, and compare the two representations term by term.
(No need to justify yet why it is enough to compare the terms.)
Express each coefficient~$a_k$ in terms of~$b_k$ and~$c_k$, and vice versa.
\end{ex}

\begin{ex}
\label{ex:1d-periodic-quotient}
Define the equivalence relation~$\sim$ on~$\R$ by declaring $x\sim y$ whenever $\frac1{2\pi}(x-y)\in\Z$.
Explain briefly why this is an equivalence relation.
(An equivalence relation has three properties: symmetry ($x\sim y\iff y\sim x$), transitivity ($x\sim y\sim z\implies x\sim z$) and reflexivity ($x\sim x$ for all~$x$.)
We define the quotient $\R/2\pi\Z$ as the set of equivalence classes.
Explain how functions $\R/2\pi\Z\to\C$ correspond uniquely to $2\pi$-periodic functions $\R\to\C$.

The equivalence class of $x\in\R$ is $[x]=\{y\in\R;x\sim y\}$, and in this case it can well be written as $x+2\pi\Z=\{x+2\pi k;k\in\Z\}$ too.
The quotient set is often denoted by $\R/{\sim}$, but we use the notation $\R/2\pi\Z$ here instead to remind of the specific relation we have.
Two equivalence classes are either disjoint or equal, so the different equivalence classes partition~$\R$.
By a quotient space we will always mean the set of all different equivalence classes.
It may help with this exercise to first argue why $f\colon\R\to\C$ being $2\pi$-periodic is equivalent with it being constant in each equivalence class.
Figure~\ref{fig:1d-fundamental-domain} may be of help.
\end{ex}

\begin{figure}[t]
    \centering
    \includegraphics[scale=0.5,trim={3cm 0 3cm 0},clip,page=4]{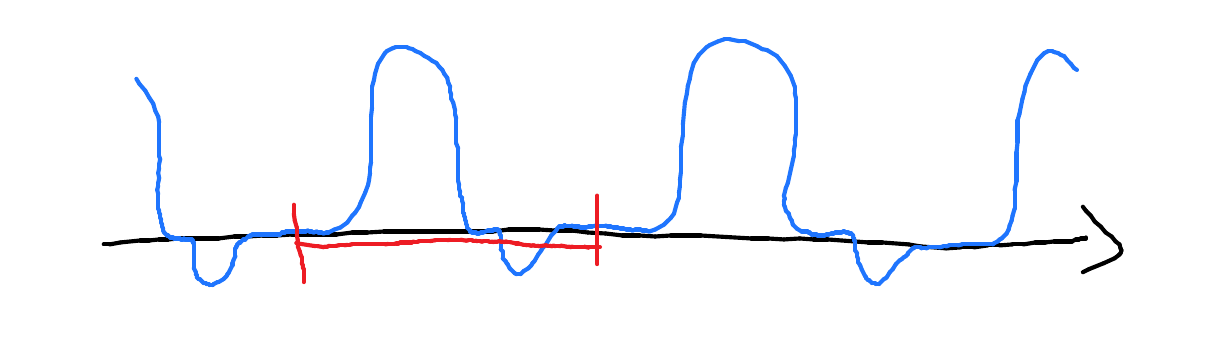}
    \caption{A \sini{periodic function} on the real line. The behavior of the function is fully captured by its values in \puna{the fundamental domain}, an interval whose length equals the period of the function. There are many options for the choice of the fundamental domain, but for $2\pi$-periodic function we may conveniently take it to be $[0,2\pi)$. The two endpoints of the fundamental domain are naturally identified with each other.}
    \label{fig:1d-fundamental-domain}
\end{figure}

\begin{figure}[t]
    \centering
    \includegraphics[scale=0.4,trim={3cm 0 3cm 0},clip,page=6]{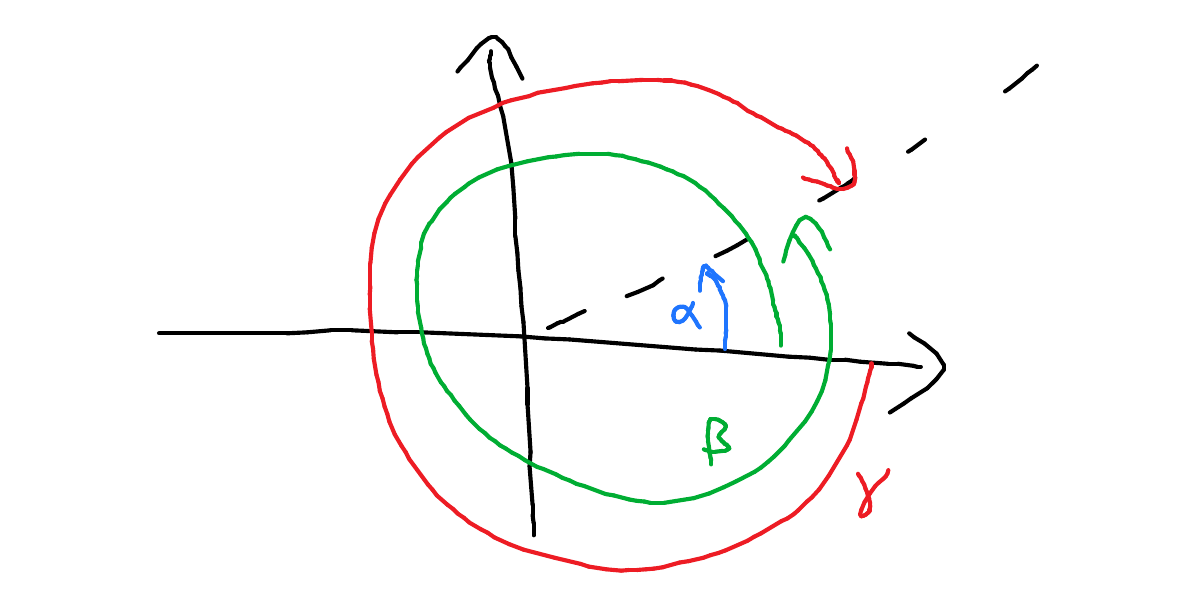}
    \caption{Different angles correspond to the same direction on the unit circle~$\Sphere^1$. Thus the numbers $\alpha,\beta,\gamma\in\R$ belong to the same equivalence class in $\R/2\pi\Z$.}
    \label{fig:angles-mod-2pi}
\end{figure}

In fact, more is true than implied by the previous exercise.
The quotient $\R/2\pi\Z$ inherits a lot of structure from~$\R$: topology, the structure of a smooth manifold, measure, various function spaces\dots

We will take much of Fourier analysis as a given fact.
More details can be found in any book or course focusing on Fourier analysis.
We will review some key results needed to successfully and understandingly apply Fourier tools to X-ray tomography.

\subsection{Fourier transform and inverse Fourier transform on a circle}

We will freely identify a $2\pi$-periodic function on $\R$, a function on $\R/2\pi\Z$, and a function on the fundamental domain $[0,2\pi)$.
See figure~\ref{fig:1d-fundamental-domain} and exercise~\ref{ex:1d-periodic-quotient}.
The identification of the circle~$\Sphere^1$ and the quotient $\R/2\pi\Z$ is illustrated in figure~\ref{fig:angles-mod-2pi}.

Consider the space $L^2(\R/2\pi\Z)$ of measurable $2\pi$-periodic functions $f\colon\R\to\C$ that satisfy
\begin{equation}
\int_0^{2\pi}\abs{f(x)}^2\dd x
<
\infty.
\end{equation}
It is a complex Hilbert space with the inner product
\begin{equation}
\ip{f}{g}
=
\int_0^{2\pi}\overline{f(x)}g(x)\dd x.
\end{equation}
We defined $L^2(\R/2\pi\Z)$ to be a space of functions $\R\to\C$, not $\R/2\pi\Z\to\C$.
However, due to periodicity we can regard the functions in this space as functions on the quotient $\R/2\pi\Z$.
We will denote this quotient later by $\T^1$.

\begin{ex}
Recall the space~$L^2(0,2\pi)$ of square integrable Lebesgue measurable functions $(0,2\pi)\to\C$.
This is a Hilbert space, and it is naturally isomorphic to $L^2(\R/2\pi\Z)$.
Give the natural isomorphisms in both directions.
Are they isometric?
(The fact that they are isomorphic follows from the fact that they are both separable infinite-dimensional complex Hilbert spaces, but there is something far simpler and more natural here.)
\end{ex}

Let us denote by~$\ell^2(\Z)$ the space of ``sequences'' (functions) $a\colon\Z\to\C$ with $\sum_{k\in\Z}\abs{a_k}^2<\infty$.
This, too, is a Hilbert space.
For convenience, we equip it with the norm
\begin{equation}
\aabs{a}^2
=
2\pi
\sum_{k\in\Z}\abs{a_k}^2
\end{equation}
and the corresponding inner product\footnote{The inner product is left implicit, and the reader is encouraged to figure out what the inner product should be. One can of course use the polar formula to find the inner product from the norm, but in a simple case like this one can see the correct inner product by eye.}.

\begin{definition}
\label{def:1d-ft}
The Fourier transform of a $2\pi$-periodic function or distribution expressed as the Fourier series~\eqref{eq:fs-complex} takes the function~$f$ into the sequence~$(a_k)_{k\in\Z}$ of Fourier coefficients.
The inverse Fourier transform takes the sequence back to the function or distribution.
In symbols, $\ft f=a$ and $\ft^{-1}a=f$.
\end{definition}

The definition above is purposely vague, just like definition~\ref{def:xrt} for the X-ray transform.
It describes the overall idea of the Fourier transform and its inverse in the present context.
The same definition can be used for a large number of different function spaces.
Observe that the Fourier transform and its inverse are linear operators.
The Fourier transform of functions on the whole line is a different animal, and we shall greet it later.

A central result in Fourier analysis is that the Fourier transform is well-defined and the inverse exists.
Even more is true:

\begin{theorem}
\label{thm:1d-fs}
The Fourier transform on $\R/2\pi\Z$ is a unitary isometry $\ft\colon L^2(\R/2\pi\Z)\to\ell^2(\Z)$, given by
\begin{equation}
(\ft f)(k)
=
\frac1{2\pi}\int_0^{2\pi}f(x)e^{-ikx}\dd x,
\end{equation}
which is well defined as a Lebesgue integral.
The inverse Fourier transform $\ft^{-1}\colon\ell^2(\Z)\to L^2(\R/2\pi\Z)$ is also unitary and isometric, and is given by
\begin{equation}
(\ft^{-1}a)(x)
=
\sum_{k\in\Z}a_ke^{ikx},
\end{equation}
where the series of functions converges in $L^2(\R/2\pi\Z)$.
\end{theorem}

This theorem will not be proven on this course.
The theorem can be rephrased as the functions $x\mapsto\frac1{\sqrt{2\pi}}e^{ikx}$, $k\in\Z$, being an orthonormal Hilbert basis for $L^2(\R/2\pi\Z)$.
In general, a Hilbert space is isometric to the~$\ell^2$ space over the index set of a Hilbert basis:
If $(v_i)_{i\in I}$ is an orthonormal basis (indexed by some set~$I$) for a Hilbert space~$H$, then $H\approx\ell^2(I)$ by mapping a vector to its coefficient sequence indexed by the set~$I$.
Now $I=\Z$ and soon we will have $I=\Z^n$.

Every~$L^2$ function $f\colon\R/2\pi\Z\to\C$ can be written uniquely as a series
\begin{equation}
f(x)
=
\sum_{k\in\Z}\ft f(k)e^{ikx}.
\end{equation}
This series is the Fourier series.
Sometimes~$\ft f$ is denoted by~$\hat f$.

The elements of a sequence are typically denoted as~$a_k$ instead of~$a(k)$, but in Fourier analysis it is customary and convenient to write~$\ft f(k)$ or~$\hat f(k)$ instead of~$\ft f_k$ or~$\hat f_k$.

\begin{ex}
Recall definitions from an earlier course or some other source.
What does it mean in formulas (involving sums and integrals) that the Fourier transform is isometric and unitary?
\end{ex}

One may wonder why the Fourier transform of a $2\pi$-periodic function $\R\to\C$ is a function on~$\Z$, not on~$\R$.
This has nothing to do with the specific problem, it is a mathematical property.
A function cannot be $2\pi$-periodic unless all frequencies are integers.
To make this statement more rigorous, one can show that the Fourier transform (in the sense of whole~$\R$, not $\T^1=\R/2\pi\Z$) of a periodic function is a distribution supported on the lattice~$2\pi\Z$.
The same is true in higher dimensions as well.
Another way to see this will come in section~\ref{sec:ft} when we discuss the Fourier transform in greater generality.
The fact that only discrete frequencies are possible is not obvious at first.
It is a key result in Fourier analysis that is seldom stated explicitly.

\subsection{Multidimensional Fourier series}

\begin{figure}[t]
    \centering
    \includegraphics[scale=0.3,trim={0cm 0 0cm 0},clip]{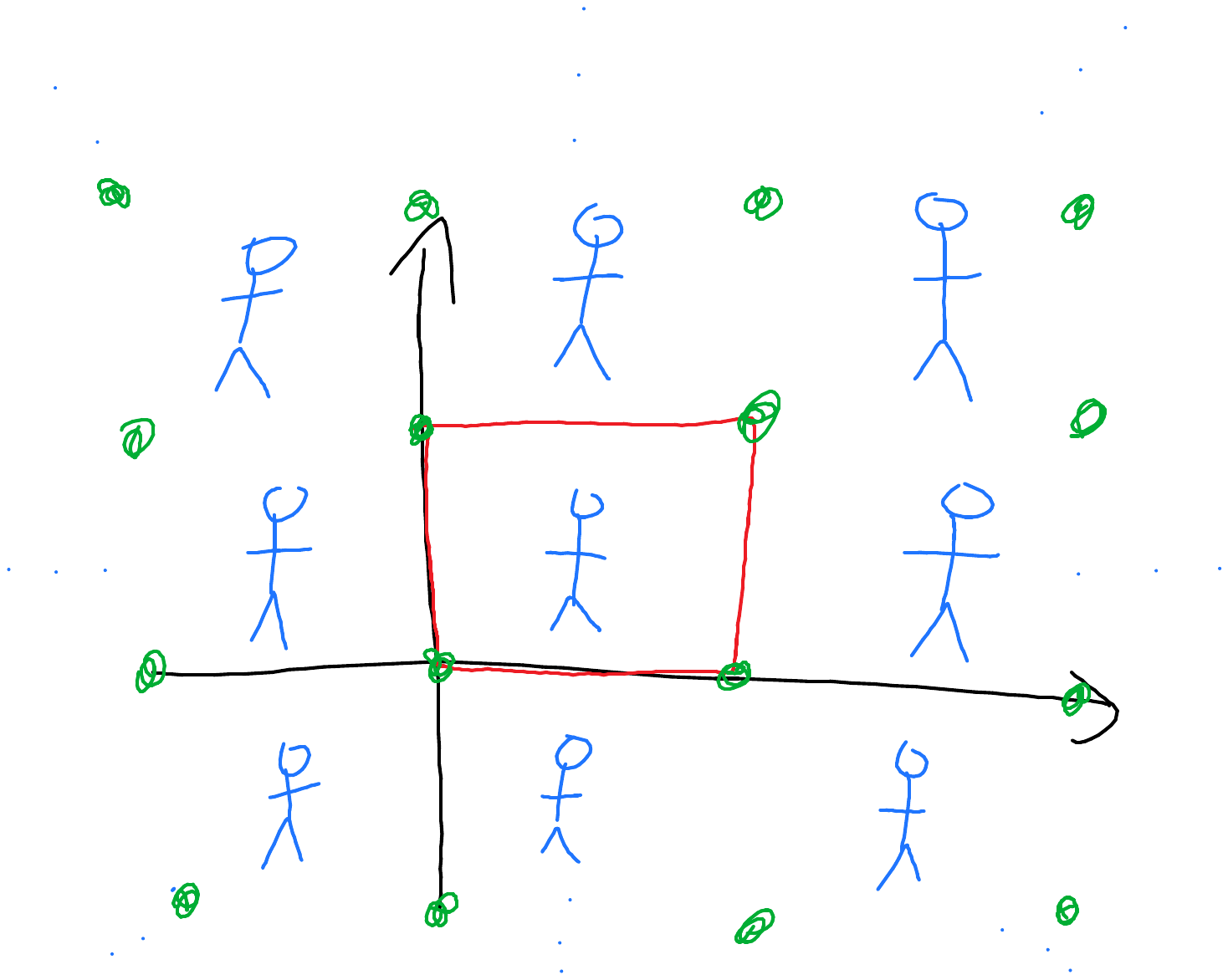}
    \caption{A \sini{periodic function} on the plane. Periodicity is described by \viher{the lattice $2\pi\Z^n$}. Now \puna{the fundamental domain} is a rectangle. The opposite sides of the rectangle can be identified with each other.}
    \label{fig:2d-fundamental-domain}
\end{figure}

In the previous section we considered Fourier series in one dimension.
The theory is very, very similar in higher dimensions, and so the two-dimensional figure~\ref{fig:2d-fundamental-domain} is very similar to its one-dimensional counterpart in figure~\ref{fig:1d-fundamental-domain}.

In higher dimensions, one studies functions $\R^n\to\C$ which are $2\pi$-periodic in all~$n$ real variables.
Notice that the space of such functions is not rotation invariant; the coordinate axes give~$n$ preferred directions.
These preferred directions are perhaps more apparent in the lattice
\begin{equation}
2\pi\Z^n
=
\{x\in\R^n;x_i/2\pi\in\Z\text{ for all }i\}.
\end{equation}
As we did above in one dimension, we may quotient the space~$\R^n$ with the lattice~$2\pi\Z^n$ to form $\R^n/2\pi\Z^n$.

\begin{ex}
What is the equivalence relation in~$\R^n$ corresponding to the lattice?
What are the equivalence classes?
\end{ex}

The quotient space (not a quotient \emph{vector} space) $\R/2\pi\Z$ is homeomorphic (in fact isometric) to the circle~$\Sphere^1$.
The quotient space $\R^n/2\pi\Z^n$ is the same as $(\R/2\pi\Z)^n$ or~$(\Sphere^1)^n$, but not~$\Sphere^n$.
We have~$n$ coordinates, each considered modulo~$2\pi$.
The topological space $\T^n\coloneqq\R^n/2\pi\Z^n$ is the $n$-dimensional torus.
The most famous torus is the two-dimensional one, and the one-dimensional torus is often called simply the circle.
Observe that the sphere $\Sphere^n$ is homeomorphic to the torus $\T^n$ if only if $n=1$.
(Proving this is unimportant for our needs, but the interested reader may consider looking into first homotopy or homology groups.)

For the differential geometrically oriented:
The lattice acts isometrically on the space~$\R^n$, so the quotients inherits the (Euclidean) Riemannian metric.
The torus with this metric is locally isometric to~$\R^n$, and is called the flat torus.

\begin{ex}
\label{ex:torus-as-quotient-group}
The Euclidean space~$\R^n$ is an additive group and~$2\pi\Z^n$ is a subgroup.
Why is the quotient group $\R^n/2\pi\Z^n$ well defined?
(That is, why is the subgroup normal?)
How does this quotient group correspond to the quotient space $\R^n/2\pi\Z^n$ described above?
Describe the group operation.
\end{ex}

A function $\R^n/2\pi\Z^n\to\C$ --- or, equivalently, a function on~$\R^n$ with period~$2\pi$ in each variable --- is written as a Fourier series as follows:
\begin{equation}
\label{eq:hd-fs}
f(x)
=
\sum_{k\in\Z^n}a_ke^{ik\cdot x}.
\end{equation}
Let us define the Fourier series similarly to what we did in definition~\ref{def:1d-ft} in one dimension:

\begin{definition}
\label{def:hd-ft}
The Fourier transform of a function or distribution on the torus~$\T^n$ expressed as the Fourier series~\eqref{eq:hd-fs} takes the function~$f$ into the sequence~$(a_k)_{k\in\Z^n}$ of Fourier coefficients.
The inverse Fourier transform takes the sequence back to the function or distribution.
In symbols, $\ft f=a$ and $\ft^{-1}a=f$.
\end{definition}

The spaces $L^2(\T^n)$ and~$\ell^2(\Z^n)$ are defined analogously to the case $n=1$ discussed above.
The norm on the latter space is
\begin{equation}
\aabs{a}^2
=
(2\pi)^n\sum_{k\in\Z^n}\abs{a_k}^2.
\end{equation}
Using these spaces, we have the following generalization of theorem~\ref{thm:1d-fs}:

\begin{theorem}
\label{thm:hd-fs}
The Fourier transform on the torus~$\T^n$ is a unitary isometry $\ft\colon L^2(\T^n)\to\ell^2(\Z^n)$, given by
\begin{equation}
(\ft f)(k)
=
\frac1{(2\pi)^n}\int_{[0,2\pi]^n}f(x)e^{-ik\cdot x}\dd x,
\end{equation}
which is well defined as a Lebesgue integral.
The inverse Fourier transform $\ft^{-1}\colon\ell^2(\Z^n)\to L^2(\T^n)$ is also unitary and isometric, and is given by
\begin{equation}
(\ft^{-1}a)(x)
=
\sum_{k\in\Z^n}a_ke^{ik\cdot x},
\end{equation}
where the series of functions converges in $L^2(\T^n)$.
\end{theorem}

\begin{ex}
When defining the Fourier transform on~$\T^1$ and~$\T^n$, we made use of the exponential functions~$e^{ik\cdot x}$.
Show that if $k\in\Z^n$ and $x\in\T^n=\R^n/2\pi\Z^n$, the value of $e^{ik\cdot x}$ does not depend on the representative of~$x$ in~$\R^n$.
This means that the exponential function~$e^{ik\cdot x}$ is indeed well defined.
Is the exponent $k\cdot x$ well defined, too?
\end{ex}

\begin{ex}
Let us denote $e_k(x)=e^{ik\cdot x}$.
For any $k\in\Z^n$ we have $e_k\in L^2(\R^n/2\pi\Z^n)$.
Using the given inner product, prove that
\begin{equation}
\ip{e_k}{e_m}
=
c\delta_{km},
\end{equation}
where~$\delta_{km}$ is the Kronecker delta and~$c$ is a constant.
What is the constant~$c$?
Do not appeal to theorem~\ref{thm:hd-fs}, but calculate by hand.
\end{ex}

The general idea is to show that $\xrt f=0\implies\ft f=0\implies f=0$.
That is, the X-ray transform of~$f$ is easier to connect to the Fourier transform~$\ft f$ than~$f$ itself.

The Fourier series can be defined on other spaces, for example non-flat tori, using the eigenfunctions of the Laplace operator.
It will be considerably more clumsy and it will not work so nicely together with the X-ray transform.
Fourier analysis tends to be most convenient when one has enough symmetry.

\begin{ex}
One aspect of the basis functions used in the Fourier series is that they are eigenfunctions of the Laplace operator.
What is the eigenvalue of the function $x\mapsto e^{ik\cdot x}$?
\end{ex}

\begin{bex}
Show that the number~$7$ cannot be written as the sum of three integer squares.
Recall Lagrange's four-square theorem.

Using these tools, show that the spectrum (the set of eigenvalues) of the Laplace operator $\Delta=\sum_{k=1}^n\partial_k^2$ on the torus~$\T^n$ is $-\N$ if and only if $n\geq4$.
\end{bex}

\qa

\section{X-ray tomography on a torus}
\label{sec:torus}

In this section we will give our first injectivity proof based on Fourier series on~$\T^n$.
Two ways to see the torus~$\T^2$ are given in figures~\ref{fig:2d-torus-donut} and~\ref{fig:2d-torus-square}.

\begin{figure}[t]
    \centering
    \includegraphics[scale=0.4,trim={3cm 0 3cm 0},clip,page=7]{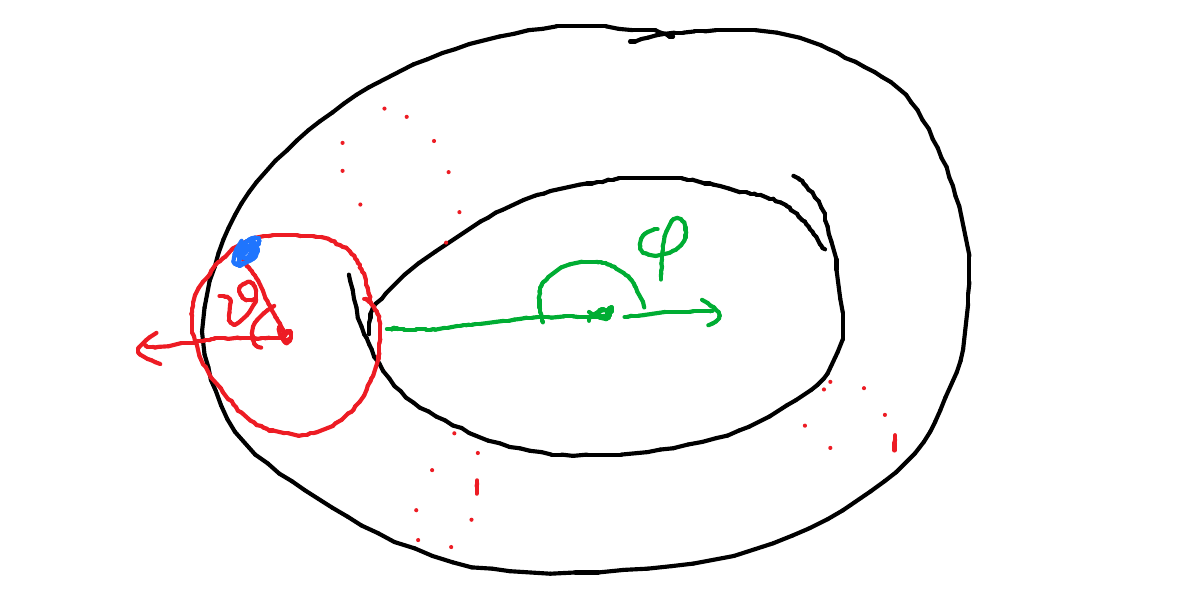}
    \caption{The two-dimensional torus $\T^2=\Sphere^1\times\Sphere^1$ can be embedded in~$\R^3$ as the surface of a donut. \sini{A point} on the torus is described by two angles (from each copy of~$\Sphere^1$), \viher{one choosing the circle} and \puna{one choosing the point on the circle}. This embedding is not isometric, so geodesics as we know them do not look straight in this picture.}
    \label{fig:2d-torus-donut}
\end{figure}

\begin{figure}[t]
    \centering
    \includegraphics[scale=0.4,trim={3cm 0 3cm 0},clip,page=8]{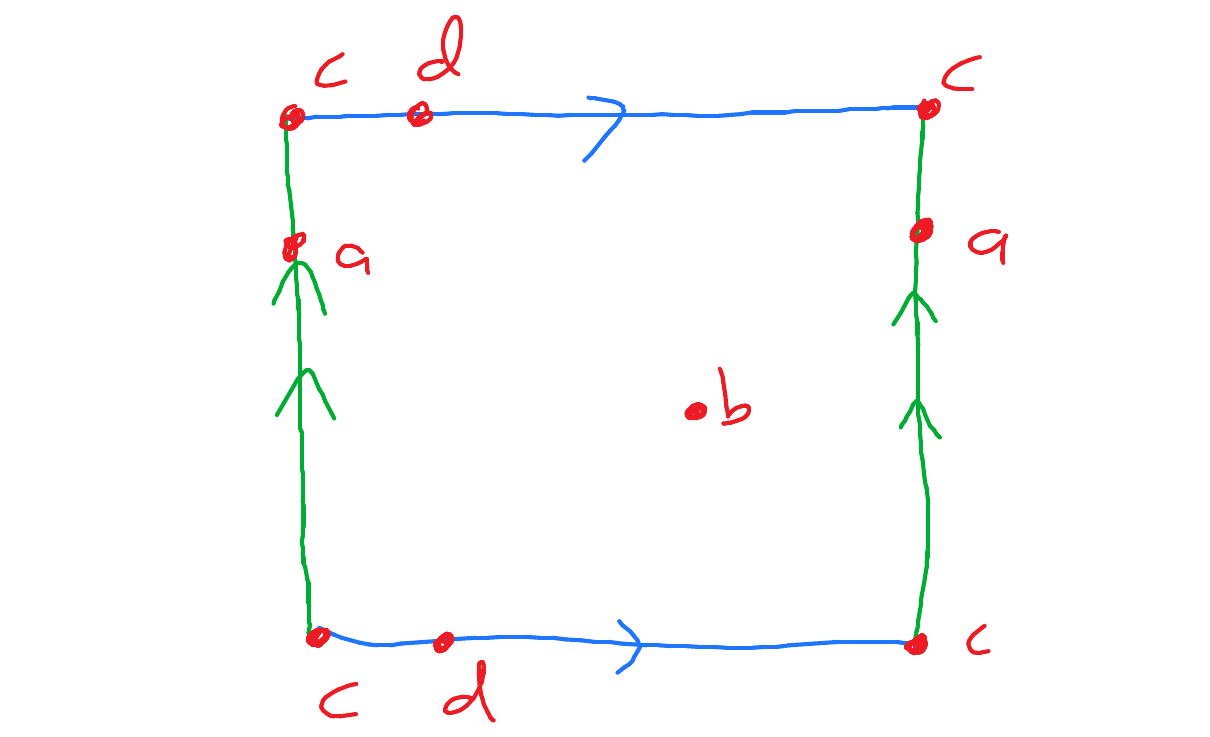}
    \caption{The two-dimensional torus~$\T^2$ can be understood as the fundamental domain of figure~\ref{fig:2d-fundamental-domain}. We identify \sini{the top edge with the bottom edge} and \viher{the left edge with the right edge}, preserving orientation (indicated by arrows). Some of \puna{the points} have multiple representations due to the identifications. If you imagine that the identified edges are physically glued together (and the paper is suitably stretched), it should be clear how this square picture corresponds to figure~\ref{fig:2d-torus-donut}. Geodesics look far cleaner in this square picture.}
    \label{fig:2d-torus-square}
\end{figure}

\subsection{Geodesics on a torus}

The analogue of a straight line in differential geometry is a geodesic.
Similarly to the problem we set out to study, one can ask whether a function on a manifold is determined by its integrals over geodesics.
This is an active field of study, but beyond the scope of this course.

However, we will study this problem now on the flat torus $\T^n=\R^n/2\pi\Z^n$.
The reason is that this provides one of the simplest\footnote{It is simple in terms of weight of technical tools if one starts from scratch, but not necessarily in terms of conceptual simplicity. The latter is up to the reader's taste.} proofs of the injectivity of the X-ray transform in a bounded Euclidean domain.

Geodesics, like any curves, can be regarded as subsets of the space or as functions from an interval to the space.
As a set, a geodesic in~$\R^n$ is simply a straight line.
As a function, a geodesic can be described as $\gamma\colon[0,1]\to\R^n$, $\gamma(t)=x+tv$.
The velocity $v\in\R^n$ can be any non-zero vector; it will be convenient not to assume unit speed in this section.
A geodesic is a curve with constant velocity, and this makes sense on far more general manifolds than~$\R^n$ and~$\T^n$.

Here we chose to parametrize the geodesic by $[0,1]$, and we have therefore described a geodesic between two points ($x$ and $x+v$).
Another option is to replace the interval with~$\R$; this leads to what is called a maximal geodesic.
In our Euclidean X-ray tomography problem we consider the integrals of an unknown function over all geodesics through a given domain.
It is irrelevant whether the geodesics are maximal or between two points, as long as the two points are outside (or at the boundary of) the domain.

Let us then turn to geodesics on a torus.
Let $q\colon\R^n\to\T^n$ be the quotient map that takes a point to its equivalence class.
One can write it as $q(x)=x+2\pi\Z^n\subset\R^n$.
This formula is seldom very useful in practice, but perhaps it helps get a hold of the idea.

A geodesic between two points on the torus~$\T^n$ is simple to describe: we may compose a geodesic on~$\R^n$ with the quotient map.
We take $x\in\R^n$ and $v\in\R^n\setminus\{0\}$ and define $\gamma\colon[0,1]\to\T^n$ by $\gamma(t)=q(x+tv)$.
Geodesics on~$\T^n$ are projections of geodesics on~$\R^n$ through~$q$, and this is true for any Riemannian covering map too.
We will be interested in maximal geodesics that do not terminate in either direction, which corresponds to replacing $[0,1]$ above by~$\R$.


On a torus, there is an interesting new class of maximal geodesics: closed geodesics, also known as periodic geodesics.
The simplest example of a periodic geodesic is
\begin{equation}
\begin{split}
&\R\to\T^n,
\\
&t\mapsto (2\pi t,0,\dots,0),
\end{split}
\end{equation}
which has period~$1$.
Geodesics $\R\to\T^n$ with period~$1$ can be naturally identified with geodesics $\gamma\colon[0,1]\to\T^n$ for which $\gamma(0)=\gamma(1)$.
The structure of the torus ensures that if the endpoints agree, then the directions are also the same.

\begin{figure}[t]
    \centering
    \includegraphics[scale=0.4,trim={3cm 0 3cm 0},clip,page=9]{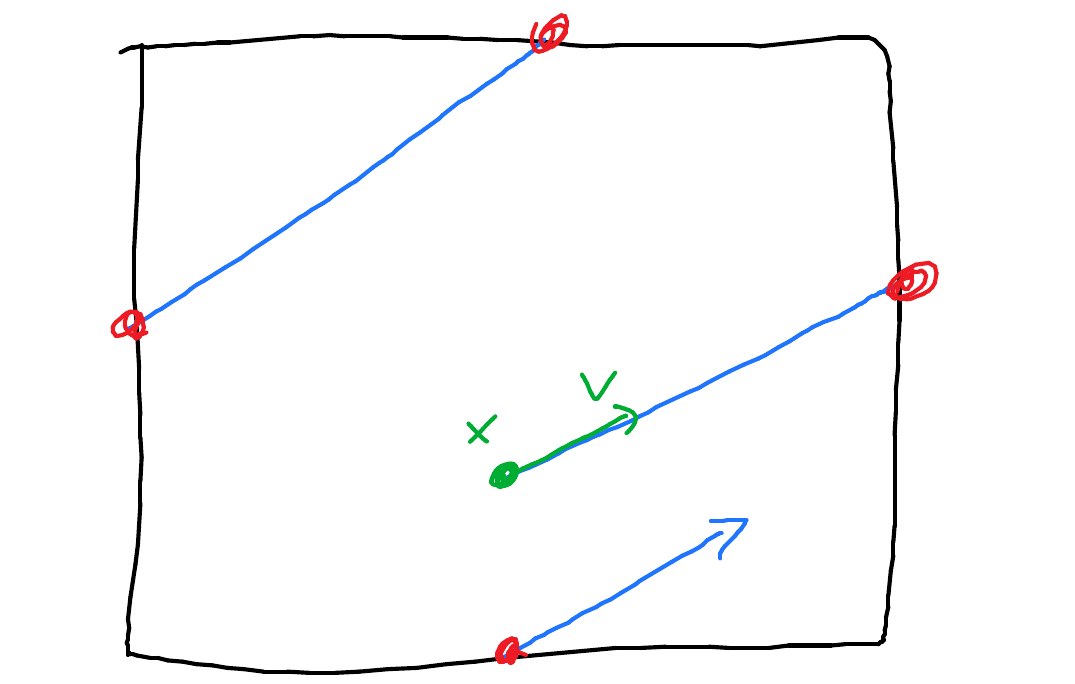}
    \caption{Geodesics on a torus are best understood through figure~\ref{fig:2d-torus-square}. Given an \viher{initial point and direction $(x,v)$}, we follow \sini{the geodesic~$\gamma(t)$ of~\eqref{eq:vv1}} as if it were a straight line in the plane, taking into account \puna{identifications (realized as jumps)} on the boundary. It depends on the initial velocity~$v$ whether the geodesic meets itself again. Note how a finitely long geodesic segment on~$\T^n$ is --- upon identifying the torus with the fundamental domain --- a finite union of line segments in~$\R^n$. This will be useful for lemma~\ref{lma:torus-to-Rn}.}
    \label{fig:2d-torus-geodesic}
\end{figure}

The most convenient way to describe geodesics for our purposes is to take two parameters $x\in\T^n$ and $v\in\R^n$, and let the corresponding geodesic $[0,1]\to\T^n$ be
\begin{equation}
\label{eq:vv1}
\gamma(t)
=
q(x'+tv),
\end{equation}
where $x'\in\R^n$ is any point so that $q(x')=x$.
This is illustrated in figure~\ref{fig:2d-torus-geodesic}.
Equivalently, we may take 
\begin{equation}
\label{eq:vv2}
\gamma(t)
=
x+q(tv),
\end{equation}
where ``$+$'' is the addition on~$\T^n$ --- which is naturally an abelian group.\footnote{The definition of the induced group operation on a quotient group may be familiar. One way to see it is that it equips the quotient~$\T^n$ with such an addition that the quotient map $q\colon\R^n\to\T^n$ is a group homomorphism. Recall exercise~\ref{ex:torus-as-quotient-group}.}
We shall write this geodesic simply as
\begin{equation}
\gamma(t)
=
x+tv\in\T^n,
\end{equation}
where the quotient is left implicit.
(The geometrically oriented may prefer to read $x+tv$ as $\exp_x(tv)$.)
See figure~\ref{fig:2d-torus-geodesic} for an illustration.

All geodesics on a torus are of this form.
This is because the quotient map $q\colon\R^n\to\T^n$ is a surjective local isometry (a Riemannian covering map) and isometries preserve geodesics.
It is crucial that the torus is flat, which is most apparent in figures~\ref{fig:2d-torus-square} and~\ref{fig:2d-torus-geodesic}.
If one uses another metric (such as the donut embedded in~$\R^3$ of figure~\ref{fig:2d-torus-donut}), the geodesics will be different and there will be less symmetry.

\begin{ex}
Consider the geodesic described in~\eqref{eq:vv1} and~\eqref{eq:vv2} above.
Show that the endpoints coincide if and only if $v\in2\pi\Z^n$.
\end{ex}

\begin{ex}
\label{ex:t1-geodesic}
Explain how a geodesic with velocity $v\in2\pi\Z^n$ can be regarded as a function $\R/\Z\to\T^n$.
\end{ex}

In the X-ray transform on a torus we will only integrate over periodic geodesics.
The reason for this is two-fold.
First, periodic geodesics are convenient and, as it turns out, sufficient.
Second, the integrals are ill-defined over a non-periodic geodesic.

By exercise~\ref{ex:t1-geodesic} a periodic geodesic is a function $\R/\Z\to\T^n$, and it is easy to integrate a continuous function over the compact set $\R/\Z$.
However, when there is no periodicity, one would have to integrate over all of~$\R$, and the resulting integral typically does not exist (as a finite number).

\subsection{Injectivity from a torus to a Euclidean space}

For any $v\in2\pi\Z^n\setminus\{0\}$, $x\in\T^n$ and $f\in C(\T^n)$, we write
\begin{equation}
\xrt_vf(x)
=
\int_0^1f(x+tv)\dd t.
\end{equation}
If~$v$ is fixed, this defines an operator
\begin{equation}
\xrt_v\colon C(\T^n)\to C(\T^n).
\end{equation}
For us the key property is that~$\xrt_v$ is linear, but it does indeed map continuous functions to continuous functions.
It has other properties as well:
$\aabs{\xrt_v}=1$
and
$\xrt_v^2=\xrt_v$.
It is also a continuous self-adjoint operator $L^2(\T^n)\to L^2(\T^n)$.

\begin{definition}
\label{def:xrt-torus}
We call the family of operators~$\xrt_v$ with $v\in2\pi\Z^n\setminus\{0\}$ the X-ray transform on the torus~$\T^n$.
\end{definition}

This point of view is convenient here, although it would be possible to realize the X-ray transform as a single operator as well.
In the usual view~$\xrt$ is not a self-adjoint operator and that leads us to consider its normal operator later in this course.

The torus has no boundary, so it is what is called a closed manifold.
It is worth noticing that in a setting like this the geodesics do not go through the domain from boundary to boundary as in our original problem.
All geodesics are trapped inside the manifold, but fortunately in an orderly fashion.

Now consider a function $f\in C_B\subset C(\R^n)$.
The function~$f$ is supported in the closed unit ball~$\bar B$, so we can extend it periodically to a function~$\tilde f$ on~$\R^n$ so that $\tilde f=f$ on $(-\pi,\pi)^n$.
Observe that $\bar B\subset(-\pi,\pi)^n$.
This periodization is illustrated in figure~\ref{fig:periodization} and the reason for the support constraint in figure~\ref{fig:periodic-failure}.

\begin{figure}[t]
    \centering
    \includegraphics[scale=0.4,trim={3cm 0 3cm 0},clip,page=10]{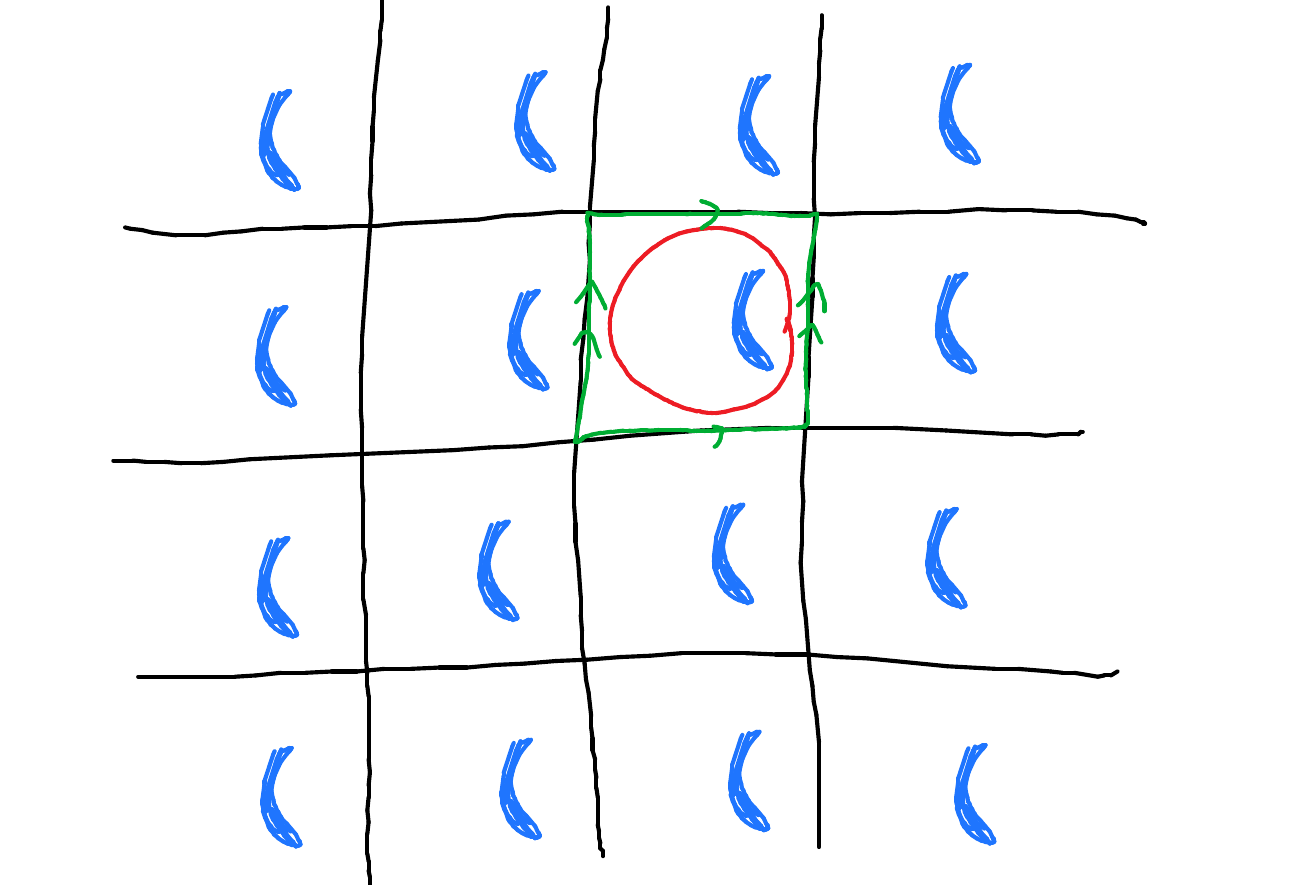}
    \caption{The original function~$f$ is supported in \puna{the ball~$B$}. It is then made periodic, giving rise to \sini{the function~$\tilde f$}. We may think of~$\tilde f$ as a function on \viher{the fundamental domain~$[-\pi,\pi)^n$}. If we think of the fundamental domain as a subset of the plane, we can naturally identify $f=\tilde f$.}
    \label{fig:periodization}
\end{figure}

\begin{figure}[t]
    \centering
    \includegraphics[scale=0.4,trim={3cm 0 3cm 0},clip,page=11]{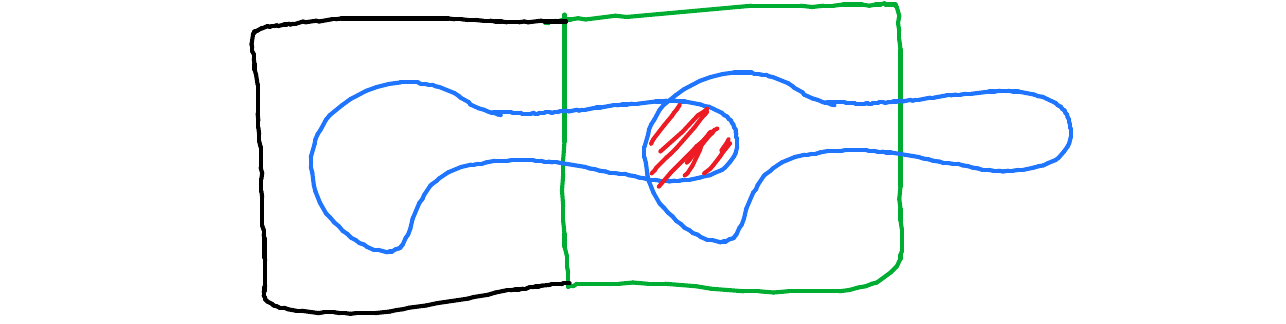}
    \caption{If \sini{the support of the function} does not fit within \viher{the fundamental domain}, there may well be \puna{regions where there are more than one candidates for the value of~$\tilde f$}. As scaling does not matter, we need that the support of~$f$ fits within a square --- meaning that it is compactly supported.}
    \label{fig:periodic-failure}
\end{figure}

\begin{ex}
Give a formula for~$\tilde f$ in terms of~$f$.
\end{ex}

Since it is periodic, the function~$\tilde f$ can be regarded as a function on the torus~$\T^n$.

\begin{lemma}
\label{lma:torus-to-Rn}
The X-ray transform of $\tilde f\in C(\T^n)$ is uniquely determined by the X-ray transform of $f\in C_B$.
\end{lemma}

The idea is simple, but writing it down is awkward.
We will do it anyway, but first in an abridged way.

\begin{proof}[Sketchy proof of lemma~\ref{lma:torus-to-Rn}]
We may think of a function on the unit ball~$\bar B$ to be a function on the torus because the ball is contained in the box-shaped fundamental domain $[-\pi,\pi)^n$.
It is obvious from figure~\ref{fig:2d-torus-geodesic} that a periodic geodesic on the torus is a union of finitely many line segments in the box.
The integrals of~$f$ over these line segments are determined by~$\xrt f$ because the support is contained in the box.
\end{proof}

\begin{proof}[Full proof of lemma~\ref{lma:torus-to-Rn}]
Take any $v\in2\pi\Z^n\setminus\{0\}$ and $x\in\T^n$.
We need to show that $\xrt_v\tilde f(x)$ can be expressed in terms of~$\xrt f$.
Recall that~$\xrt f$ is a function defined on the set of all lines in~$\R^n$.

We first observe that~$\xrt f$ determines the integral of~$f$ over any line segment with endpoints on the boundary of the fundamental domain.
This is simply because by extending the segment to a full line we never hit the support of the function again.

The restriction $q|_{[-\pi,\pi)^n}\colon[-\pi,\pi)^n\to\T^n$ is a bijection.
Let us denote its inverse by~$\iota$.
It satisfies $q\circ\iota=\id_{\T^n}$.

Since~$f$ is supported in~$\bar B$, we may consider it to be a function defined on $[-\pi,\pi)^n$, which we now take to be our fundamental domain.
Then we have $f=\tilde f\circ q$ and $\tilde f=f\circ\iota$.

Consider any $x\in\T^n$ and $v\in2\pi\Z^n\setminus\{0\}$.
Let $\tilde\gamma\colon[0,1]\to[-\pi,\pi)^n$ be the ``curve'' corresponding to the geodesic $t\mapsto\gamma(t)=x+q(tv)$, defined by $\tilde\gamma(t)=\iota(x+q(tv))$.
The image $\tilde\gamma([0,1])$ consists of line segments in~$\R^n$, so it is not strictly a curve.

The relevant sets and maps are collected in this diagram, which fully commutes:
\begin{equation}
\begin{tikzcd}[row sep=huge]
&&&
\viher{[0,1]}
\arrow[dr,"\gamma",color=Green]
\arrow[dl,"\tilde\gamma"',color=Green]
&
\\
\sini{\R^n}
\arrow[drrr,"f"',blue]
\arrow[rrrr,bend left=20,"q",blue]
&&
{[-\pi,\pi)^n}
\arrow[dr,"f"']
\arrow[rr,bend left=10,"q"]
\arrow[ll,"\text{~~~~inclusion}"',color=blue]
&&
\T^n
\arrow[dl,"\tilde f"]
\arrow[ll,bend left=10,"\iota"]
\\
&&& \C &
\end{tikzcd}
\end{equation}
The \sini{blue part} concerns the original function $f\colon\R^n\to\C$ but will be largely ignored as the function~$f$ is supported within the fundamental domain.
The \viher{green part} above spaces is for the geodesic~$\gamma$ and its boxier incarnation~$\tilde\gamma$.
Notice how the original function~$f$ is on the left and the modified function~$\tilde f$ on the right, whereas the original curve~$\gamma$ is on the right and the modified curve~$\tilde\gamma$ is on the left.
Given the nature of the claim of the lemma, this should feel natural.

Let us denote $C=\{x\in[-\pi,\pi)^n;x_i=-\pi\text{ for some }i\}$.
This cone-shaped set~$C$ is the part of the boundary of the half-open cube contained in the cube: $C=[-\pi,\pi)^n\cap\partial([-\pi,\pi)^n)$.
As we will see in exercise~\ref{ex:torus-seam-hit}, the curve~$\tilde\gamma$ meets~$C$ at some point.
Replacing~$x$ with $x+q(sv)$ does not change the value of $\xrt_vf(x)$, as we will see in exercise~\ref{ex:torus-xrt-shift-invariance}.
Thus, we may just as well assume that $\iota(x)\in C$.

If $\tilde\gamma([0,1])$ is contained in~$C$, then it does not meet the support of~$\tilde f$.
Therefore $\tilde f\circ\gamma$ vanishes identically and so $\xrt_vf(x)=0$.
If the curve~$\tilde\gamma$ is not contained in this set, then it meets~$C$ only finitely many times; see exercise~\ref{ex:torus-seam-hit-finite}.
Call the number of times $m+1$ --- or~$m$ if we ignore the double counting of the endpoints.
By the previous considerations one of these times is at $t=0$, and by periodicity also at $t=1$.
Let the other times of hitting~$C$ be $0<t_1<\dots<t_m<1$, and denote $t_0=0$ and $t_{m+1}=1$.
It is possible that $m=0$ and there are no other times than the initial and final one.

We have
\begin{equation}
\begin{split}
\xrt_v\tilde f(x)
&=
\int_0^1\tilde f(x+tv)\dd t
\\&=
\int_0^1f(\iota(x+tv))\dd t
\\&=
\sum_{j=0}^{m}\int_{t_j}^{t_{j+1}}f(\iota(x+tv))\dd t
.
\end{split}
\end{equation}
Now, each $\int_{t_j}^{t_{j+1}}f(\iota(x+tv))\dd t$ is an integral of~$f$ over a straight line joining two boundary points of the cube~$[-\pi,\pi]^n$.
As noted earlier, these integrals are  determined by the X-ray transform $\xrt f$.
Therefore~$\xrt_v\tilde f$ can be written in terms of~$\xrt f$ --- although there is no particularly pretty formula --- and the proof is complete.
\end{proof}

\begin{ex}
Describe the function $\iota\circ q\colon\R^n\to[-\pi,\pi)^n$ in words, formulas, pictures, or a combination thereof.
\end{ex}

\begin{ex}
\label{ex:torus-seam-hit}
Explain why there is $s\in\R$ so that $x+sv\in C$ in the proof above.
That is, justify more carefully why the geodesic must hit the ``boundary''\footnote{This is the boundary of the fundamental domain. The torus itself has no boundary. A more appropriate word than ``boundary'' in the context of this construction of the torus would be ``seam''.}~$C$.
There are many different ways to do this, including the fact that $\abs{v}\geq2\pi$ or that it is impossible to return back to the starting point without changing direction and without jumping back.
\end{ex}

\begin{ex}
\label{ex:torus-seam-hit-finite}
Take any $w,x\in\T^1$ and $v\in2\pi\Z$.
Show that the function $[0,1]\ni t\mapsto x+tv\in\T^1$ either takes the value~$w$ at finitely many times or it only takes the value~$w$.
Then show that in general dimension a geodesic on the torus either stays on the seam $q(C)\subset\T^n$ (as used in the proof of lemma~\ref{lma:torus-to-Rn}) or only meets it finitely many times.
\end{ex}

\begin{ex}
\label{ex:torus-xrt-shift-invariance}
Explain why $\xrt_vg(x)=\xrt_vg(x+sv)$ for any $s\in\R$, $g\in C(\T^n)$, and $v\in2\pi\Z^n$.
\end{ex}

The conclusion of the lemma is important:

\begin{ex}
\label{ex:torus-to-Rn}
Suppose we know that $\xrt_vg=0$ for all $v\in2\pi\Z^n\setminus\{0\}$ implies that the function $g\in C(\T^n)$ has to vanish identically.
Show that if $\xrt f=0$ for some $f\in C_B$, then $f=0$.
\end{ex}

In other words, injectivity of the X-ray transform in the Euclidean space follows from an injectivity result on the torus.
This is our first solution of the inverse problem of X-ray tomography.
The missing step is proving the desired result on the torus.

\subsection{Interplay between the X-ray and Fourier transforms on a torus}

For any fixed $v\in2\pi\Z^n\setminus\{0\}$ and $f\in C(\T^n)$, the X-ray transform~$\xrt_vf$ is a continuous function on the torus~$\T^n$.
Therefore it makes sense to calculate its Fourier transform.
The lemma below is a version of the Fourier slice theorem, which we will see in its more common form in exercise~\ref{ex:og-fourier}.

\begin{lemma}
\label{lma:ft-xrt-torus}
Let $v\in2\pi\Z^n\setminus\{0\}$ and $f\in C(\T^n)$.
Then for every $k\in\Z^n$
\begin{equation}
\ft(\xrt_vf)(k)
=
\begin{cases}
\ft f(k) & \text{when }k\cdot v=0\\
0 & \text{otherwise}.
\end{cases}
\end{equation}
\end{lemma}

\begin{proof}
The proof is a mere calculation:
\begin{equation}
\label{eq:vv3}
\begin{split}
\ft(\xrt_vf)(k)
&=
\frac1{(2\pi)^n}
\int_{\T^n} e^{-ik\cdot x} \xrt_vf(x) \dd x
\\&=
\frac1{(2\pi)^n}
\int_{\T^n} e^{-ik\cdot x} \int_0^1 f(x+tv) \dd t \dd x
\\&\stackrel{\text{a}}{=}
\frac1{(2\pi)^n}
\int_0^1\int_{\T^n} e^{-ik\cdot x} f(x+tv) \dd x \dd t
\\&\stackrel{\text{b}}{=}
\frac1{(2\pi)^n}
\int_0^1\int_{\T^n} e^{-ik\cdot (y-tv)} f(y) \dd y \dd t
\\&\stackrel{\text{c}}{=}
\frac1{(2\pi)^n}
\int_{\T^n} e^{-ik\cdot y} f(y) \dd y
\times
\int_0^1e^{i(k\cdot v)t}\dd t
\\&\stackrel{\text{d}}{=}
\ft f(k)
\times
\begin{cases}
1 & \text{when }k\cdot v=0\\
0 & \text{otherwise}.
\end{cases}
\end{split}
\end{equation}
It only remains to justify the steps.
\end{proof}

\begin{ex}
Explain the steps a--d in~\eqref{eq:vv3}.
\end{ex}

We will next show that the X-ray transform is injective.
Bear in mind that the X-ray transform is understood as a family of operators.
Here injectivity means ``collective injectivity''; the individual operators are not injective.

\begin{theorem}
\label{thm:xrt-torus}
Let $f\in C(\T^n)$.
If $\xrt_vf=0$ for all $v\in2\pi\Z^n\setminus\{0\}$, then $f=0$.
\end{theorem}

\begin{proof}
Since the Fourier transform is bijective by theorem~\ref{thm:hd-fs}, it suffices to show that the Fourier series of~$f$ vanishes.
To that end, take any $k\in\Z^n$.
There is some $v\in2\pi\Z^n\setminus\{0\}$ so that $k\cdot v=0$ (exercise~\ref{ex:Z-OG}).
By lemma~\ref{lma:ft-xrt-torus} we have $\ft\xrt_vf(k)=\ft f(k)$.
Since $\xrt_vf=0$ by assumption, we get $\ft f(k)=0$ for all $k\in\Z^n$.
\end{proof}

\begin{ex}
\label{ex:Z-OG}
Show that for any $k\in\Z^n$ there exists $w\in\Z^n\setminus\{0\}$ so that $k\cdot w=0$.
\end{ex}

\begin{ex}
Show that if $v\in2\pi\Z^n\setminus\{0\}$, then~$\xrt_v$ is not injective.
Use lemma~\ref{lma:ft-xrt-torus} or take a function $f\in C^\infty(\T^n)$ and consider the function $v\cdot\nabla f(x)$.
\end{ex}

\begin{ex}
\label{ex:torus-xrt-scaling}
Let $v\in2\pi\Z^n\setminus\{0\}$ and $m\in\Z\setminus\{0\}$.
Show that $\xrt_{mv}=\xrt_v$.
\end{ex}

\begin{ex}
All of the results in this section are valid for $n=1$ apart from exercise~\ref{ex:Z-OG}.
When $n=1$, one can only find an orthogonal $w\in\Z$ for $k=0$.
By exercise~\ref{ex:torus-xrt-scaling} all one can measure about $f\in C(\T^1)$ is~$\xrt_{2\pi}f$.
What does this mean for recovering the Fourier coefficients~$\ft f(k)$?
(Compare to exercise~\ref{ex:1.7}.)
\end{ex}

\begin{ex}
We have excluded $v=0$ from our discussion.
Why is this reasonable, considering the original problem?
What is the operator~$\xrt_0$?
\end{ex}

As a corollary, we get the following injectivity result:

\begin{theorem}
\label{xrtthm:torus}
Suppose $f\in C_B$ integrates to zero over all lines through~$B$.
Then $f=0$.
\end{theorem}

\begin{proof}
This follows from lemma~\ref{lma:torus-to-Rn}, exercise~\ref{ex:torus-to-Rn}, and theorem~\ref{thm:xrt-torus}.
\end{proof}

\begin{ex}
Summarize in your own words the proof of injectivity of the X-ray transform given in this section.
\end{ex}

Injectivity in a larger ball and therefore in the whole space~$C_c(\R^n)$ follows by a scaling argument.

\begin{figure}[t]
    \centering
    \includegraphics[scale=0.35,trim={3cm 0 3cm 0},clip,page=12]{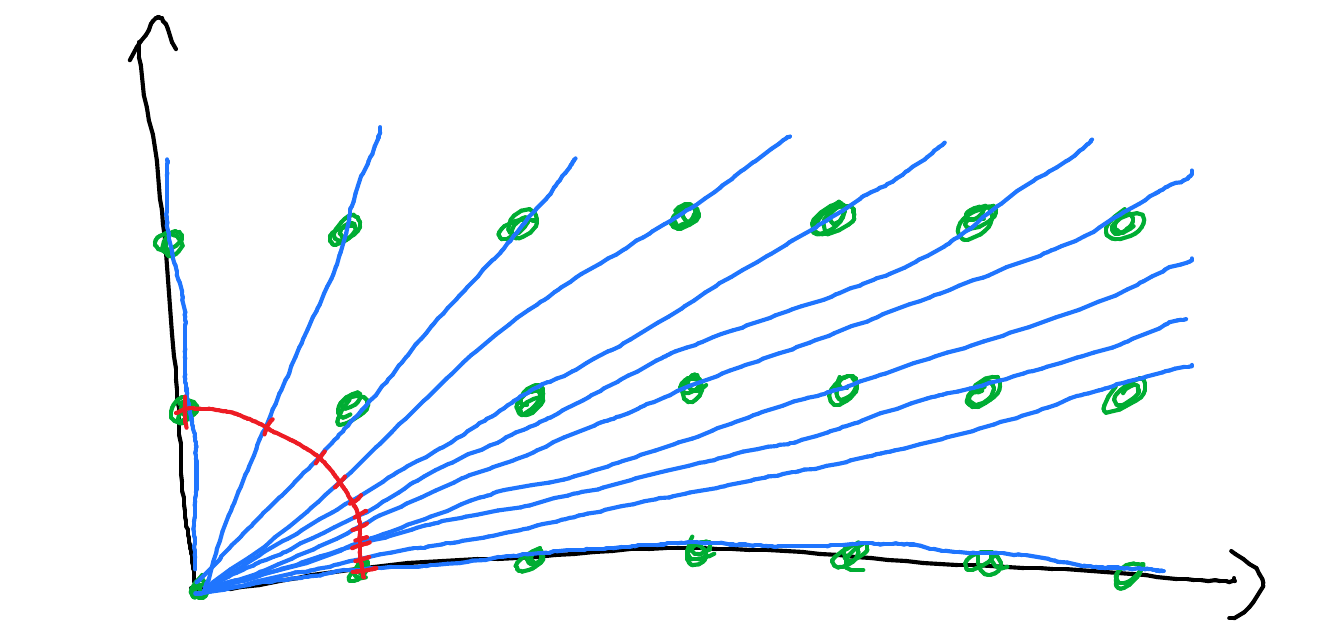}
    \caption{The \sini{directions} passing through \viher{lattice points} are the ones we used on the torus. When \puna{projected to the unit sphere~$\Sphere^1$}, the set of directions is dense. Any single line passes through several lattice points, which is closely related to exercise~\ref{ex:torus-xrt-scaling}.}
    \label{fig:torus-directions}
\end{figure}

It is worth noting that in this proof we did not use X-rays in all directions.
Only the directions in $2\pi\Z^n\setminus\{0\}$ were used.
If one projects this set radially to the unit sphere~$\Sphere^{n-1}$, one gets a countable dense set as illustrated in figure~\ref{fig:torus-directions}.

\qa

\section{Injectivity via angular Fourier series}
\label{sec:ang-fs}

In this section and the next section we will give our second injectivity proof based on Fourier series with respect to the angular variable in polar coordinates.

\subsection{Angular Fourier series}

In this section we will give a new way to prove injectivity of the X-ray transform.
This is the one found by Allan Cormack, who together with the electrical engineer Godfrey Hounsfield was awarded the Nobel Prize in Physiology or Medicine for the development of computer assisted tomography in 1979.
However, Cormack was not the first one to solve the mathematical inverse problem; it had been done in 1917 by Johann Radon, but without an idea to apply it to tomography.
Radon's inversion method will be covered in section~\ref{sec:radon}.

We study the problem in two dimensions.
It is most convenient to consider the problem in the punctured closed disc
\begin{equation}
\label{eq:pud}
\pud
=
\{x\in\R^2;0<\abs{x}\leq1\}.
\end{equation}
Recall exercise~\ref{ex:2d-hd} concerning the two-dimensional case.

Our aim is to reconstruct a continuous function $f\colon\pud\to\C$ from its integrals over all lines through~$\pud$.
We will not use the lines that pass through the origin.
That is, we throw away some data.
Avoiding the origin simply makes the use of polar coordinates more convenient and does not make the result any weaker.

Our original problem was to reconstruct a function in the whole disc, but it turns out to be convenient to throw away some data.
This is not unusual in inverse problems.
It is often best to look at a convenient subset of the data.
Notice that in this section we will throw away a different subset of data than in the previous one.
However, the results are often stated for all of the data for clarity.

We will use polar coordinates $r\in(0,1]$ and $\theta\in\R/2\pi\Z$ on~$\pud$.
For any fixed~$r$, the function $f(r,\dummy)$ is a continuous function $\R/2\pi\Z\to\C$.
We expand it in a Fourier series.
Now the coefficients of the Fourier series depend on the variable~$r$.
We have
\begin{equation}
\label{eq:fs-ang-ak}
f(r,\theta)
=
\sum_{k\in\Z}a_k(r)e^{ik\theta}.
\end{equation}
The Fourier coefficients may be calculated as
\begin{equation}
\label{eq:ang-f-component}
a_k(r)
=
\frac1{2\pi}
\int_0^{2\pi} e^{-ik\theta}f(r,\theta)\dd\theta.
\end{equation}
From this expression one can see that each $a_k\colon(0,1]\to\C$ is continuous.
The only difference to the usual Fourier series on the circle is the appearance of the parameter~$r$.

We will write $f_k(r,\theta)=a_k(r)e^{ik\theta}$, so that the Fourier series becomes simply
\begin{equation}
\label{eq:fs-ang-fk}
f(r,\theta)
=
\sum_{k\in\Z}
f_k(r,\theta).
\end{equation}
We will not study the details of this series too deeply, but we remark that the terms are $L^2$-orthogonal and the usual~$L^2$ theory of Fourier series applies with some modifications due to the presence of~$r$.
It will suffice for us that $f=0$ if and only if $f_k=0$ for all $k\in\Z$.

\begin{ex}
\label{ex:ang-fs}
Suppose $f\colon\pud\to\C$ is continuous.
Show that the following are equivalent:
\begin{enumerate}[(a)]
\item $f=0$
\item $f_k=0$ for all $k\in\Z$
\item $a_k=0$ for all $k\in\Z$
\end{enumerate}
Theorem~\ref{thm:1d-fs} will be of use.
In fact, the whole angular Fourier series makes sense because of this theorem.
\end{ex}

In higher dimensions the functions~$e^{ik\theta}$ need to be replaced with spherical harmonics.
This is one of the reasons why it is convenient to restrict to dimension two.

One can study the angular Fourier series in the whole plane if one wants.
As long as the function is continuous or~$L^2$ (or whatever space one might be working with), one can apply the one-dimensional Fourier series circle by circle.

\subsection{The X-ray transform in polar coordinates}

For any point $x\in\pud$, let~$L_x$ be the line segment connecting boundary points of the unit disc so that~$x$ is the closest point to the origin on~$L_x$.
If $\abs{x}=1$, the line will degenerate into a point.
This is a convenient way to parametrize all lines through the closed unit disc that do not meet the origin.

For a continuous function $f\colon\pud\to\C$, we define~$\xrt f(x)$ to be the integral of~$f$ over~$L_x$.
The mapping $x\mapsto L_x$ gives a bijection between the set of points and the set of lines.
Again, we use polar coordinates, so that the X-ray transform of~$f$ is a function $\xrt f(r,\theta)$.
It will be useful to write this as a Fourier series in the variable~$\theta$.

\begin{figure}[t]
    \centering
    \includegraphics[scale=0.4,trim={3cm 0 3cm 0},clip,page=13]{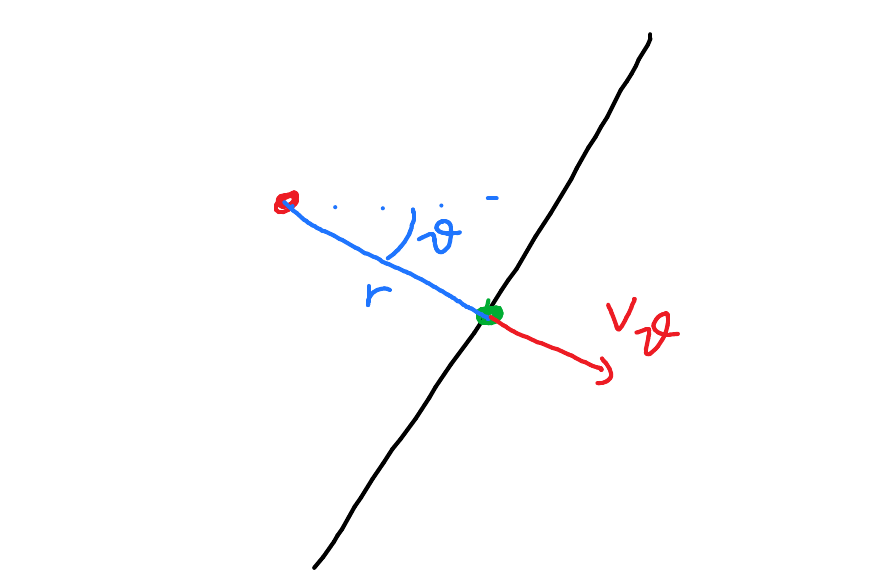}
    \caption{The line~$L_{r,\theta}$ described by \viher{the closest point} with \sini{polar coordinates $(r,\theta)$} and \puna{the corresponding unit vector~$v_\theta$}.}
    \label{fig:L-r-theta}
\end{figure}

For $\theta\in\R/2\pi\Z$, denote $v_\theta=(\cos(\theta),\sin(\theta))$.
For $r>0$ and $\theta\in\R/2\pi\Z$, the corresponding line can be written as
\begin{equation}
L_{r,\theta}
=
\{x\in\R^2;x\cdot v_\theta=r\}.
\end{equation}
See figure~\ref{fig:L-r-theta}.
As mentioned above, this covers all the lines that do not meet the origin.
If we use ``extended polar coordinates'' where $r\geq0$, the we can indeed parametrize all lines.
In some sense, this corresponds to replacing the origin with ``directed origins'', which is a compactification of the punctured disc and can be considered a blow-up of the origin.
In fact, one can even let the radius~$r$ to be any real number; this would lead to a global two-fold parametrization of all the lines.

\subsection{Rotations and diagonalizability}

Fix any $\phi\in\R$.
Let us define the rotation operator~$\rot_\phi$ on functions defined on~$\pud$ so that $(\rot_\phi f)(r,\theta)=f(r,\theta+\phi)$.
It is clear that~$\rot_\phi$ maps continuous functions to continuous functions.

For a continuous $f\colon\pud\to\C$, both~$f$ and~$\xrt f$ are functions on~$\pud$.
That is, we have two operators acting conveniently on the same space: $\xrt,\rot_\phi\colon C(\pud)\to C(\pud)$.
This allows us to make sense of the function~$\rot_\phi\xrt f$.
The interplay between rotations and the X-ray transform is important.

\begin{ex}
\label{ex:rot-commute}
Take any $\phi\in\R$ and a continuous $f\colon\pud\to\C$.
Explain why our two operators commute: $\xrt\rot_\phi f=\rot_\phi\xrt f$.
\end{ex}

The fact that rotations commute with the X-ray transform will bring additional structure.

The next lemma says that the X-ray transform commutes with the integral with respect to a parameter.

\begin{lemma}
\label{lma:ang-fubini}
Let $f(x;\phi)$ be a continuous function defined on $\pud\times[0,2\pi]$ (or $\pud\times\T^1$).
Let~$\gamma$ be any line through~$\pud$ that does not meet the origin.
Then
\begin{equation}
\int_0^{2\pi}(\xrt f(\dummy;\phi))(\gamma)\dd\phi
=
\xrt F(\gamma),
\end{equation}
where $F(x)=\int_0^{2\pi}f(x;\phi)\dd\phi$.
\end{lemma}

\begin{ex}
Prove the lemma.
It may or may not be useful that it is irrelevant that the parameter space is $[0,2\pi]$.
Integrating over a much more general parameter space would work just as well.
\end{ex}

\begin{lemma}
\label{lma:ang-fs-xrt}
Let $f\colon\pud\to\C$ be a continuous function.
Then
\begin{equation}
\frac1{2\pi}
\int_0^{2\pi}e^{-ik\theta}\xrt f(r,\theta)\dd\theta
=
\xrt f_k(r,0)
\end{equation}
for all $k\in\Z$.
\end{lemma}

The angle~$0$ might seem weird at first.
It is best to regard the right-hand side as the X-ray transform of the one-dimensional function~$a_k$.
Introducing a non-zero angle is possible in the formula above, but it gives no additional information.

\begin{proof}[Proof of lemma~\ref{lma:ang-fs-xrt}]
In the integrals below limits are occasionally shifted from $(0,2\pi)$ due to changes of variables.
Since the relevant functions are $2\pi$-periodic, we do not need to change the interval of integration.

Fix any $k\in\Z$.
First, we observe that
\begin{equation}
\label{eq:vv4}
\begin{split}
\frac1{2\pi}
\int_0^{2\pi}e^{-ik\theta}\rot_\theta f(r,\phi)\dd\theta
&=
\frac1{2\pi}
\int_0^{2\pi}e^{-ik\theta}f(r,\phi+\theta)\dd\theta
\\&=
\frac1{2\pi}
\int_0^{2\pi}e^{-ik(\omega-\phi)}f(r,\omega)\dd\omega
\\&=
e^{ik\phi}a_k(r)
\\&=
f_k(r,\phi).
\end{split}
\end{equation}
Using
the definitions,
exercise~\ref{ex:rot-commute},
lemma~\ref{lma:ang-fubini}, and
equation~\eqref{eq:vv4},
we get
\begin{equation}
\begin{split}
\frac1{2\pi}
\int_0^{2\pi}e^{-ik\theta}\xrt f(r,\theta)\dd\theta
&=
\frac1{2\pi}
\int_0^{2\pi}e^{-ik\theta}\rot_\theta\xrt f(r,0)\dd\theta
\\&=
\frac1{2\pi}
\int_0^{2\pi}e^{-ik\theta}\xrt\rot_\theta f(r,0)\dd\theta
\\&=
\frac1{2\pi}
\int_0^{2\pi}\xrt (e^{-ik\theta}\rot_\theta f)(r,0)\dd\theta
\\&=
\xrt\left(
\frac1{2\pi}\int_0^{2\pi}(e^{-ik\theta}\rot_\theta f)\dd\theta
\right)
(r,0)
\\&=
\xrt f_k(r,0).
\end{split}
\end{equation}
This concludes the proof.
\end{proof}

\begin{ex}
\label{ex:xrt-fs}
The function $f(r,\theta)$ was written as a Fourier series $f=\sum_{k\in\Z}f_k$ in~\eqref{eq:fs-ang-fk}.
Similarly, $g(r,\theta)=\xrt f(r,\theta)$ can be written as a Fourier series $g=\sum_{k\in\Z}g_k$.
Give a formula for the function~$g_k$ in terms of~$\xrt f$.

Explain why~$g_k$ depends on~$f_k$ but not on any other~$f_m$ for $m\neq k$.
\end{ex}

Our goal, as always, is to show that if $\xrt f=0$, then $f=0$.
By exercise~\ref{ex:xrt-fs} it follows from the assumption that $\xrt f_k=0$ for every $k\in\Z$.
We will then fix any~$k$ and show that $\xrt f_k=0$ implies $f_k=0$.
This problem is essentially one-dimensional, since~$f_k$ corresponds to the continuous function $a_k\colon(0,1]\to\C$.
That is, we are converting a single two-dimensional problem into a sequence of independent one-dimensional problems.
The aim of the next section is to solve this family of one-dimensional problems.
After that we know that $f_k=0$ for each~$k$, and so $f=0$.
What we have now established is that these one-dimensional problems are independent of each other; without this independence the decomposition into Fourier components would have been of little use.

\subsection{Remarks on symmetries and functional analysis}

The Fourier series of the X-ray transform depends in a rather simple way on the Fourier series of the original function.
The~$k$th Fourier component of the X-ray transform only depends on the~$k$th Fourier component of the function.
This is not a coincidence.

The X-ray transform is an operator that takes a function on~$\pud$ to another function on~$\pud$.
It commutes with the rotation operator~$\rot_\theta$ for any~$\theta$, so it also commutes with the derivative~$\partial_\theta$ with respect to the angular coordinate.
One can see this by considering the angular derivative as an infinitesimal rotation or by appealing to equation~\eqref{eq:lie-algebra} below.
Of course, the derivative operator does not map continuous functions to continuous functions, so it should be defined on a different space or be treated as an unbounded operator, but we ignore this technicality.
Anyhow, the important observation is that the derivative~$\partial_\theta$ and the X-ray transform~$\xrt$ commute.

At least physicists are likely to remember that if two real symmetric (or complex hermitian) matrices commute, they are simultaneously diagonalizable.
Similar results hold for infinite-dimensional spaces.
The symmetric situation is an example of a broader phenomenon.
The operator~$i\partial_\theta$ is self-adjoint, but the integral transform does not need to have (or have) any such adjointness properties.

In our specific case this suggests that if we write the whole function space as a direct sum of eigenspaces of~$\partial_\theta$ (or~$i\partial_\theta$), then the X-ray transform is block diagonal.
This is indeed what happens.

\begin{ex}
\label{ex:block-diag1}
Let us first consider a simpler operator in block form.
Take some positive integers $n_1,\dots,n_K$ and define $E=\R^{n_1}\times\cdots\times\R^{n_K}$.
Let $E_k\subset E$ denote the subspace where all but the $k$th component are zero.
Then a linear operator $A\colon E\to E$ can be written in a block form
\begin{equation}
A
=
\begin{pmatrix}
A_{11} & \cdots & A_{1K} \\
\vdots & \ddots & \vdots \\
A_{K1} & \cdots & A_{KK}
\end{pmatrix},
\end{equation}
where each $A_{ij}$ is a linear map $\R^{n_j}\to\R^{n_i}$.
(If each~$n_k$ is~$1$, then this is just the usual matrix representation.)
Show that the matrix $A$ is block diagonal (i.e. $A_{ij}=0$ when $i\neq j$) if and only if $A(E_k)\subset E_k$ for all $k=1,\dots,K$.
\end{ex}

\begin{ex}
\label{ex:block-diag2}
Take the block diagonal matrix of the previous exercise.
Show that $A$ is invertible if and only if each diagonal matrix $A_{kk}$ is.
\end{ex}

For example, in~$L^2(D)$, the eigenspace of~$\partial_\theta$ with eigenvalue~$ik$, $k\in\Z$, is
\begin{equation}
H_k
=
\{f\in L^2(D);f(r,\theta)=a(r)e^{ik\theta}\text{ for some function }a\}.
\end{equation}
One can do similar things over other function spaces.
Our result in this section shows (apart from regularity assumptions), that $\xrt(H_k)\subset H_k$.
It is also somewhat easy to see that $H_k\perp H_m$ when $k\neq m$.
Moreover, if the operator has suitable continuity properties in~$L^2$ --- and the X-ray transform does --- then one has a very convenient theory in a Hilbert space.

To be more specific, we have
\begin{equation}
L^2(D)
=
L^2(\pud)
=
\bigoplus_{k\in\Z}H_k.
\end{equation}
This direct sum means simply that any $f\in L^2(D)$ can be written as a convergent series $f=\sum_{k\in\Z}f_k$ with $f_k\in H_k$.
This is the decomposition of~$f$ in the eigenspaces of $i\partial_\theta$.

In higher dimensions~$H_k$ can be defined similarly, with the exponential functions~$e^{ik\theta}$ replaced by spherical harmonics.

We have found that in this angular Fourier decomposition the X-ray transform looks like the block diagonal matrix of exercise~\ref{ex:block-diag1}.
In light of exercise~\ref{ex:block-diag2} we will turn our attention to the operators on the diagonal in the next section.


\begin{ex}
Let the two matrices $A,B\in\R^{n\times n}$ be symmetric.
For simplicity, you may additionally assume that all eigenvalues have multiplicity one.
(The result will be true without this assumption.)
Prove that if $AB=BA$, then there is an orthogonal matrix~$U$ so that~$UAU^T$ and~$UBU^T$ are both diagonal.
You may assume it known that for a single real symmetric matrix such a~$U$ exists.
\end{ex}

Passing from rotation symmetry ($\rot_\phi$) to angular derivatives ($\partial_\theta$) was a useful trick.
One may ask how one might find the derivative operator, given the rotations.
A formal calculation gives
\begin{equation}
\label{eq:lie-algebra}
\partial_\theta
=
\left.\frac{\der}{\der\phi}\rot_\phi\right|_{\phi=0}.
\end{equation}
This is a derivative of an operator with respect to a parameter, and it is convenient at times to differentiate at the level of operators rather than functions.
The derivative does not exist as a limit of the difference quotient in~$L^2$, but it does exist in~$C^\infty$, for example.

Much of this may also be seen from the point of view from Lie groups.
Passing from a full symmetry to a differential or infinitesimal symmetry is an example of passing from a Lie group to its Lie algebra.
If something commutes with the action of a Lie group, then it commutes with that of the Lie algebra.
If the Lie group acts by rotations or translations, then the Lie algebra acts by differentiation.
In our planar case the Lie group~$SO(2)$ and its Lie algebra~$\mathfrak{so}(2)$ are one-dimensional and the basis vector of the algebra acts as~$\partial_\theta$.
In more complicated cases the algebra is larger, and it may be more convenient to think of the action by the Casimir operator of the Lie algebra.
In the case of~$SO(n)$ acting on~$\Sphere^{n-1}$ in the usual way, the Casimir operator acts as the spherical Laplacian.
This is how the spherical harmonic decomposition --- an eigenvalue decomposition of the the spherical Laplacian --- arises naturally from spherical symmetry.
The overall lesson is that when an operator commutes with the action of a Lie group, you should expect it to be block-diagonal in the sense that it respects eigenspaces of the action of the corresponding Lie algebra --- or its Casimir element.

\begin{ex}
Show that if~$f\in C^1(\pud)$, then
\begin{equation}
\partial_\theta f(r,\theta)
=
\left.\frac{\der}{\der\phi}(\rot_\phi f)(r,\theta)\right|_{\phi=0}.
\end{equation}
In other words, prove equation~\eqref{eq:lie-algebra}.
(Notice that while~$\rot_\phi$ maps $C^1(\pud)\to C^1(\pud)$, the derivative~$\partial_\theta$ only maps $C^1(\pud)\to C^0(\pud)$.)
\end{ex}

Another thing worth pointing out is that rotation symmetry was crucially important, but Euclidean geometry was not.
Similar arguments work in other rotation symmetric situations.

\qa

\section{Abel transforms}
\label{sec:abel}

\subsection{The block diagonal structure of the X-ray transform in polar coordinates}

As discovered in the previous section, the X-ray transform has a peculiar block diagonal structure.
Now it remains to find the operators on the diagonal.
We shall not use the block diagonal structure in any formal way, but it is an underlying idea the reader should be aware of.

Consider the function $f_k(r,\theta)=e^{ik\theta}a_k(r)$, where $a_k\colon(0,1]\to\C$ is continuous.
We want to find an explicit formula for the X-ray transform of~$f_k$.

To this end, consider the line~$L_{s,\phi}$ whose closest point to the origin is~$(s,\phi)$.
We may assume $0<s<1$.
Using arc length parametrization (unit speed parametrization) with zero parameter at the midpoint, we can write this line as the curve
\begin{equation}
\gamma\colon[-\sqrt{1-s^2},\sqrt{1-s^2}]\to\pud
\end{equation}
with
\begin{equation}
\gamma(t)
=
(\sqrt{s^2+t^2},\phi+\arctan(t/s)).
\end{equation}
The trace of the curve~$\gamma$ is exactly~$L_{s,\phi}$.s

\begin{figure}[t]
    \centering
    \includegraphics[scale=0.5,trim={3cm 0 3cm 0},clip,page=14]{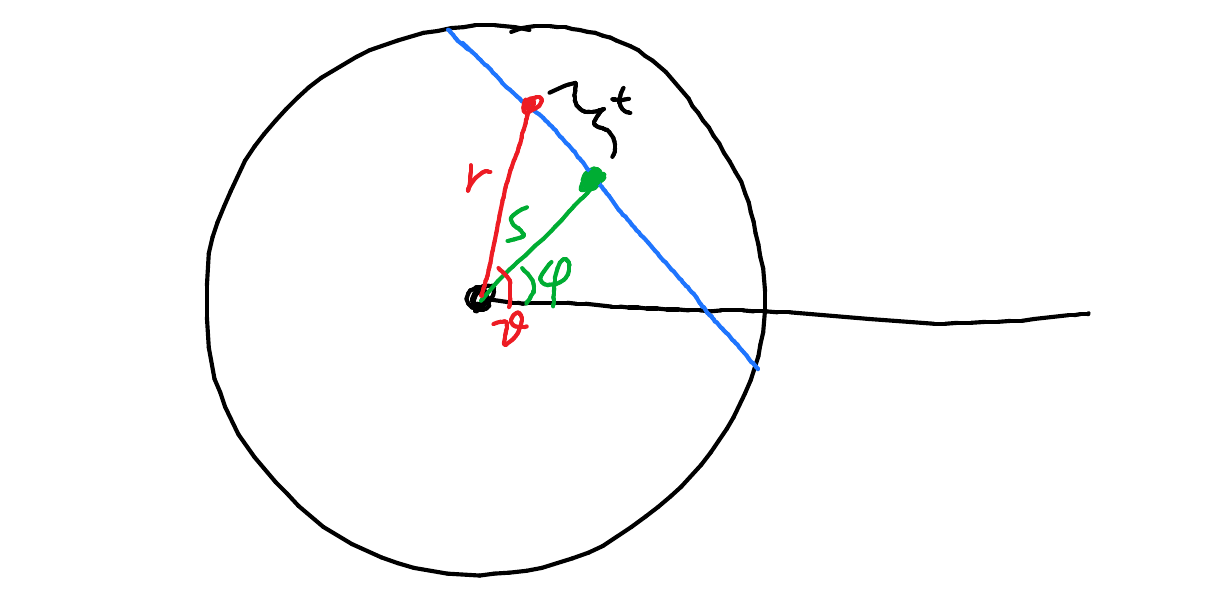}
    \caption{An integral over \sini{a line} leads to an Abel transform. \viher{The closest point} and \puna{a point on the line} are both described in polar coordinates. The parameter~$t$ is the travel time along \sini{the line}, with $t=0$ at \viher{the closest point}.}
    \label{fig:abel-line}
\end{figure}

\begin{ex}
Justify this formula geometrically.
Figure~\ref{fig:abel-line} may be of help.
\end{ex}

We will split the interval in two halves and change the variable of integration from arc length $t\in(0,\sqrt{1-s^2})$ to radius $r=\sqrt{s^2+t^2}\in(s,1)$ on each half.

Now we can simply calculate:\footnote{Here and henceforth $\sum_\pm$ means summing over the two sign options.}
\begin{equation}
\label{eq:vv5}
\begin{split}
\xrt f_k(s,\phi)
&=
\int_{-\sqrt{1-s^2}}^{\sqrt{1-s^2}}f_k(\sqrt{s^2+t^2},\phi+\arctan(t/s))\dd t
\\&=
\sum_{\pm}
\int_0^{\sqrt{1-s^2}}f_k(\sqrt{s^2+t^2},\phi\pm\arctan(t/s))\dd t
\\&=
\sum_{\pm}
\int_s^1f_k(r,\phi\pm\arccos(s/r))\frac{\der t}{\der r}\dd r
\\&=
\sum_{\pm}
\int_s^1a_k(r)e^{ik\phi\pm ik\arccos(s/r)}\frac{1}{\sqrt{1-(s/r)^2}}\dd r
.
\end{split}
\end{equation}
Two steps need justification, and they are left as the following two exercises.

\begin{ex}
Explain why $\arctan(t/s)=\arccos(s/r)$.
\end{ex}

\begin{ex}
Why is the Jacobian $\frac{\der t}{\der r}$ equal to $1/\sqrt{1-(s/r)^2}$ as indicated above?
\end{ex}

Our change of variable was in fact singular.
But the singularity is integrable and our calculation is still valid, but to be pedantic, one may want to consider the integral with $t\in(\eps,\sqrt{1-s^2})$ first and then let $\eps\to0$.

To proceed with the calculation, we must do some trigonometric manipulations.

\begin{ex}
Show that $\sum_{\pm}e^{ik\phi\pm ik\arccos(s/r)}=2e^{ik\phi}\cos(k\arccos(s/r))$.
\end{ex}

It turns out that for $k\in\Z$ and $x\in[-1,1]$, we have $\cos(k\arccos(x))=T_{\abs{k}}(x)$, where~$T_k$ is the~$k$th Chebyshev polynomial of the first kind.
For convenience, we will use the notation~$T_k$ instead of~$T_{\abs{k}}$ even when $k<0$.
It follows from this cosine property of the Chebyshev polynomials that $\max_{x\in[0,1]}T_k(x)=1$ for any $k\in\Z$.
This family of polynomials can be defined recursively for $k\in\N$ by $T_0(x)=1$, $T_1(x)=x$, and $T_k(x)=2xT_{k-1}(x)-T_{k-2}(x)$.
Once one establishes this recursion relation, it follows that the function~$T_k$ is indeed a polynomial.

\begin{ex}
Justify the formulas for~$T_0$ and~$T_1$ and recursion relation for~$T_k$ using the property that $\cos(kx)=T_k(\cos(x))$ for all $k\in\N$.
\end{ex}

Now we can proceed from~\eqref{eq:vv5} to
\begin{equation}
\xrt f_k(s,\phi)
=
2e^{ik\phi}
\int_s^1a_k(r)\frac{T_k(s/r)}{\sqrt{1-(s/r)^2}}\dd r.
\end{equation}
Based on the last section (exercise~\ref{ex:xrt-fs} implies that Fourier transform of the~$k$th Fourier component of~$f$ contains only the~$k$th Fourier component), we expected to pull out the factor~$e^{ik\phi}$, but the exact structure of the rest might be a bit of a surprise.

\subsection{Abel transforms}

\begin{definition}
\label{def:abel}
Fix any $k\in\Z$.
For a continuous function $h\colon(0,1]\to\C$ we define a new continuous function $\A_kh\colon(0,1]\to\C$ by
\begin{equation}
\label{eq:Ak-def}
(\A_kh)(s)
=
2\int_s^1h(r)\frac{T_k(s/r)}{\sqrt{1-(s/r)^2}}\dd r.
\end{equation}
Here~$T_k$ is the $\abs{k}$th Chebyshev polynomial.
We call~$\A_k$ the~$k$th generalized Abel transform.
\end{definition}

\begin{ex}
The integral above is actually only defined for $s\in(0,1)$.
Show that $\lim_{s\to1}\A_kh(s)=0$, so that it makes sense to let $\A_kh(1)=0$ regardless of the value~$h(1)$.

This can be interpreted geometrically through the X-ray transform.
We are calculating the integral of a bounded function over a line segment with distance~$s$ from the origin.
As $s\to1$, this line segment shrinks to a point, so we should expect the integral to go to zero.
\end{ex}

The reason for calling~$\A_k$ a generalized Abel transform is that for $k=0$ we have $T_0\equiv1$ and~$\A_0$ is (one form of) the Abel transform.
These are all integral transforms in the sense that they take one function on the interval and turn it into another function on the interval by means of an integral formula.

We have thus found that if
\begin{equation}
f(r,\theta)
=
\sum_{k\in\Z}e^{ik\theta}a_k(r),
\end{equation}
then
\begin{equation}
\xrt f(r,\theta)
=
\sum_{k\in\Z}e^{ik\theta}\A_k a_k(r).
\end{equation}
This means that the Abel transforms are the operators on our block diagonal.
(See exercises~\ref{ex:block-diag1} and~\ref{ex:block-diag1} and the surrounding discussion.)

We want to show for all $k\in\Z$ that if $\A_kf_k=0$, then $f_k=0$.
That is, we want to show that the generalized Abel transform $\A_k\colon C((0,1])\to C((0,1])$ is an injection for all $k\in\Z$.
(We will not need or prove that~$\A_k$ maps continuous functions to continuous functions, although it is true.)

\begin{lemma}
\label{lma:abel}
The generalized Abel transform $\A_k\colon C((0,1])\to C((0,1])$ is an injection.
Moreover, $h\in C((0,1])$ can be calculated from~$\A_k h$ via
\begin{equation}
\label{eq:abel-inv}
h(r)
=
-\frac1\pi\frac{\der}{\der r}
\int_r^1\A_kh(s)\frac{T_k(s/r)}{s\sqrt{(s/r)^2-1}}\dd s.
\end{equation}
\end{lemma}

\begin{proof}
It suffices to prove~\eqref{eq:abel-inv}.
We start by examining the integral:
\begin{equation}
\label{eq:vv6}
\begin{split}
J(r)
&\coloneqq
\int_r^1\A_kh(s)\frac{T_k(s/r)}{s\sqrt{(s/r)^2-1}}\dd s
\\&=
\int_r^1
\left(
2\int_s^1h(t)\frac{T_k(s/t)}{\sqrt{1-(s/t)^2}}\dd t
\right)
\frac{T_k(s/r)}{s\sqrt{(s/r)^2-1}}\dd s
\\&=
2
\int_r^1
\int_s^1
h(t)
\frac{T_k(s/t)}{\sqrt{1-(s/t)^2}}
\frac{T_k(s/r)}{s\sqrt{(s/r)^2-1}}
\dd t
\dd s
\\&=
2
\int_r^1
\int_r^t
h(t)
\frac{T_k(s/t)}{\sqrt{1-(s/t)^2}}
\frac{T_k(s/r)}{s\sqrt{(s/r)^2-1}}
\dd s
\dd t
\\&=
2
\int_r^1
h(t)
K_k(r,t)
\dd t
,
\end{split}
\end{equation}
where
\begin{equation}
\label{eq:K-def}
K_k(r,t)
=
\int_r^t
\frac{T_k(s/t)}{\sqrt{1-(s/t)^2}}
\frac{T_k(s/r)}{s\sqrt{(s/r)^2-1}}
\dd s.
\end{equation}
Now, somewhat magically,
\begin{equation}
\label{eq:K-fact}
K_k(r,t)
=
\frac\pi2
\quad
\text{whenever $0<r<t$ and $k\in\Z$}.
\end{equation}
Some cases will be treated in the exercises.
Therefore
\begin{equation}
J(r)
=
\pi
\int_r^1
h(t)
\dd t.
\end{equation}
The desired result now follows by the fundamental theorem of calculus.
\end{proof}

\begin{ex}
Explain why the limits work as they do when we applied Fubini's theorem in~\eqref{eq:vv6}.
\end{ex}

\begin{ex}
Prove that for every $\lambda>0$ we have $K_k(\lambda r,\lambda t)=K_k(r,t)$.
You can use this to make simplifying assumptions in subsequent calculations if you want to.
\end{ex}

\begin{ex}
Make the change of variable $s^2=\frac12[(t^2+r^2)+y(t^2-r^2)]$ to~\eqref{eq:K-def} and simplify the resulting expression.
It is wise to regard the measure as~$\der s/s$.
You can leave the Chebyshev polynomials untouched.
\end{ex}

\begin{ex}
Calculate by hand
\begin{equation}
\int_{-1}^1\frac{1}{(a+y)\sqrt{1-y^2}}\dd y,
\end{equation}
where $a>1$ is a real parameter.
It may be convenient to differentiate $\arctan\left(\frac{1+ay}{\sqrt{(a^2-1)(1-y^2)}}\right)$.
\end{ex}

\begin{ex}
Making use of $T_0(x)=1$, $T_1(x)=x$, and the previous exercises, calculate $K_0(r,t)$ and $K_1(r,t)$.
\end{ex}

\begin{bex}
Prove the recurrence relation $K_{k+2}(r,t)=K_k(r,t)$.
This together with the previous results shows~\eqref{eq:K-fact}.
\todo{Do this yourself and rethink. Hints?}
\end{bex}

One can define an operator~$\B_k$ by
\begin{equation}
\B_kh(r)
=
\frac1\pi
\int_r^1h(s)\frac{T_k(s/r)}{s\sqrt{(s/r)^2-1}}\dd s.
\end{equation}
It seems that~$\B_k$ is very similar in nature to~$\A_k$, so one expects it to have similar mapping properties.
Since $\B_k\A_k h(r)=\int_r^1h(r)\dd r$, this means that both~$\A_k$ and~$\B_k$ are ``integrals of order~$\frac12$''.\footnote{They are pseudodifferential operators of order~$-\frac12$. The composition of these operators treats order as one might expect, so the composition differentiates by~$-1$ order --- it is an integral. In fact, the X-ray transform itself is a pseudodifferential operator of order~$-\frac12$. We will not dive into this theory here, but a curious reader is invited to do so. This is a matter of microlocal analysis, and we will touch on it very lightly in sections~\ref{sec:stab-sing} and~\ref{sec:loc-reconstruction}.}

Having Chebyshev polynomials in~$\A_k$ is not important at all for injectivity.
It does help with finding an explicit inversion formula, but similar injectivity results are true in far more generality.
The important things are the limits of integration and the kind of singularity at the lower limit.

\subsection{Injectivity of the X-ray transform}

We have now collected the needed tools, and it remains to declare the result.

\begin{theorem}
\label{xrtthm:cormack}
A continuous function $\pud\to\C$ is uniquely determined by its integrals over all straight lines.
\end{theorem}

\begin{ex}
Summarize the proof of the theorem in your own words.
Refer to the key steps (equations, lemmas, exercises, or other).
\end{ex}

Observe that no regularity assumption was made at the origin.
Singularities at the origin do not matter.

\begin{bex}
One can also go the other way around.
Assume that theorem~\ref{xrtthm:cormack} is true (this has been proved with other methods in section~\ref{sec:torus}).
Use the tools developed in this and the previous section to prove that the generalized Abel transforms~$\A_k$ are injective.

Opportunities like this arise often when we have multiple methods of proof.
\end{bex}

\subsection{Helgason's support theorem}

In fact, even more is true than theorem~\ref{xrtthm:cormack}.

\begin{proposition}
\label{prop:helgason}
Let $R\in(0,1)$.
If a continuous function $f\colon\bar D\to\C$ integrates to zero over all lines with distance${}>R$ to the origin, then $f(x)=0$ when $\abs{x}>R$.
(See figure~\ref{fig:disk-helgason} for an illustration.)
\end{proposition}

\begin{proof}
By assumption $\xrt f(r,\theta)=0$ whenever $r>R$, so by taking the Fourier transform in~$\theta$ we get $\A_ka_k(s)=0$ for all $s>R$.

The inversion formula for the generalized Abel transform~$\A_k$ is also valid for this case:
If $h\colon(0,1]\to\C$ is continuous and $\A_kh(r)=0$ for all $r\in(R,1]$, then $h(s)=0$ for $s\in(R,1]$.
(This can be verified by inspecting the proof of our injectivity theorem for Abel transforms.)
Therefore $a_k(s)=0$ for all $s>R$ and $k\in\Z$, and so $f(s,\phi)=0$ for all $s>R$ and $\phi\in\R/2\pi\Z$, as claimed.
\end{proof}

\begin{figure}[t]
    \centering
    \includegraphics[scale=0.35,trim={3cm 0 3cm 0},clip,page=16]{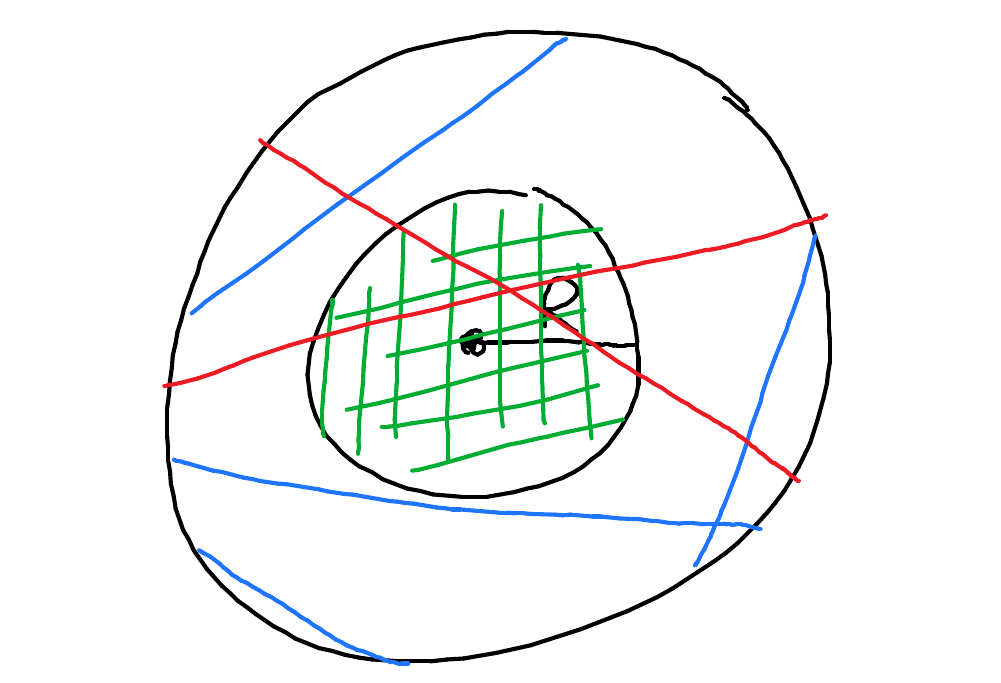}
    \caption{An illustration of proposition~\ref{prop:helgason}. \viher{The disk of radius~$R$} is an obstacle and \puna{lines passing through it are not included in the data}. We only have the integrals over \sini{lines that do not meet the obstacle} at our disposal.}
    \label{fig:disk-helgason}
\end{figure}

One way to interpret the situation is to see the Abel transform as an upper triangular matrix.
For an upper triangular matrix~$A$ the component~$(Av)_i$ only depends on the components~$v_j$ for $j\geq i$.
Examining the Abel transform reveals that~$\A_kh(s)$ only depends on the values of~$h(r)$ for $r\geq s$.
The inverse of an upper triangular matrix is upper triangular, and indeed our inversion formula has the same type of dependence.
A diagonal submatrix of an invertible upper triangular matrix is an invertible upper triangular matrix.

We may consider the disc~$\bar D(0,R)$ to be an obstacle.
A sufficiently nice function is uniquely determined outside the obstacle by its integrals over all lines that avoid the obstacle.
Of course, nothing can be said about the function inside the obstacle from this data.

Results of this kind are often called support theorems for the X-ray transform.
From a more physical point of view, this is a matter of exterior tomography --- there are actual physical obstacles in the real world that one cannot fire X-rays through.

One of the most famous support theorems is due to Sigur\dh{}ur Helgason.
We present a variant of the two-dimensional version.

\begin{theorem}[Helgason's support theorem in the plane]
\label{thm:helgason}
Let $K\subset\R^2$ be a compact and convex set.
Suppose $f\in C_c(\R^2)$ integrates to zero over all lines $L\subset\R^2$ for which $L\cap K=\emptyset$.
Then $f|_{\R^2\setminus K}=0$.
\end{theorem}

An alternative way to formulate the conclusion $f|_{\R^2\setminus K}=0$ is to say $\spt(f)\subset K$, whence the name ``support theorem''.
It is only natural that no conclusion is drawn in the set~$K$ because none of the lines available to us meet it.

\begin{ex}
\label{ex:convex-intersection}
Argue that a compact and convex planar set is the intersection of all closed discs containing it.
Figure~\ref{fig:helgason} might be helpful.
Then prove theorem~\ref{thm:helgason} using proposition~\ref{prop:helgason}.
(You may use the result that states that a compact convex set and a point outside it can be separated by a line which is disjoint from both the point and the set.)
\end{ex}

\begin{figure}[t]
    \centering
    \includegraphics[scale=0.35,trim={3cm 0 3cm 0},clip,page=15]{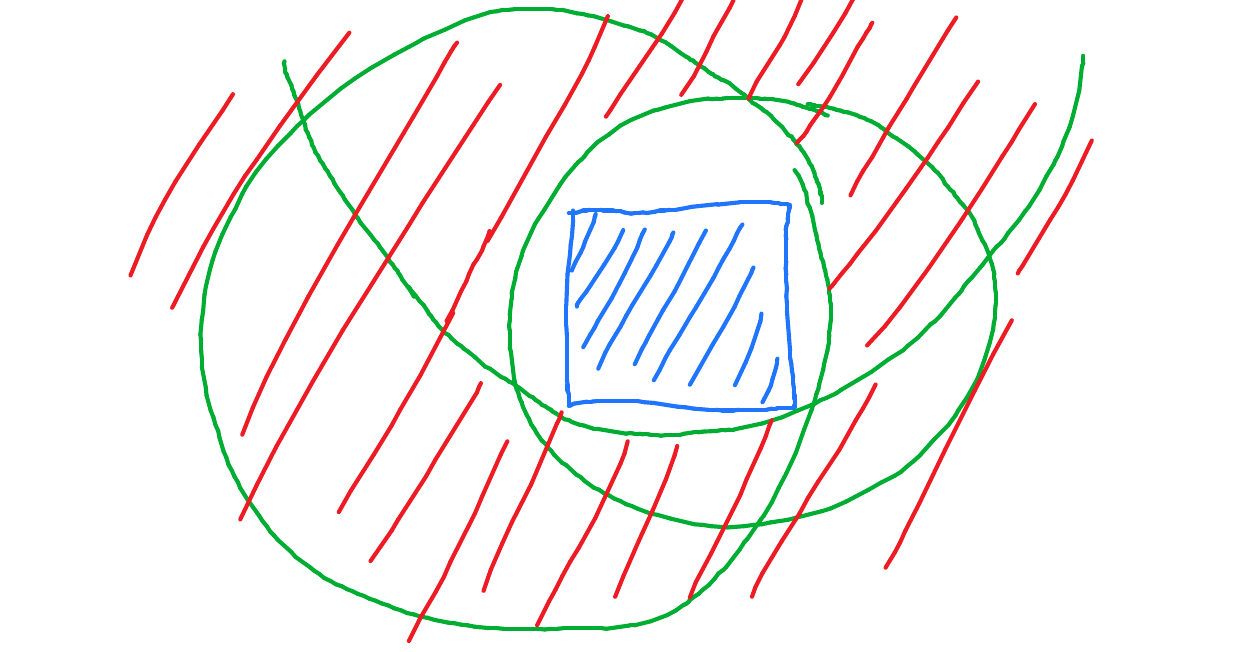}
    \caption{A \sini{compact and convex set} is the intersection of \viher{all disks containing it}. Therefore \puna{the union of the complements of the disks} is the complement of \sini{the original set}.}
    \label{fig:helgason}
\end{figure}

\begin{ex}
Explain why Helgason's support theorem (often) fails if the compact set~$K$ is not convex.
Also, what does the support theorem say if $K=\emptyset$?
\end{ex}

If~$K$ is not compact, the support theorem can fail.
For example, if~$K$ is a closed half plane, then the data only contains integrals parallel to~$\partial K$, which is certainly insufficient.

If a set is ``almost convex'', then Helgason's support theorem can still work.
For example, if~$K$ is the union of a closed ball and a point, the theorem is still valid as stated.
This is because, in some sense, a single point is removable; the missing lines can be approximated by the available ones.

\qa

\section{Inversion by circular averages}
\label{sec:radon}

In this section we will give our third injectivity proof by (an adaptation of) the historically first method due to Johann Radon in 1917.

\subsection{The X-ray transform and circular averages}

We will reconstruct a function in~$C_c(\R^2)$ from its line integral with Radon's method.
It is closely related to the previous one using the angular Fourier series as we shall see.
Our notation follows mainly that of Radon's original work, but we have made some adjustments.

The circular average of~$f$ over the circle centered at $x\in\R^2$ with radius $r>0$ is
\begin{equation}
\label{eq:r-fbar}
\bar f_x(r)
=
\frac1{2\pi}
\int_0^{2\pi}
f(x_1+r\cos(\theta),x_2+r\sin(\theta))\dd\theta.
\end{equation}

\begin{ex}
Fix any $r\in\R$ and $\theta\in\R/2\pi\Z=\T^1$.
Consider the curve $\gamma_{r,\theta}\colon\R\to\R^2$ given by
\begin{equation}
\gamma_{r,\theta}(t)
=
(r\cos(\theta)-t\sin(\theta),r\sin(\theta)+t\cos(\theta)).
\end{equation}
Show that its closest point to the origin is at $(r\cos(\theta),r\sin(\theta))$, that $\abs{\dot\gamma(t)}=1$ for all $t\in\R$, and that $\gamma(\R)=\{x\in\R^2;x_1\cos(\theta)+x_2\sin(\theta)=r\}$.

These are explicit parametrizations of the lines used in section~\ref{sec:ang-fs}, now in terms of Cartesian coordinates.

This exercise is related to figure~\ref{fig:psi-coordinates}.
\end{ex}

Using these~$r$ and~$\theta$ we can parametrize all the lines in the plane, including the ones going through the origin.
There is a two-fold redundancy as exercise~\ref{ex:r-flip} shows.

We define the X-ray transform of $f\in C_c(\R^2)$ as $\xrt f\colon\R\times\T^1\to\R$ given by the formula
\begin{equation}
\xrt f(r,\theta)
=
\int_\R f(\gamma_{r,\theta}(t))\dd t.
\end{equation}

\begin{ex}
Explain why~$\xrt f$ is bounded and continuous when $f\in C_c(\R^2)$.
This way of writing the X-ray transform is similar but not identical to the one of exercise~\ref{ex:xrt-def}.
\end{ex}

\begin{ex}
\label{ex:r-flip}
Show that $\xrt f(r,\theta)=\xrt f(-r,\theta+\pi)$.
How do the curves~$\gamma_{r,\theta}$ and~$\gamma_{-r,\theta+\pi}$ differ?
\end{ex}

We will also define a circular average of the X-ray transform~$\xrt f$.
The average over the circle with center $x\in\R^2$ and radius $r>0$ is defined to be
\begin{equation}
\overline{\xrt f}_x(r)
=
\frac1{2\pi}
\int_0^{2\pi}
\xrt f(x_1\cos(\theta)+x_2\sin(\theta)+r,\theta)\dd\theta.
\end{equation}
We will verify in exercise~\ref{ex:circle-tangent-integral} that this formula is geometrically correct; this is really the average over all the lines tangent to the said circle.

The two circular averages are illustrated in figure~\ref{fig:circular-averages}.

\begin{figure}[t]
    \centering
    \includegraphics[scale=0.4,trim={3cm 0 3cm 0},clip,page=17]{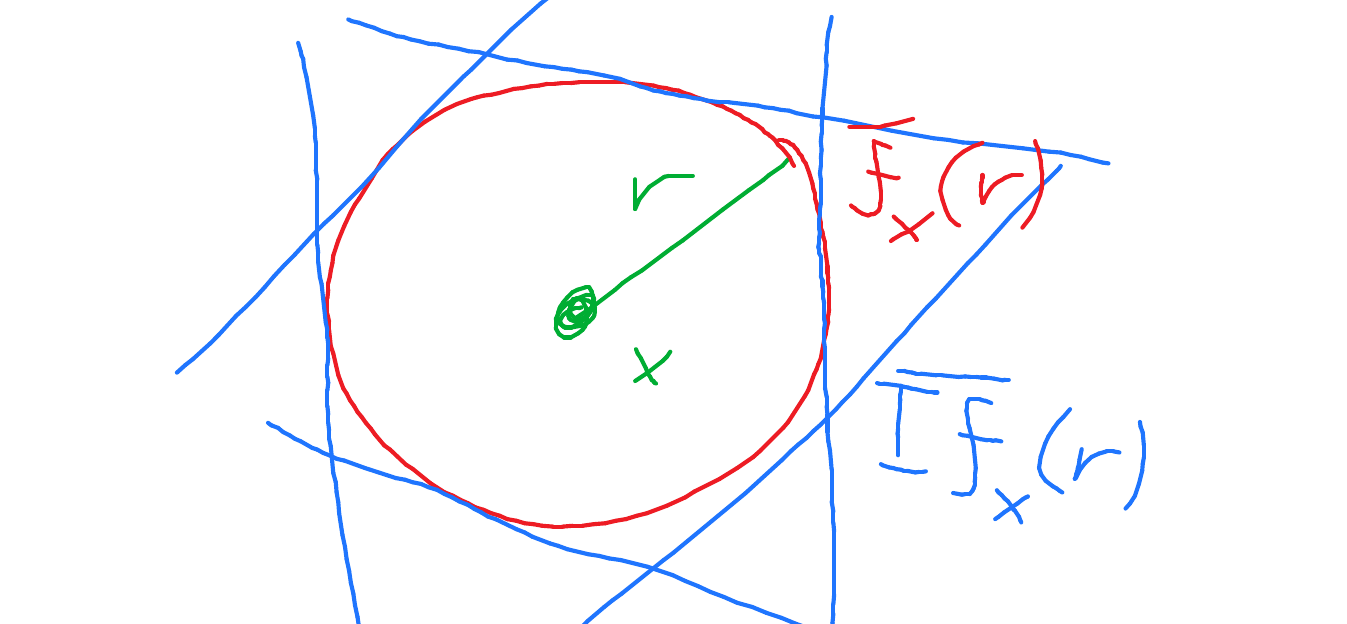}
    \caption{Two kinds of averages on the circle with \viher{center $x$ and radius $r$}: \puna{the average~$\bar f_x(r)$ of the (unknown) function~$f$} and \sini{the average~$\overline{\xrt f}_x(r)$ of the (known) data~$\xrt f$}.}
    \label{fig:circular-averages}
\end{figure}

\begin{ex}
\label{ex:circle-tangent-integral}
Consider the circle with center~$x$ and radius~$r$.
Take any angle $\theta\in\R/2\pi\Z$.
Consider the point~$z$ on the circle where the exterior unit normal vector is $(\cos(\theta),\sin(\theta))$.
Let~$L$ be the line tangent to the circle at~$z$.
Show that the integral of~$f$ over~$L$ is
\begin{equation}
\xrt f(x_1\cos(\theta)+x_2\sin(\theta)+r,\theta).
\end{equation}
Draw a picture to illustrate the situation.
\end{ex}

We will reconstruct~$f$ from~$\xrt f$ via~$\overline{\xrt f}$.

\subsection{Reduction to the Abel transform}

The key of the proof is the following integral identity.

\begin{lemma}
\label{lma:radon-identity}
If $f\in C_c(\R^2)$, $x\in\R^2$, and $r>0$, then
\begin{equation}
\overline{\xrt f}_x(r)
=
2\int_r^\infty\frac{\bar f_x(s)s}{\sqrt{s^2-r^2}}\dd s.
\end{equation}
\end{lemma}

\begin{proof}
By translation invariance we may assume that $x=0$ (exercise~\ref{ex:radon-translation-invariance}).
Consider $r>0$ fixed.

Define $\psi_r\colon[0,\infty)\times\T^1\to\R^2\setminus D(0,r)$ by
\begin{equation}
\label{eq:r-psi}
\psi_r(t,\theta)
=
(r\cos(\theta)-t\sin(\theta),r\sin(\theta)+t\cos(\theta))
=
\gamma_{r,\theta}(t).
\end{equation}
This is a diffeomorphism and the Jacobian determinant is simply~$t$ (exercise~\ref{ex:radon-jacobian}).
The map is illustrated in figure~\ref{fig:psi-coordinates}.

\begin{figure}[t]
    \centering
    \includegraphics[scale=0.5,trim={3cm 0 3cm 0},clip,page=18]{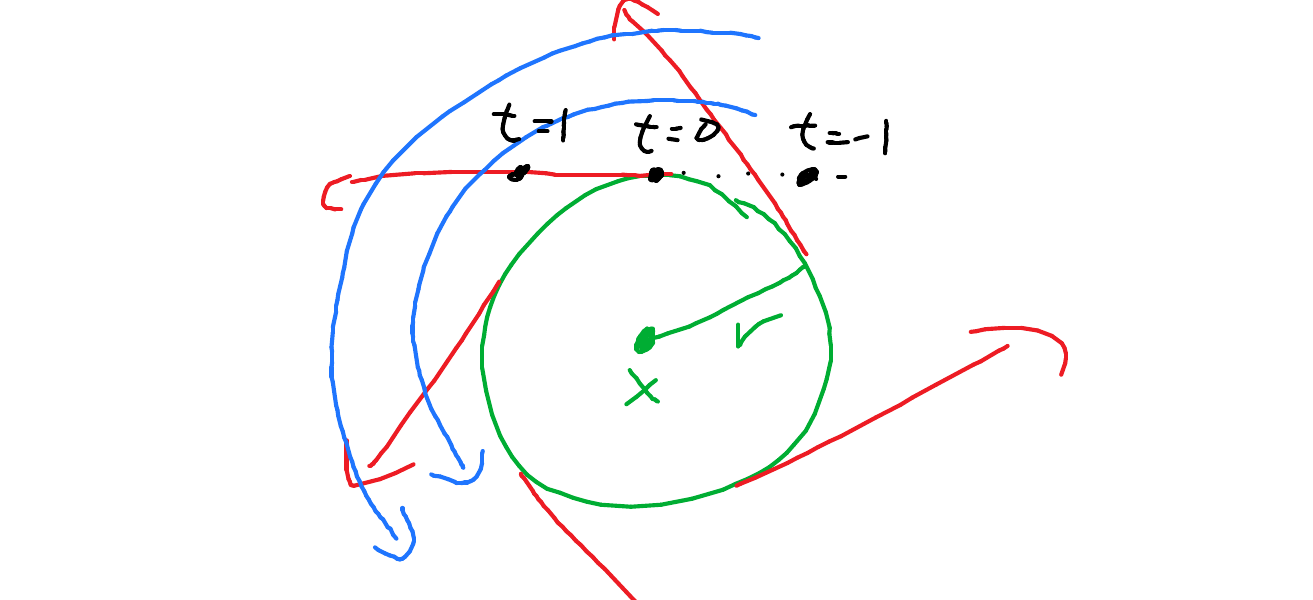}
    \caption{The map~$\psi_r$ of~\eqref{eq:r-psi} maps to the complement of \viher{the disk with center~$x$ and radius~$r$}. The \puna{curves of constant $\theta$ are half lines tangent to the circle}. This corresponds to $t\geq0$; the map is well defined for $t<0$ as well, but then the lines intersect and the map fails to be bijective. The \sini{curves of constant $t$ are circles centered at~$x$ with radius $\sqrt{r^2+t^2}$}. Bijectivity of the map is (hopefully) visually clear from the picture, and the reader gnawed by doubt may verify it algebraically.}
    \label{fig:psi-coordinates}
\end{figure}

A computation gives
\begin{equation}
\label{eq:r-id-calc}
\begin{split}
\overline{\xrt f}_0(r)
&\stackrel{\text{a}}{=}
\frac1{2\pi}
\int_0^{2\pi}
\xrt f(r,\theta)\dd\theta
\\&\stackrel{\text{b}}{=}
\frac1{2\pi}
\int_0^{2\pi}
\int_\R
f(r\cos(\theta)-t\sin(\theta),r\sin(\theta)+t\cos(\theta))
\dd t
\dd\theta
\\&\stackrel{\text{c}}{=}
\frac1{\pi}
\int_{\T^1}
\int_0^\infty
f(r\cos(\theta)-t\sin(\theta),r\sin(\theta)+t\cos(\theta))
t^{-1}
t
\dd t
\dd\theta
\\&\stackrel{\text{d}}{=}
\frac1{\pi}
\int_{[0,\infty)\times\T^1}
f(\psi_r(t,\theta))
\left(\abs{\psi_r(t,\theta)}^2-r^2\right)^{-1/2}
t
\dd t
\dd\theta
\\&\stackrel{\text{e}}{=}
\frac1{\pi}
\int_{\R^2\setminus D(0,r)}
f(x)
\left(\abs{x}^2-r^2\right)^{-1/2}
\dd x
\\&\stackrel{\text{f}}{=}
\frac1{\pi}
\int_r^\infty
\left(
\int_0^{2\pi}
\frac{f(s\cos(\theta),s\sin(\theta))}{\sqrt{s^2-r^2}}
\dd\theta
\right)
s\dd s
\\&\stackrel{\text{g}}{=}
2
\int_r^\infty
\frac{\bar f_0(s)}{\sqrt{s^2-r^2}}
s\dd s.
\end{split}
\end{equation}
This is the claimed identity for $x=0$.
\end{proof}

\begin{ex}
\label{ex:radon-translation-invariance}
Fix any $a\in\R^n$ and denote by $T_a\colon C_c(\R^2)\to C_c(\R^2)$ the translation operator defined by $T_af(x)=f(x+a)$.
Explain geometrically why
\begin{equation}
\bar f_x(r)
=
\overline{T_xf}_0(r)
\end{equation}
and
\begin{equation}
\overline{\xrt f}_x(r)
=
\overline{\xrt T_xf}_0(r).
\end{equation}
This means that the statement of lemma~\ref{lma:radon-identity} can be formulated in terms of shifted functions while keeping all the circles centered at the origin.
\end{ex}

\begin{ex}
\label{ex:radon-jacobian}
Show that the Jacobian determinant $\det(D\psi_r(t,\theta))$ of~$\psi_r$ is~$t$.
\end{ex}

\begin{ex}
Explain briefly what happened in the steps a--g of~\eqref{eq:r-id-calc}.
\end{ex}

We define the Abel transform of a compactly supported continuous function $h\colon[0,\infty)\to\R$ to be $\A h\colon(0,\infty)\to\R$ given by
\begin{equation}
\A h(r)
=
2\int_r^\infty \frac{h(s)s}{\sqrt{s^2-r^2}}\dd s.
\end{equation}
With the help of this notation we can rewrite lemma~\ref{lma:radon-identity} as
\begin{equation}
\overline{\xrt f}_x(r)
=
\A\bar f_x(r).
\end{equation}
Comparing to~\eqref{eq:Ak-def}, we see that in fact $\A=\A_0$ apart from the upper limit of integration.
As long as this limit is finite --- as it is due to the compact support of~$h$ --- we may use the same inversion formula~\eqref{eq:abel-inv} to invert~$\A$.
We record the relevant result here without proof; the proof is the same as that of lemma~\ref{lma:abel}.

\begin{lemma}
\label{lma:abel-infty}
The generalized Abel transform $\A_k\colon C_c((0,\infty))\to C_c((0,\infty))$ is an injection.
Moreover, $h\in C_c((0,\infty))$ can be calculated from~$\A_k h$ via
\begin{equation}
h(r)
=
-\frac1\pi\frac{\der}{\der r}
\int_r^\infty\A_kh(s)\frac{T_k(s/r)}{s\sqrt{(s/r)^2-1}}\dd s.
\end{equation}
\end{lemma}

We only proved injectivity of~$\A_k$ for $k=0$ and $k=\pm1$ by hand, and here we only need the special case $k=0$.
We have
\begin{equation}
\label{eq:radon-inv-xr}
\bar f_x(r)
=
-\frac1\pi\frac{\der}{\der r}
\int_r^\infty\frac{\overline{\xrt f}_x(s)}{s\sqrt{(s/r)^2-1}}\dd s.
\end{equation}
for all $x\in\R^2$ and $r>0$.

\begin{ex}
Prove that $\lim_{r\to0}\bar f_x(r)=f(x)$ when $f\colon\R^2\to\R$ is continuous.
\end{ex}

This little observation together with the identity~\eqref{eq:radon-inv-xr} shows that~$\xrt f$ determines~$f(x)$ for all~$x$ and therefore proves the desired injectivity result.
However, the formula becomes more useful when we actually calculate the limit, and this we shall do next.

\subsection{An explicit inversion formula}

To find the explicit inversion formula without requiring too many tools, we make the additional assumption that $f\in C^2_c(\R^2)$.
Fix any point $x\in\R^2$.
Let us denote $F(r)=\overline{\xrt f}_x(r)$.
It follows from this regularity assumption that $F\in C^2_c(\R)$.
Notice that the same formula can be used to define~$F(r)$ for all $r\in\R$, not only $r>0$.

If we denote
\begin{equation}
G(r)
=
\int_r^\infty\frac{F(s)}{s\sqrt{(s/r)^2-1}}\dd s,
\end{equation}
then the problem is to find
\begin{equation}
f(x)=-\frac1\pi\lim_{r\to0}G'(r).
\end{equation}
Let us first find a new formula for~$G'(r)$ when $r>0$.

If all things are regular enough, the derivative of
\begin{equation}
\Phi(x)
=
\int_x^\infty K(x,y)\phi(y)\dd y
\end{equation}
is
\begin{equation}
\Phi'(x)
=
-K(x,x)\phi(x)
+\int_x^\infty \partial_1K(x,y)\phi(y)\dd y.
\end{equation}
But in the case of the Abel transform and our function~$G(r)$ neither term makes sense: the integral diverges and the diagonal values~$K(x,x)$ look infinite.
A more elaborate differentiation scheme is needed, and these two singularities will eventually cancel each other.

After the change of variable from $s\in(r,\infty)$ to $z=\sqrt{s^2-r^2}\in(0,\infty)$, we get
\begin{equation}
\label{eq:r-vv1}
G(r)
=
\int_0^\infty\frac{rF(\sqrt{z^2+r^2})}{z^2+r^2}\dd z.
\end{equation}
The limits no longer depend on~$r$, and that helps against the trouble outlined in the previous paragraph.

\begin{ex}
Make this change of variable and verify the formula above for~$G(r)$.
\end{ex}

When $r>0$, we may easily differentiate under the integral sign, and we obtain
\begin{equation}
\label{eq:r-vv2}
\begin{split}
G'(r)
&=
\int_0^\infty
\Bigg(
F'(\sqrt{z^2+r^2})\frac{r^2}{(z^2+r^2)^{3/2}}
\\&\qquad
+
F(\sqrt{z^2+r^2})\frac{z^2-r^2}{(z^2+r^2)^{2}}
\Bigg)\dd z.
\end{split}
\end{equation}
Integrating by parts in the second term gives
\begin{equation}
\label{eq:r-vv3}
G'(r)
=
\int_0^\infty
\frac{F'(\sqrt{z^2+r^2})}{\sqrt{z^2+r^2}}
\dd z.
\end{equation}

\begin{ex}
Justify the steps from~\eqref{eq:r-vv1} to~\eqref{eq:r-vv2} and~\eqref{eq:r-vv3}.
\end{ex}

To study the limit, we need some regularity estimates.
Due to the symmetry property (see exercise~\ref{ex:r-flip}) of $\xrt f(r,\theta)$, we have $F(r)=F(-r)$.
Since $F\in C^2_c(\R)$, it then follows that $F'(0)=0$ and
\begin{equation}
\abs{F'(x)-F'(y)}
\leq
C\abs{x-y}
\end{equation}
for some constant~$C$.
We may in fact choose simply $C=\max_{x\in\R}\abs{F''(x)}$.
By letting $y=0$ it follows that $\abs{F'(x)}\leq C\abs{x}$.
These observations will help us study the integral in~\eqref{eq:r-vv3}.

\begin{ex}
Show that if $F\colon\R\to\R$ satisfies $F(x)=F(-x)$ for all $x\in\R$ and is differentiable at the origin, then $F'(0)=0$.
\end{ex}

The natural guess is that the limit $\lim_{r\to0}G'(r)$ would be
\begin{equation}
L
=
\int_0^\infty
\frac{F'(z)}{z}
\dd z.
\end{equation}
Notice that since $\abs{F'(z)}\leq C\abs{z}$ and~$F'$ is compactly supported and continuous, the integral~$L$ exists.
We have
\begin{equation}
\begin{split}
\abs{G'(r)-L}
&=
\abs{
\int_0^\infty
\left(
\frac{F'(\sqrt{z^2+r^2})}{\sqrt{z^2+r^2}}
-
\frac{F'(z)}{z}
\right)
\dd z
}
\\&=
\Bigg\lvert
\int_0^\infty
\Bigg(
\frac{F'(\sqrt{z^2+r^2})-F'(z)}{\sqrt{z^2+r^2}}
\\&\qquad+
F'(z)
\left(\frac1{\sqrt{z^2+r^2}}-\frac1z\right)
\Bigg)
\dd z
\Bigg\rvert
\\&\leq
\int_0^\infty
\Bigg(
\frac{\abs{F'(\sqrt{z^2+r^2})-F'(z)}}{\sqrt{z^2+r^2}}
\\&\qquad+
\abs{F'(z)}
\abs{\frac{z-\sqrt{z^2+r^2}}{z\sqrt{z^2+r^2}}}
\Bigg)
\dd z
\\&\leq
\int_0^\infty
\left(
\frac{C\abs{\sqrt{z^2+r^2}-z}}{\sqrt{z^2+r^2}}
+
Cz
\frac{\sqrt{z^2+r^2}-z}{z\sqrt{z^2+r^2}}
\right)
\dd z
\\&=
2C
\int_0^\infty
\frac{\sqrt{z^2+r^2}-z}{\sqrt{z^2+r^2}}
\dd z
\\&=
2C
\left[
z-\sqrt{z^2+r^2}
\right]_0^\infty
\\&=
2Cr.
\end{split}
\end{equation}
Therefore we have proved that $G'(r)\to L$ as $r\to0$.

Let us collect our findings into a theorem:

\begin{theorem}
\label{xrtthm:radon}
A function $f\in C_c(\R^2)$ is uniquely determined by its X-ray transform.
Moreover, if $f\in C^2_c(\R^2)$, it can be reconstructed pointwise by
\begin{equation}
f(x)
=
-\frac1\pi
\int_0^\infty
r^{-1}\frac{\der}{\der r}\overline{\xrt f}_x(r)
\dd r.
\end{equation}
\end{theorem}

\begin{ex}
Summarize the proof of theorem~\ref{xrtthm:radon}.
\end{ex}

The reconstruction formula can also be written as a Stieltjes integral like Radon did:
\begin{equation}
f(x)
=
-\frac1\pi
\int_0^\infty
r^{-1}\der\overline{\xrt f}_x(r)
.
\end{equation}

\subsection{Relation to the angular Fourier series}

Let us now see how the methods of section~\ref{sec:ang-fs} are related to the idea of this section.
By translation invariance it suffices to show that~$\xrt f$ uniquely determines~$f(0)$.
We write the function $f(r,\theta)$ as a Fourier series in~$\theta$:
\begin{equation}
f(r,\theta)
=
\sum_{k\in\Z}a_k(r)e^{ik\theta}.
\end{equation}
We also write the X-ray transform as a Fourier series:
\begin{equation}
\xrt f(r,\theta)
=
\sum_{k\in\Z}b_k(r)e^{ik\theta}.
\end{equation}
As discussed in section~\ref{sec:ang-fs} the function~$b_k$ only depends on the function~$a_k$, and we found in section~\ref{sec:abel} that $b_k=\A_ka_k$.

Let us look at $k=0$.
Now $a_0(r)=\bar f_0(r)$ and $b_0(r)=\overline{\xrt f}_0(r)$.
It should be noted that~$a_0$ and~$b_0$ are indexed by $0\in\Z$ but~$\bar f_0(r)$ and~$\overline{\xrt f}_0(r)$ by $0\in\R^2$.
The two observations $b_0=\A_0a_0$ and $\overline{\xrt f}_0(r)=\A\bar f_0(r)$ are therefore the same.
The function~$a_0$ can be reconstructed from~$b_0$ by inverting the Abel transform.
Now $f(0)=\lim_{r\to0}a_0(r)$, which gives a reconstruction formula at the origin.

In conclusion, it is enough to look at the zeroth component of the angular Fourier series if one varies the origin of the polar coordinates.
The reconstruction works at any chosen origin.
The zeroth component of the angular Fourier series is nothing but the circular average.

\qa

\section{The geometry of Euclidean geodesics}
\label{sec:geod-geom}

In this section and the next two sections we will give our fourth injectivity proof based on analysis on the sphere bundle.

\subsection{The sphere bundle}

As silly as it sounds, we will now study the geometry of straight lines in the Euclidean space~$\R^n$.
These same geometrical ideas will remain valid and applicable on Riemannian manifolds.
We will restrict our attention to Euclidean spaces, but some differential geometric ideas will be involved, and we will have to consider a certain non-Euclidean space.

In this section we will study straight lines as curves parametrized by arc length.
This is most conveniently done on the sphere bundle
\begin{equation}
S\R^n
=
\R^n\times \Sphere^{n-1}
\end{equation}
of the Euclidean space~$\R^n$.
This is a bundle over~$\R^n$ and comes with the natural projection $\pi\colon S\R^n\to\R^n$ to the first component.
In the Euclidean setting the bundle is simply a product of two spaces and the bundle is trivial.
In a non-Euclidean situation the bundle structure becomes more complicated.

A point $(x,v)\in S\R^n$ describes a point and a velocity at that point.
The velocity variable~$v$ takes values in the fiber~$\Sphere^{n-1}$ of the bundle.
Parametrization by unit speed instead of arbitrary speed is very convenient, as it makes the fibers compact.
The usefulness becomes apparent when we integrate over a sphere bundle later on.

There are two kinds of directions on~$S\R^n$.
Directions in the~$\R^n$ component are called horizontal and those in~$\Sphere^{n-1}$ are called vertical.
This terminology will reappear in the next section when we consider horizontal and vertical derivatives of a function on the sphere bundle.
This choice of words correspond to the canonical way of drawing the base~$\R^n$ of the bundle horizontally and the fibers~$\Sphere^{n-1}$ vertically, as shown in figure~\ref{fig:bundle}.

\begin{figure}[t]
    \centering
    \includegraphics[scale=0.4,trim={3cm 0 3cm 0},clip,page=19]{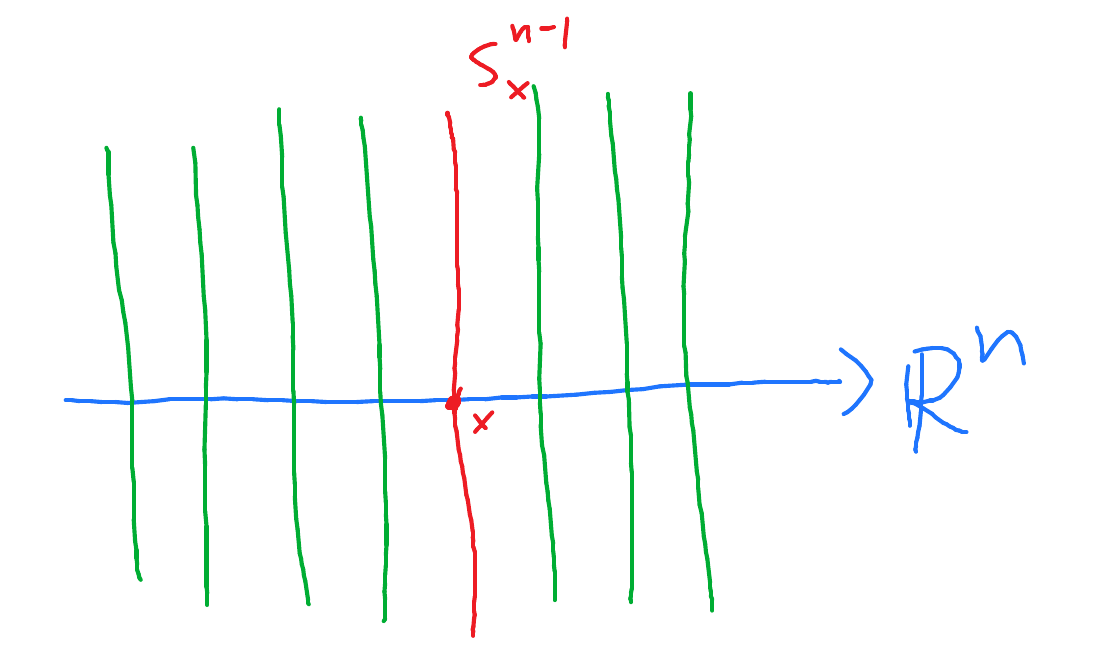}
    \caption{The sphere bundle~$S\R^n$ is the union of \viher{fibers} attached to each point on \sini{the base~$\R^n$}. Each fiber is a copy of~$\Sphere^{n-1}$, and we may call \puna{the sphere at~$x$ by the name~$\Sphere^{n-1}_x$} to be consistent with how bundles are usually denoted in differential geometry. It makes little difference in Euclidean geometry, though, as all fibers are very literally the same. Note that \sini{the base} is drawn horizontally and \viher{the fibers} are drawn vertically --- this will be important for nomenclature.}
    \label{fig:bundle}
\end{figure}

\subsection{The geodesic flow}

Straight lines are geodesics.
There is a dynamical system associated with geodesics, and we will examine it next.

\begin{definition}
\label{def:dyn-sys}
A continuous time dynamical system on a set~$Z$ is a function $\phi\colon\R\times Z\to Z$ which satisfies $\phi_0(z)=z$ and $\phi_s(\phi_t(z))=\phi_{s+t}(z)$ for all $z\in Z$ and $s,t\in\R$.
\end{definition}

A dynamical system describes the time evolution of a point in the phase space~$Z$.
Every $z\in Z$ has a unique trajectory or integral curve $t\mapsto\phi_t(z)$.
It is also possible to view a dynamical system more algebraically, as the action of the additive group~$\R$ on the set~$Z$.

\begin{ex}
\label{ex:ds1}
Which of the following are dynamical systems on~$\R$?
\begin{enumerate}[(a)]
\item $\phi_t(z)=z+3t$.
\item $\phi_t(z)=4$.
\item $\phi_t(z)=z^t$.
\item $\phi_t(z)=z-t^2$.
\item $\phi_t(z)=e^{-2t}z$.
\item $\phi_t(z)=z+tz$.
\end{enumerate}
Explain briefly.
\end{ex}

The geodesic flow is a dynamical system on the sphere bundle.
It is simply given by
\begin{equation}
\phi_t(x,v)
=
(x+tv,v).
\end{equation}

The geodesic flow could be equally well defined on any bundle $\R^n\times A$ for $A\subset\R^n$ with the same formula.
The most natural choices are $A=\R^n$ (all velocities possible) and $A=\Sphere^{n-1}$ (unit speed geodesics).
This is what we meant above by saying that unit speed geodesics make fibers compact.

As in the case of geodesics, dynamical systems are often studied on manifolds.
Then there is a vector field~$W$ on the manifold~$Z$ so that for any initial point $z\in Z$ the function $f(t)=\phi_t(z)$ solves the differential equation $f'(t)=W(f(t))$.
Such a vector field is called the generator of the flow, and an illustration is attempted in figure~\ref{fig:flow}.
Whenever~$\phi$ is smooth enough, the generator exists and can be computed by differentiating the flow with respect to~$t$ --- that is, $W(z)=\partial_t\phi_t(z)|_{t=0}$.
One can also impose much more structure on a flow but it will not be necessary for us here.
We only remark that the geodesic flow can be seen as a contact flow or a Hamiltonian flow.

\begin{figure}[t]
    \centering
    \includegraphics[scale=0.45,trim={3cm 0 3cm 0},clip,page=20]{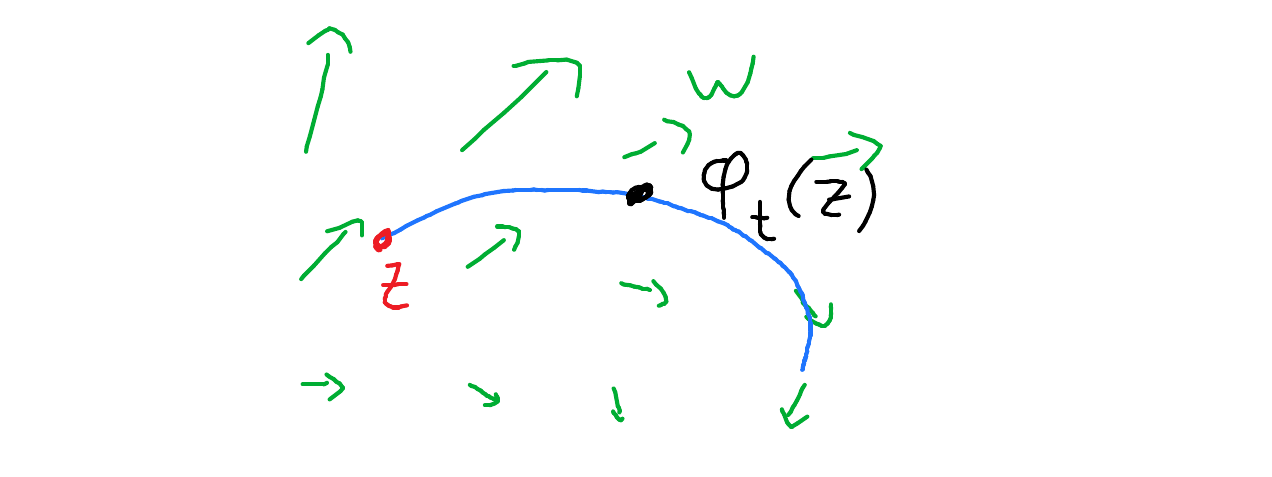}
    \caption{\viher{A vector field~$W$} defines a flow. \puna{A point~$z$} flows along the field, and \sini{its integral curve or trajectory} is always tangent to \viher{the generating vector field}. For $t>0$ the point~$\phi_t(z)$ is down the flow from~$z$.}
    \label{fig:flow}
\end{figure}

\begin{ex}
On the real line~$\R$ a vector field can be considered to be just a function $\R\to\R$.
Go back to the dynamical systems of exercise~\ref{ex:ds1}.
What are their generators?
\end{ex}

The generator of the geodesic flow is called the geodesic vector field, and it is denoted by~$X$.
(For a general flow we denoted the generator by~$W$.)
It is typical in differential geometry to identify a vector field with the associated differential operator.
For example, a vector field $w\colon\R^n\to\R^n$ is identified with the differential operator $f\mapsto w\cdot\nabla f$ which maps scalar functions to scalar functions.
For us the geodesic vector field is just a differential operator, but we still call it a vector field to follow standard terminology.

The differential operator corresponding to the generator of the flow is the derivative along the flow.
Consider a function $u\colon S\R^n\to\R$.
The geodesic vector field is defined to be
\begin{equation}
Xu(x,v)
=
\partial_t u(\phi_t(x,v))|_{t=0}.
\end{equation}

\begin{ex}
Using the definition of the geodesic flow, find a formula for the geodesic vector field.
For a function $u\colon\R^n\times \Sphere^{n-1}\to\R$, let us denote the gradient with respect to the first component by~$\nabla_x u$.
\end{ex}

If we were to write~$X$ as a vector field instead of a differential operator, it would be $X(x,v)=(v,0)$.
The second component is the zero vector field on~$\Sphere^{n-1}$.
The fact that the second component vanishes means that~$X$ is horizontal.

\begin{ex}
Suppose $u(x,v)=x\cdot v$.
What is $Xu(x,v)$?
\end{ex}

A geodesic on~$\R^n$ is the projection of a trajectory of the geodesic flow.
Trajectories are of the form $t\mapsto\phi_t(x,v)=(x+tv,v)$, and geodesics are of the form $t\mapsto\pi(\phi_t(x,v))=x+tv$.

This process can also be reversed.
If $\gamma\colon\R\to\R^n$ is a differentiable unit speed curve, we can define its lift $\tilde\gamma\colon\R\to S\R^n$ by $\tilde\gamma(t)=(\gamma(t),\dot\gamma(t))$.
The lift of a geodesic is a trajectory of the geodesic flow.
Lifting is illustrated in figure~\ref{fig:lift}.

\begin{figure}[t]
    \centering
    \includegraphics[scale=0.5,trim={3cm 0 3cm 0},clip,page=21]{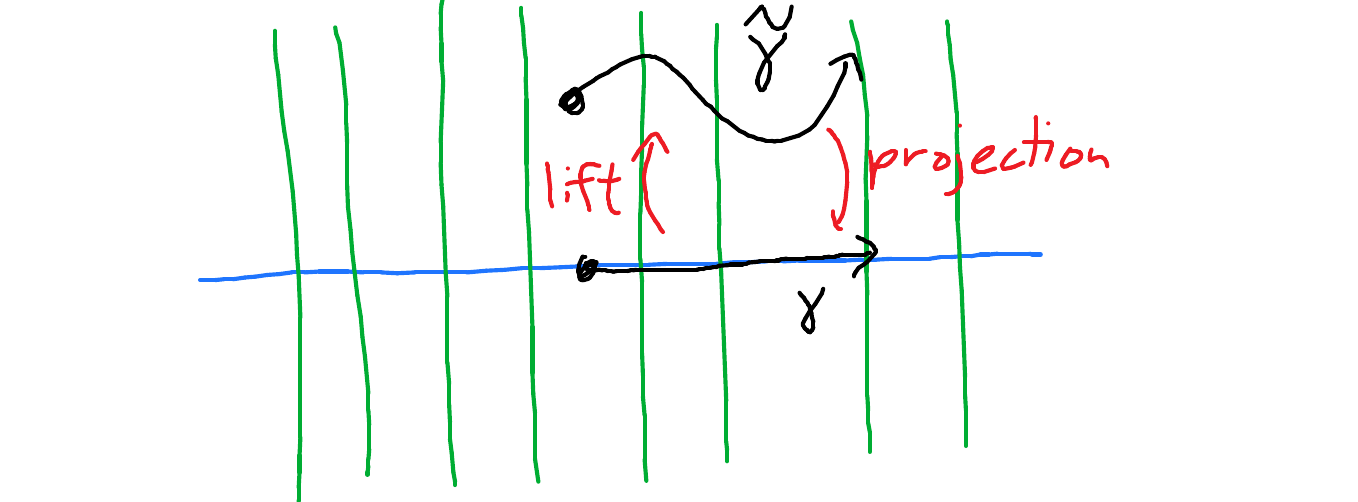}
    \caption{A geodesic~$\gamma$ is a curve on \sini{the base}, whereas its lift~$\tilde\gamma$ is a curve on the bundle. \puna{Projecting} along \viher{the fibers} brings~$\tilde\gamma$ back to~$\gamma$, undoing the \puna{lift}.}
    \label{fig:lift}
\end{figure}

\subsection{The manifold of geodesics}

Previously we discussed the set~$\Gamma$ of all straight lines in~$\R^n$.
We can describe its structure a little more now.

The geodesic flow gives rise to an equivalence relation on the sphere bundle, where~$(x,v)$ is considered equivalent to~$(x',v')$ if and only if $\phi_t(x,v)=(x',v')$ for some $t\in\R$.
We can form the quotient space of the sphere bundle with this relation, and we denote it by $S\R^n/\phi$.

This quotient space has the structure of a smooth Riemannian manifold.
In general, quotient manifolds are somewhat ill-behaved, but this particular quotient does make sense.
For geodesics on a general manifold this is no longer the case.
However, the quotient is always sensible as a topological space, but the resulting space can be wild.

This $S\R^n/\phi$ is the space of all oriented lines.
We previously denoted it by~$\Gamma$.
To get the set of all unoriented lines, one has to take another quotient to identify opposite orientations.
This second quotient is well behaved.

\begin{bex}
Consider the geodesic flow on the sphere bundle of the torus~$\T^2$.
The flow is well defined since the manifold is geodesically complete.
The sphere bundle $S\T^2=\T^2\times \Sphere^1$ is a topological space and the quotient with respect to any equivalence relation can be given the quotient topology.
Show that the topological quotient space $S\T^2/\phi$ is not Hausdorff.
Does the same apply in~$\T^n$ in any dimension~$n$?
\end{bex}

In the Euclidean space~$\R^n$ the space of lines has an explicit description.
Let us first do the oriented lines.
We first pick a direction $v\in\Sphere^{n-1}$.
All lines with this direction are described by a point on a plane orthogonal to~$v$.
As illustrated in figure~\ref{fig:line-bundle-fiber}, these lines therefore constitute a fiber~$T_v\Sphere^{n-1}$ of the tangent bundle of the unit sphere.
Therefore the whole set of oriented lines is simply the tangent bundle~$T\Sphere^{n-1}$ of the unit sphere.
If we want unoriented lines, we quotient the sphere by the antipodal map $v\mapsto-v$, leaving us with the tangent bundle of the projective space.

\begin{figure}[t]
    \centering
    \includegraphics[scale=0.35,trim={3cm 0 3cm 0},clip]{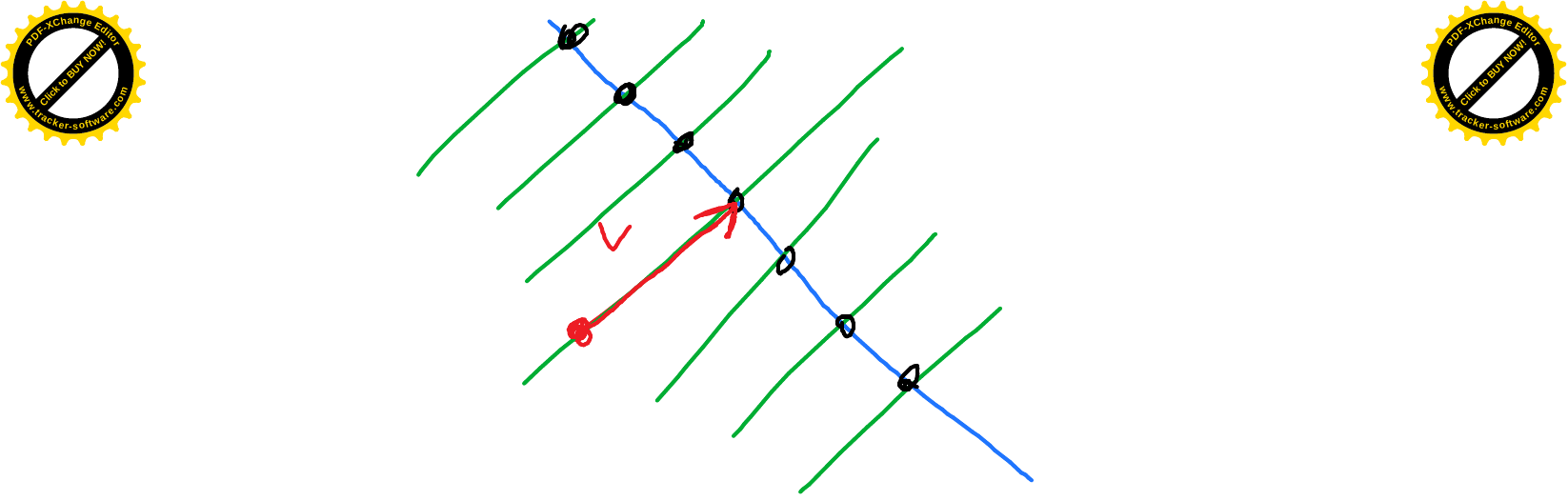}
    \caption{Given a \puna{direction $v\in\Sphere^{n-1}$}, \viher{all lines with the same direction} can be identified with points on \sini{a plane normal to $v$}. If we choose the plane to go through~$v$, this corresponds visually to the fiber~$T_v\Sphere^{n-1}$ of the tangent bundle of the unit sphere.}
    \label{fig:line-bundle-fiber}
\end{figure}

\subsection{Bounded domains and integral functions}

It will be convenient to consider geodesics on bounded sets.
For a bounded set $\Omega\subset\R^n$ we define the unit sphere bundle~$S\Omega$ as $\Omega\times \Sphere^{n-1}$.
This causes a technical inconvenience for the geodesic flow on~$S\Omega$, since it is not defined for all times.

One can define a local dynamical system on a space~$Z$ so that for each $z\in Z$ the flow~$\phi_t(z)$ is defined for times~$t$ in some interval $J_z\subset\R$.
For a simple example, consider $Z=[0,1)$ and $\phi_t(z)=z+t$ whenever it is well defined.
In this case $J_z=[-z,1-z)$.
For $z\in(0,1)$ the interval~$J_z$ contains a neighborhood of $0\in\R$ and the dynamical system behaves locally just as well as the usual kind of a dynamical system, but one just cannot go arbitrarily far in time.
Things are a little more complicated at $z=0$, where one can only follow the flow to the positive direction.
The geodesic flow on a bounded domain is essentially of this type, with only one direction available at~$\partial(S\Omega)$.
At a tangential boundary point the flow is stuck; it cannot move in either direction.

Now, let $\Omega\subset\R^n$ be a smooth and strictly convex domain and~$\bar\Omega$ its closure.

For $(x,v)\in S\Omega$, let $\tau(x,v)=\max\{t>0;\phi_t(x,v)\in S\bar\Omega\}$.
This is the time it takes for a geodesic starting at $(x,v)$ to escape~$\Omega$, as illustrated in figure~\ref{fig:tau}.

\begin{figure}[t]
    \centering
    \includegraphics[scale=0.4,trim={3cm 0 3cm 0},clip,page=23]{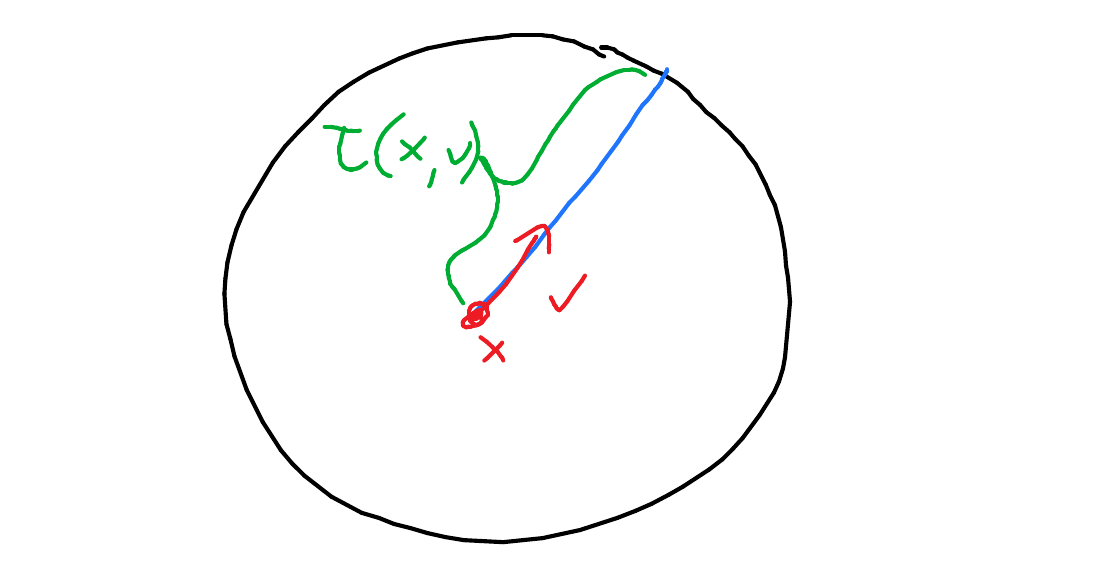}
    \caption{\viher{The travel time~$\tau(x,v)$} is the length of \sini{the forward-maximal geodesic} starting at \puna{the point~$x$ in direction~$v$}. The function~$u^f$ will be defined by integrating over \sini{these half geodesics}.}
    \label{fig:tau}
\end{figure}

The boundary of the sphere bundle~$S\Omega$ is $\partial(S\Omega)=\partial\Omega\times \Sphere^{n-1}$.
Observe that $\partial(S\Omega)=S\bar\Omega\setminus S\Omega$.
For $x\in\partial\Omega$, let~$\nu(x)$ denote the outer unit normal vector to~$\partial\Omega$.
A point $(x,v)\in\partial(S\Omega)$ is called an inward boundary point if $v\cdot\nu(x)<0$.
Similarly, the outward part of the boundary consists of points in~$\partial(S\Omega)$ where the inner product is positive and the tangential part of the points where it is zero.
Let us denote the inward boundary by $\inwb\subset\partial(S\Omega)$.

\begin{ex}
The definition of~$\tau(x,v)$ can be naturally extended to all $(x,v)\in S\bar\Omega$.
What is the definition at an inward boundary point?
What should~$\tau$ be defined to be at other boundary points?
\end{ex}

\begin{ex}
\label{ex:tau-c1}
The domain $\Omega\subset\R^n$ is called smooth if there is a smooth boundary defining function $\rho\colon\R^n\to\R$ so that $\Omega=\{x\in\R^n;\rho(x)>0\}$, $\partial\Omega=\{x\in\R^n;\rho(x)=0\}$ and $\nabla\rho\neq0$ at~$\partial\Omega$.
By smoothness of~$\partial\Omega$ we refer to the smoothness of the boundary defining function~$\rho$.
(This is equivalent with~$\partial\Omega$ being locally a graph of the required smoothness.)
In addition, we may assume that $\abs{\nabla\rho(x)}=1$ for all $x\in\partial\Omega$.
The the outer unit normal is given by $\nu(x)=-\nabla\rho(x)$.

Consider a point $(x,v)\in S\Omega$ for which $\phi_{\tau(x,v)}(x,v)$ points outward (is not tangential).
Use the implicit function theorem to show that the function~$\tau$ is~$C^1$ in a neighborhood of~$(x,v)$.
\end{ex}

We define the integral function $u^f\colon S\bar\Omega\to\R$ of a function $f\colon S\bar\Omega\to\R$ as
\begin{equation}
\label{eq:uf-def}
u^f(x,v)
=
\int_0^{\tau(x,v)}f(\phi_t(x,v))\dd t
\end{equation}
whenever this integral makes sense.
In words, $u^f(x,v)$ is the integral of~$f$ over the lift of the geodesic starting at the point~$x$ in the direction~$v$.
This kind of integral function will play a big role in our next proof of injectivity of the X-ray transform.

\begin{ex}
\label{ex:u-characteristic}
Let~$\Omega$ be the unit ball and $f\equiv1$ the constant function on~$S\bar\Omega$.
Find a formula for $u^f\colon S\bar\Omega\to\R$.
As you will notice, the resulting function has differentiability issues at the tangential part of the boundary.
\end{ex}

This integral function satisfies a fundamental theorem of calculus:

\begin{ex}
\label{ex:ftc-SM}
Prove that $Xu^f=-f$ in~$S\Omega$ for $f\in C(S\bar\Omega)$.
\end{ex}

\qa

\section{The sphere bundle in two dimensions}
\label{sec:2D-SM}

\subsection{Horizontal and vertical vector fields}

To simplify matters, we choose $n=2$ for this and the next section.
Now~$\Omega$ is a smooth and convex planar domain.
The unit sphere bundle $S\Omega=\Omega\times \Sphere^1$ could also be called the circle bundle.

When convenient, we may consider a function~$f$ on~$\Omega$ to be a function on~$S\Omega$ which just happens to be independent of~$v$.
Formally, this amounts to replacing~$f$ with~$\pi^*f$.
The pullback is defined as $\pi^*f=f\circ\pi$.

The sphere bundle is three-dimensional.
One special direction is given by the geodesic vector field.
There is a second horizontal direction and one vertical direction, too, and we will study derivatives in these directions next.
We will define the horizontal vector field and the vertical vector field.
As in the case of the geodesic vector field, these will be differential operators.

Points $(x,v)\in S\bar\Omega$ can be written as $(x,v_\theta)$ for $x\in\bar\Omega$ and $\theta\in\R/2\pi\Z$, where $v_\theta=(\cos(\theta),\sin(\theta))$.
The vertical vector field is simply differentiation with respect to~$\theta$.
That is, for $u\colon S\bar\Omega\to\R$ we define $Vu\colon S\bar\Omega\to\R$ by
\begin{equation}
Vu(x,v_\theta)
=
\partial_\theta u(x,v_\theta).
\end{equation}

The geodesic vector field may be written as $Xu=v_\theta\cdot\nabla_xu=\cos(\theta)\partial_{x_1}u+\sin(\theta)\partial_{x_2}u$.
Let~$v^\perp$ denote the rotation of~$v$ by~$\frac\pi2$ clockwise.
In terms of angles, $v_\theta^\perp=v_{\theta-\pi/2}$.

We define the horizontal vector field~$X_\perp$ so that
\begin{equation}
\begin{split}
X_\perp u(x,v_\theta)
&=
v^\perp_\theta\cdot\nabla_xu(x,v_\theta)
\\&=
\sin(\theta)\partial_{x_1}u(x,v_\theta)
-
\cos(\theta)\partial_{x_2}u(x,v_\theta).
\end{split}
\end{equation}
The concepts of horizontal and vertical align with figure~\ref{fig:bundle}.

\begin{ex}
Let~$\Omega$ be the unit disc and write~$x$ in polar coordinates.
Write the function~$u^f$ of exercise~\ref{ex:u-characteristic} explicitly in these coordinates (one radius and two angles).
Calculate~$Vu^f$ and~$X_\perp u^f$.
(We know that $Xu^f=-f=-1$.)
Do these functions blow up at at the tangential part of the boundary of the sphere bundle?
(This is the typical place for regularity issues.)
\end{ex}

\begin{bex}
For $f\in C^2(\Omega)$, let~$\pi^*f$ be the pullback over the projection $\pi\colon S\Omega\to\Omega$.
Show that~$f$ is harmonic if and only if $(X^2+X_\perp^2)\pi^*f=0$ in~$S\Omega$.
\end{bex}

The geodesic vector field~$X$ at~$(x,v)$ is the derivative with respect to~$x$ in the direction of~$v$.
The horizontal vector field~$X_\perp$ is the derivative with respect to~$x$ in the direction orthogonal to~$v$.
The vertical vector field is the derivative with respect to the direction~$v$.

As vector fields (as opposed to differential operators), these three vector fields are orthogonal and have unit length.
They are an orthonormal basis to the tangent spaces of the sphere bundle.
This happens on any two-dimensional Riemannian manifold when the sphere bundle is equipped with the so-called Sasaki metric.
In higher dimensions the Sasaki metric is trickier, and its definition needs to be taken to a separate course.

In higher dimensions there are still natural horizontal and vertical derivatives, but they are no longer vector fields.
To avoid technicalities, we stick to dimension two.

\begin{ex}
\label{ex:VXu=0}
Show that if $f\in C(\bar\Omega)$, then $VXu^f=0$ in~$S\Omega$.
Here $f\in C(\bar\Omega)$ is identified with $\pi^*f\in C(S\bar\Omega)$.
\end{ex}

\subsection{Commutators}
\label{sec:commutators}

To calculate with differential operators, we need a couple of basic tools.
We need to be able to integrate by parts and change the order of differentiation.
Integration by parts comes in the next section, and now we will study what happens when the order of differentiation changes.
In our situation the order of differentiation does matter, but it only matters to a lower order, so to say.
The effect of changing the order is captured by commutators.

The commutator of two linear operators~$A$ and~$B$ is $[A,B]=AB-BA$.
We clearly have $[A,A]=0$ for any $A$.
For example, consider the following two operators on functions on the real line:
\begin{equation}
\begin{split}
(Af)(x)&=f'(x),\\
(Bf)(x)&=h(x)f(x),
\end{split}
\end{equation}
where~$h$ is a sufficiently smooth function.
Then
\begin{equation}
\begin{split}
[A,B]f(x)
&=
(ABf)(x)
-
(BAf)(x)
\\&=
\frac{\der}{\der x}(h(x)f(x))
-
h(x)f'(x)
\\&=
h'(x)f(x).
\end{split}
\end{equation}

\begin{ex}
\label{ex:[vf,vf]}
Let us call~$A$ a first order differential operator on the real line if it is of the form
\begin{equation}
Af(x)
=
h(x)f'(x)+g(x)f(x)
\end{equation}
for some smooth functions~$h$ and~$g$.
Show that the commutator of two first order differential operators is a first order differential operator.
\end{ex}

In general, the product of differential operators of orders~$k$ and~$m$ is a differential operator of order $k+m$, and the commutator has order $k+m-1$.
The leading order derivative may vanish, in which case the order is actually lower.
After exercise~\ref{ex:[vf,vf]} we expect the commutator of two vector fields to be a vector field (which is indeed true), and we can verify all our commutators by calculation:

\begin{ex}
\label{ex:comm-XV}
Show that $[X,V]=X_\perp$.
\end{ex}

\begin{ex}
\label{ex:comm-VXp}
Show that $[V,X_\perp]=X$.
\end{ex}

\begin{ex}
\label{ex:comm-XXp}
Show that $[X,X_\perp]=0$.
\end{ex}

The case of a two-dimensional Riemannian manifold is surprisingly similar.
The first two commutators above stay intact, and $[X,X_\perp]$ contains~$V$ and the curvature.
In higher dimensions the formulas are a little trickier, but still the same in spirit.
We will discuss this a little more in section~\ref{sec:pestov-mfld}.

In fact, the only thing we need to know about commutators in the proof is the following lemma.
Observe that~$X_\perp$ does not appear in the claim, but it is useful for the proof.

\begin{lemma}
\label{lma:XV,VX}
Our vector fields satisfy
\begin{equation}
[XV,VX]
=
-X^2.
\end{equation}
\end{lemma}

\begin{proof}
%
There are many ways to prove this, the simplest one being to simply expand the commutators in full and use the commutator formulas of exercises~\ref{ex:comm-XV}, \ref{ex:comm-VXp}, and~\ref{ex:comm-XXp} to swap the orders so that highest order operators cancel out.
One option is to use the commutator property $[AB,C]=A[B,C]+[A,C]B$, which is to be verified by hand in exercise~\ref{ex:commutator-product}:
\begin{equation}
\begin{split}
[XV,VX]
&=
[XV,XV-X_\perp]
\\&=
[XV,XV]
-
[XV,X_\perp]
\\&=
0
-
X[V,X_\perp]
-
[X,X_\perp]V
\\&=
-X^2
.
\end{split}
\end{equation}
All methods lead to the same conclusion, of course.
The intermediate steps here are but an example.
\end{proof}

The commutator of two second order operators is typically of third order.
In this particular case it happens to be second order because~$XV$ and~$VX$ only differ by the first order operator~$X_\perp$.

\begin{ex}
\label{ex:commutator-product}
Consider linear operators ($n\times n$ matrices, for example) $A$, $B$, and~$C$.
Show that $[A,BC]=[A,B]C+B[A,C]$ and $[AB,C]=A[B,C]+[A,C]B$.
\end{ex}

\begin{ex}
Compute $[XVV,VX_\perp]$ and $[V^2,X_\perp^2]$.
There might not be a clear simplest form as in lemma~\ref{lma:XV,VX}, but simplify as far as you can.
\end{ex}

\subsection{Integration on the sphere bundle}

The sphere bundle is a product space, and we can naturally use the product measure~$\Sigma$.
Therefore the integral of $g\in C(S\bar\Omega)$ is
\begin{equation}
\int_{S\Omega}g(x,v)\dd\Sigma(x,v)
=
\int_\Omega\int_{\Sphere^1}g(x,v)\dd S(v)\dd x.
\end{equation}
Alternatively, the~$\Sphere^1$ integral can be written as $\int_0^{2\pi}g(x,v_\theta)\dd\theta$.

The boundary~$\partial\Omega$ is a closed smooth curve, and we have a natural measure on it.
One way to describe it is to write the curve as $\alpha\colon[0,L]\to\R^2$ with arc length parametrization and then integrate on the interval $[0,L]$.

This gives rise to a measure~$\tilde\sigma$ on~$\partial(S\Omega)$, given by
\begin{equation}
\int_{\partial(S\Omega)}g\dd\tilde\sigma
=
\int_{\Sphere^1}
\int_0^L
g(\alpha(t),v)
\dd t
\dd S(v)
\end{equation}
for any $g\in C(S\bar\Omega)$.

It turns out that the measure $\sigma=\abs{v\cdot\nu(x)}\tilde\sigma$ is more natural.
It will appear in a change of variables formula for integration over the sphere bundle.

\begin{proposition}[Santal\'o's formula]
\label{prop:santalo}
Let $\Omega\subset\R^2$ be a convex, bounded, and smooth domain.
For $g\in C(S\bar\Omega)$ we have
\begin{equation}
\int_{S\Omega}g\dd\Sigma
=
\int_\inwb
\int_0^{\tau(x,v)}
g(\phi_t(x,v))
\dd t
\dd\sigma(x,v).
\end{equation}
Alternatively, the integral can be taken over the entire~$\partial(S\Omega)$ since~$\tau$ vanishes outside~$\inwb$.
\end{proposition}

\begin{proof}
First, we change the order of integration in the integral over~$S\Omega$:
\begin{equation}
\label{eq:vv8}
\begin{split}
\int_{S\Omega}g\dd\Sigma
&=
\int_\Omega\int_{\Sphere^1}g(x,v)\dd S(v)\dd x
\\&=
\int_{\Sphere^1}\int_\Omega g(x,v)\dd x\dd S(v).
\end{split}
\end{equation}
Now fix any $v\in \Sphere^1$ and consider the inner integral
\begin{equation}
\label{eq:vv9}
I(v)
=
\int_\Omega g(x,v)\dd x.
\end{equation}
We extend~$g$ to~$S\R^2$ by zero for convenience.
We write the plane as an orthogonal direct sum $\R^2=v\R\oplus v^\perp\R$.
With this decomposition, we have
\begin{equation}
I(v)
=
\int_\R
\int_\R
g(sv^\perp+tv,v)
\dd t
\dd s
.
\end{equation}
The inner integral is an integral along the geodesic flow as desired.
We will turn the outer integral into an integral over the boundary.

Let us denote
\begin{equation}
\partial_v\Omega
=
\{x\in\partial\Omega;v\cdot\nu(x)<0\}.
\end{equation}
We parametrize this part of the boundary by a (counterclockwise) unit speed curve $\beta\colon[0,L_v]\to\partial_v\Omega$.
Observe that $\{(x,v);v\in \Sphere^1,x\in\partial_v\Omega\}=\inwb$.

Let us denote
\begin{equation}
a(v)
=
\min\{s\in\R;(sv^\perp+v\R)\cap\bar\Omega\neq\emptyset\}
\end{equation}
and
\begin{equation}
b(v)
=
\max\{s\in\R;(sv^\perp+v\R)\cap\bar\Omega\neq\emptyset\}.
\end{equation}
Now, there is a function $w\colon(a(v),b(v))\to(0,L_v)$ so that $\beta(w(s))-sv^\perp\in v\R$.
In fact,~$w$ is a~$C^1$ diffeomorphism with $w'(s)=-1/v\cdot\nu(\beta(w(s)))>0$.
The details are left as exercise~\ref{ex:explain-w-2}.
Some of the features are captured in figure~\ref{fig:santalo}.

\begin{figure}[t]
    \centering
    \includegraphics[scale=0.5,trim={3cm 0 3cm 0},clip,page=24]{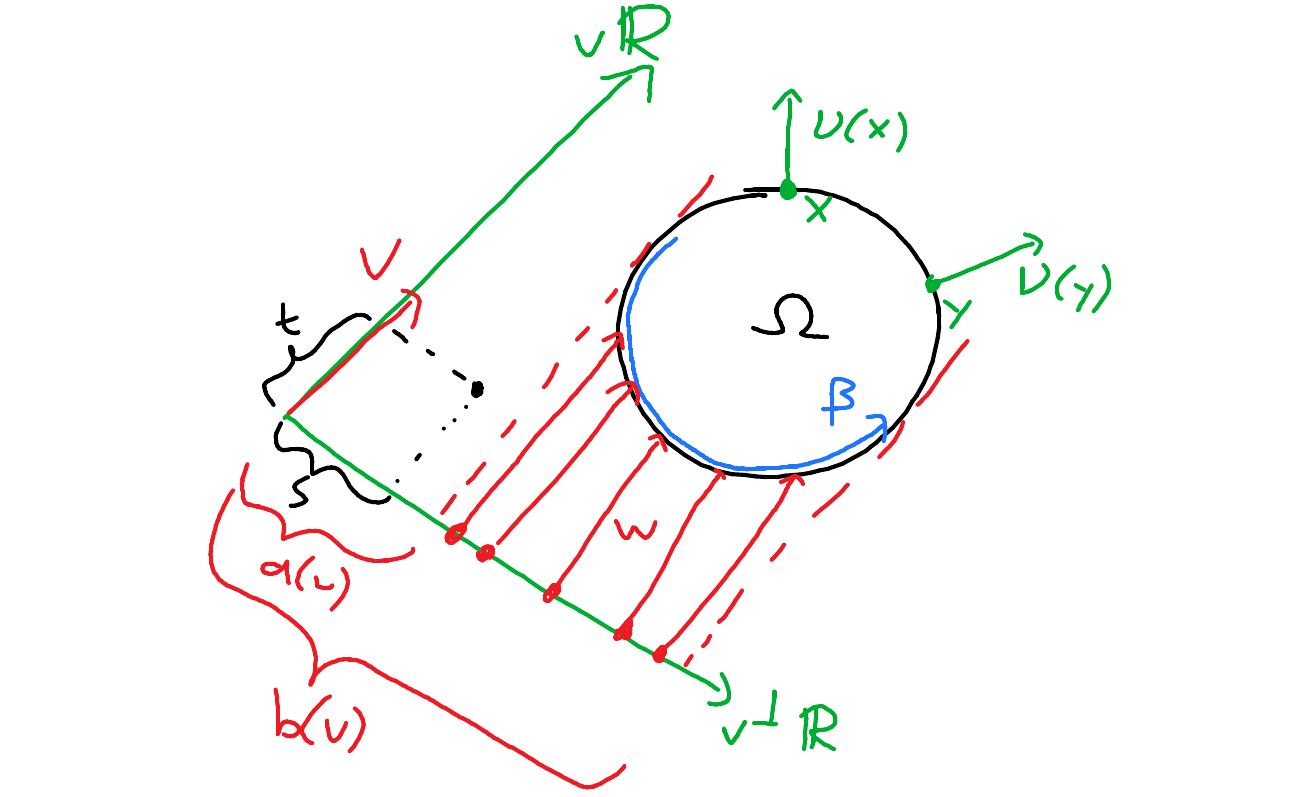}
    \caption{\puna{The vector~$v$} defines a coordinate system whose \viher{axes are along and orthogonal to~$v$}. The coordinates are called~$t$ (along~$v$) and~$s$ (orthogonal to~$v$). The projection of the domain $\Omega$ to \viher{the line orthogonal to~$v$} occupies an interval \puna{from~$a(v)$ to~$b(v)$}. The \sini{unit speed curve~$\beta$} parametrizes the side of $\partial\Omega$ visible from below, defined in terms of \viher{the outer unit normal vector field~$\nu$ of~$\partial\Omega$}. \puna{the map~$w$} maps the interval on \viher{on the $s$-axis} to \sini{our preferred half of the boundary}.}
    \label{fig:santalo}
\end{figure}

We change the variable of integration from~$s$ to $z=w(s)$ and obtain
\begin{equation}
\begin{split}
I(v)
&=
\int_\R
\int_\R
g(sv^\perp+tv,v)
\dd t
\dd s
\\&=
\int_{a(v)}^{b(v)}
\int_\R
g(sv^\perp+tv,v)
\dd t
\dd s
\\&=
\int_{0}^{L_v}
\int_\R
g(w^{-1}(z)v^\perp+tv,v)
\dd t
(-v\cdot\nu(\beta(z)))
\dd z
\\&=
\int_0^{L_v}
\int_0^{\tau(\beta(z),v)}
g(\beta(z)+tv,v)
\dd t
\abs{v\cdot\nu(\beta(z))}
\dd z.
\end{split}
\end{equation}
Since~$\beta$ is a subcurve of~$\alpha$ (restriction to a subinterval, possibly after rechoosing the initial and final point on~$\alpha$) and $\tau=0$ on the part $\alpha\setminus\beta$, we get
\begin{equation}
\label{eq:vv10}
I(v)
=
\int_0^L
\int_0^{\tau(\alpha(z),v)}
g(\alpha(z)+tv,v)
\dd t
\abs{v\cdot\nu(\alpha(z))}
\dd z.
\end{equation}
Combining~\eqref{eq:vv8}, \eqref{eq:vv9}, and~\eqref{eq:vv10}, we find
\begin{equation}
\begin{split}
\int_{S\Omega}g\dd\Sigma
&=
\int_{\Sphere^1}
\int_0^L
\int_0^{\tau(\alpha(z),v)}
g(\alpha(z)+tv,v)
\dd t
\abs{v\cdot\nu(\alpha(z))}
\dd z
\dd S(v)
\\&=
\int_{\partial(S\Omega)}
\left(
\int_0^{\tau(x,v)}
g(x+tv,v)
\dd t
\right)
\abs{v\cdot\nu(x)}
\dd\tilde\sigma(x,v)
\\&=
\int_{\partial(S\Omega)}
\left(
\int_0^{\tau(x,v)}
g(\phi_t(x,v))
\dd t
\right)
\dd\sigma(x,v)
\end{split}
\end{equation}
as claimed.
\end{proof}


\begin{ex}
\label{ex:explain-w-2}
Explain why $w\in C^1$ and $w'(s)=-1/v\cdot\nu(\beta(w(s)))>0$.
\end{ex}

\begin{ex}
Show that the measure of~$\Omega$ is $\frac1{2\pi}\int_\inwb\tau(x,v)\dd\sigma(x,v)$.
\end{ex}

Santal\'o's formula states that the integral over the sphere bundle can be calculated by calculating it one geodesic at a time, first integrating over the (lifted) geodesic and the integrating over the initial points and directions of these geodesics at~$\inwb$.
The space of all geodesics through~$\Omega$ can be identified with~$\inwb$.
The measure spaces $(\Gamma,\mu)$ and $(\inwb,\sigma)$ (with Borel $\sigma$-algebras) are two descriptions of the same thing.

The formula is a change of variables.
It will help us find Green-type formulas for our three vector fields (see exercises~\ref{ex:green-V}, \ref{ex:green-X} and~\ref{ex:green-Xperp}), and those will lead to integration by parts on the sphere bundle.

The formula also gives rise to some integral properties which will be convenient.

\begin{ex}
\label{ex:green-V}
Show that if $g\in C^1(S\bar\Omega)$, then
\begin{equation}
\int_{S\Omega}Vg\dd\Sigma
=
0.
\end{equation}
Santal\'o is not needed.
\end{ex}

\begin{ex}
\label{ex:green-X}
Show that if $g\in C^1(S\bar\Omega)$, then
\begin{equation}
\int_{S\Omega}Xg\dd\Sigma
=
\int_\inwb
\left(
g(\phi_{\tau(x,v)}(x,v))
-
g(x,v)
\right)
\dd\sigma(x,v).
\end{equation}
Santal\'o is useful.
\end{ex}

\begin{ex}
\label{ex:green-Xperp}
Show that if $g\in C^2(S\bar\Omega)$ and $g|_{\partial(S\Omega)}=0$, then
\begin{equation}
\int_{S\Omega}X_\perp g\dd\Sigma
=
0.
\end{equation}
Use exercises~\ref{ex:green-V} and~\ref{ex:green-X} and the commutator formulas.
\end{ex}

\qa

\section{X-ray tomography and the transport equation}
\label{sec:pestov}

\subsection{Integration revisited}

For $g,h\in L^2(S\Omega,\Sigma)$ we write
\begin{equation}
\ip{g}{h}
=
\int_{S\Omega}gh\dd\Sigma
\end{equation}
and $\aabs{g}=\sqrt{\ip{g}{g}}$.
Our functions in this section are real-valued so no conjugation is needed.
The complex case is not harder, but the real approach is technically convenient.

Santal\'o's formula (proposition~\ref{prop:santalo}) makes it easy to find integration by parts formulas for our three vector fields.

\begin{lemma}
\label{lma:sm-ibp}
For any $g,h\in C^\infty_c(S\Omega)$ we have
\begin{equation}
\begin{split}
\ip{g}{Xh}
&=
-\ip{Xg}{h},
\\
\ip{g}{Vh}
&=
-\ip{Vg}{h},\quad\text{and}
\\
\ip{g}{X_\perp h}
&=
-\ip{X_\perp g}{h}.
\end{split}
\end{equation}
\end{lemma}

This lemma allows us to integrate by parts with our derivatives.
As always, integration by parts follows from two elements: a product rule for the derivative and the fundamental theorem of calculus linking derivatives and integrals.
The product rule is studied in exercise~\ref{ex:derivation} and role of the fundamental theorem of calculus is played by exercises~\ref{ex:green-V}, \ref{ex:green-X}, and~\ref{ex:green-Xperp}, now with zero boundary values.

\begin{ex}
\label{ex:derivation}
An operator $A\colon C^\infty(S\Omega)\to C^\infty(S\Omega)$ is called a derivation if it is linear and satisfies $A(gh)=gAh+hAg$ and $A1=0$, where~$1$ stands for the constant function.

Show that the commutator of two derivations is a derivation.
Explain why~$X$, $V$, and~$X_\perp$ are derivations.
\end{ex}

\begin{ex}
Prove lemma~\ref{lma:sm-ibp} using results from the previous section.
Exercise~\ref{ex:derivation} is also useful.
\end{ex}

\subsection{A second order PDE}

Now we finally begin our analysis of the X-ray transform.
We want to show that if $f\in C_c^\infty(\Omega)$ integrates to zero over all lines, then $f=0$.
The starting point is the so-called transport equation $Xu^f=-f$ from exercise~\ref{ex:ftc-SM}.
However, it will be more convenient to take a second derivative and pass to a second order homogeneous equation.

First, let us define the integral function~$u^f$ as above.
Recall from exercise~\ref{ex:VXu=0} that $VXu^f=0$.
In addition,~$u^f$ satisfies the boundary condition $u^f|_{\partial(S\Omega)}=0$.
For outward pointing and tangential directions, this is because $\tau=0$.
For inward pointing directions, this is because $\xrt f=0$.
See figure~\ref{fig:uf-0-on-boundary}.

\begin{figure}[t]
    \centering
    \includegraphics[scale=0.35,trim={3cm 0 3cm 0},clip,page=25]{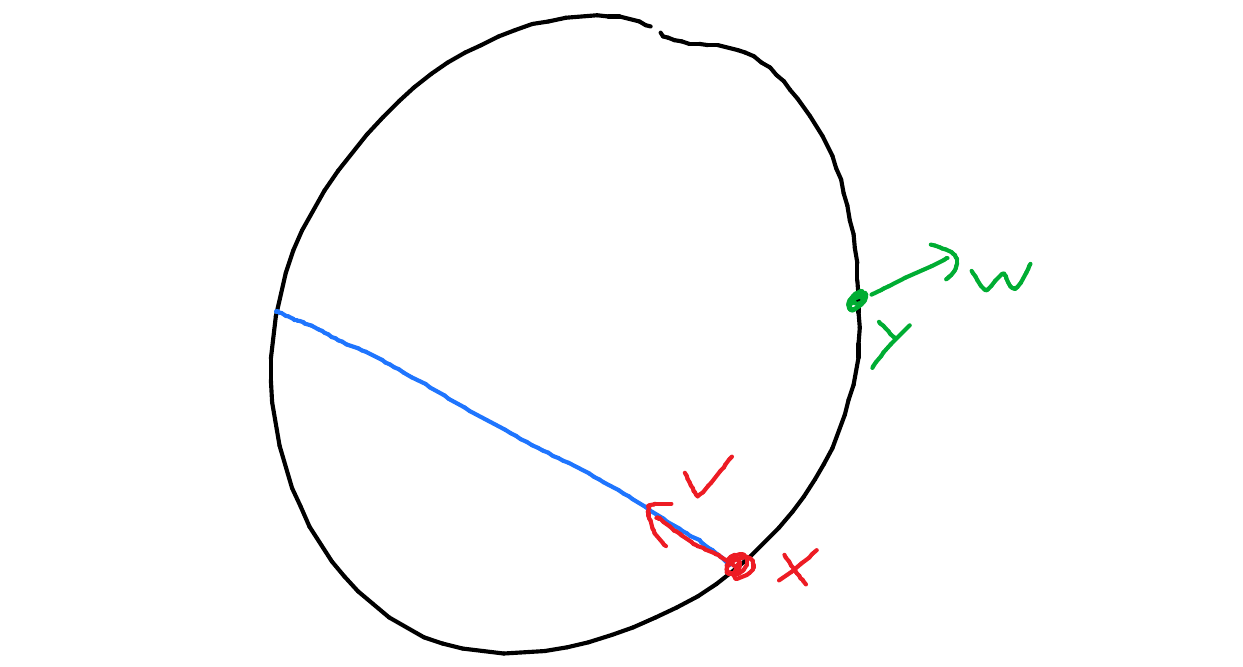}
    \caption{If \puna{a point $(x,v)\in\partial(S\Omega)$ points inwards}, then $u^f(x,v)=0$ because $\xrt f=0$. If \viher{a point $(y,w)\in\partial(S\Omega)$ points outwards}, then $u^f(x,v)=0$ because the geodesic to be integrated over has zero length. }
    \label{fig:uf-0-on-boundary}
\end{figure}

\begin{ex}
\label{ex:inwb-u-xrt}
Explain how to identify the functions~$\xrt f$ and~$u^f|_\inwb$.
\end{ex}

Previously we defined the X-ray transform of a scalar function $f\colon\R^n\to\R$.
With the help of exercise~\ref{ex:inwb-u-xrt} we can in fact define the X-ray transform of a compactly supported continuous function $f\colon S\R^n\to\R$.
This leads to so-called tensor tomography, which is outside the scope of this course.

The function $u=u^f$ solves the boundary value problem
\begin{equation}
\label{eq:bvp}
\begin{cases}
VXu=0 & \text{in }S\Omega\\
u=0 & \text{on }\partial(S\Omega).
\end{cases}
\end{equation}
Clearly $u=0$ is a solution.
If the solution to this second order PDE is unique, then it follows that $u^f=0$ and therefore $f=-Xu^f=0$.
This leads to injectivity of the X-ray transform.

The operator~$XV$ is not elliptic, hyperbolic, or parabolic --- you could call it diabolic.
Therefore we do not have access to standard uniqueness theorems, and we have to show uniqueness by hand.

\begin{ex}
Show that
\begin{equation}
XV=\frac14(X+V)^2-\frac14(X-V)^2+\frac12X_\perp
\end{equation}
and
\begin{equation}
VX=\frac14(X+V)^2-\frac14(X-V)^2-\frac12X_\perp
.
\end{equation}
Since the three vector fields~$X$, $V$, and~$X_\perp$ are orthonormal (with respect to the Sasaki metric), so are $\frac1{\sqrt{2}}(X+V)$, $\frac1{\sqrt{2}}(X-V)$, and~$X_\perp$.
Therefore our two operators look locally like the operators $\frac12(\partial_x^2-\partial_y^2\pm\partial_z)$ in~$\R^3$.
\end{ex}

One might expect that the first order term is not that relevant, but it turns out to be very important.
In other words, the order of the operators~$X$ and~$V$ is crucial.
We will show that assuming zero boundary values the PDE $VXu=0$ has unique solutions, but $XVu=0$ never does.
The uniqueness result will be proven in the next section.
The non-uniqueness result is easier and is left as exercise~\ref{ex:XV}.


\begin{ex}
\label{ex:X=0}
Suppose $g\in C^1(S\bar\Omega)$.
Show that if $g|_\inwb=0$ and $Xg=0$, then $g=0$.
\end{ex}

\begin{ex}
\label{ex:V-boundary}
Also, explain why $g|_\inwb=0$ implies $Vg|_\inwb=0$.
\end{ex}

\begin{ex}
\label{ex:XV}
Suppose $u\in C^\infty(S\bar\Omega)$ with $u|_{\partial(S\Omega)}=0$.
Show that $XVu=0$ in~$S\Omega$ if and only if there is $f\in C^\infty(\bar\Omega)$ with zero boundary values so that $u=\pi^*f$.
\end{ex}

Exercise~\ref{ex:XV} will, perhaps surprisingly, play a role in section~\ref{sec:vf}.

\subsection{Properties of the integral function}

Before studying the boundary value problem~\eqref{eq:bvp} further, it is good to verify that~$u^f$ is sufficiently regular.

\begin{lemma}
\label{lma:u-smooth}
The function~$u^f$ defined in~\eqref{eq:uf-def} is in~$C_c^\infty(S\Omega)$ when $f\in C_c^\infty(\Omega)$ and $\xrt f=0$.
\end{lemma}

\begin{proof}
Recall exercise~\ref{ex:tau-c1}.
If the boundary~$\partial\Omega$ is smooth, it follows with the same argument and the smooth version of the implicit function theorem that~$\tau$ is smooth near $(x,v)\in S\Omega$ when $\phi_{\tau(x,v)}(x,v)$ points outward.
The function~$u^f$ is defined by
\begin{equation}
u^f(x,v)
=
\int_0^{\tau(x,v)}f(x+tv)\dd t.
\end{equation}
Since~$f$ and~$\tau$ are smooth, so is~$u^f$.
(We omit some technical details here, but the statement is hopefully plausible to the reader.
)

So far we have not used the fact that~$f$ is compactly supported nor tried to prove that so is~$u^f$.
Also, smoothness at tangential exits has not been established yet.

Assume~$f$ is supported in a compact set $K\subset\Omega$.
By exercise~\ref{ex:compact-convex} we may assume that~$K$ is convex.
We will show that~$u^f$ is supported in $SK=K\times \Sphere^1$, which means that $u^f(x,v)=0$ whenever $x\notin K$.

This will also prove smoothness near points $(x,v)\in S\Omega$ where $\phi_{\tau(x,v)}(x,v)$ is tangential to~$\partial\Omega$; see exercise~\ref{ex:convex-transversal}.
The technicalities are reduced considerably by showing that the function is zero where differentiating it would be hardest.
Therefore it only remains to prove the support condition for~$u^f$.

Take any $x\in\bar\Omega\setminus K$ and $v\in \Sphere^1$.
Let
\begin{equation}
\gamma(x,v)
=
\{x+tv;t\in[0,\tau(x,v)]\}
\end{equation}
be the line from~$x$ to~$\partial\Omega$ in the direction of~$v$.
If $\gamma(x,v)\cap K=\emptyset$, then $u^f(x,v)=0$ since~$u^f(x,v)$ is the integral of~$f$ over~$\gamma(x,v)$.

Because~$K$ is convex and $x\notin K$, at most one of the line segments~$\gamma(x,v)$ and~$\gamma(x,-v)$ can meet~$K$.
See figure~\ref{fig:no-2-hits}.
Thus if $\gamma(x,v)\cap K\neq\emptyset$, then $\gamma(x,-v)\cap K=\emptyset$.
By the argument given above, $u^f(x,-v)=0$.

\begin{figure}[t]
    \centering
    \includegraphics[scale=0.4,trim={3cm 0 3cm 0},clip,page=1]{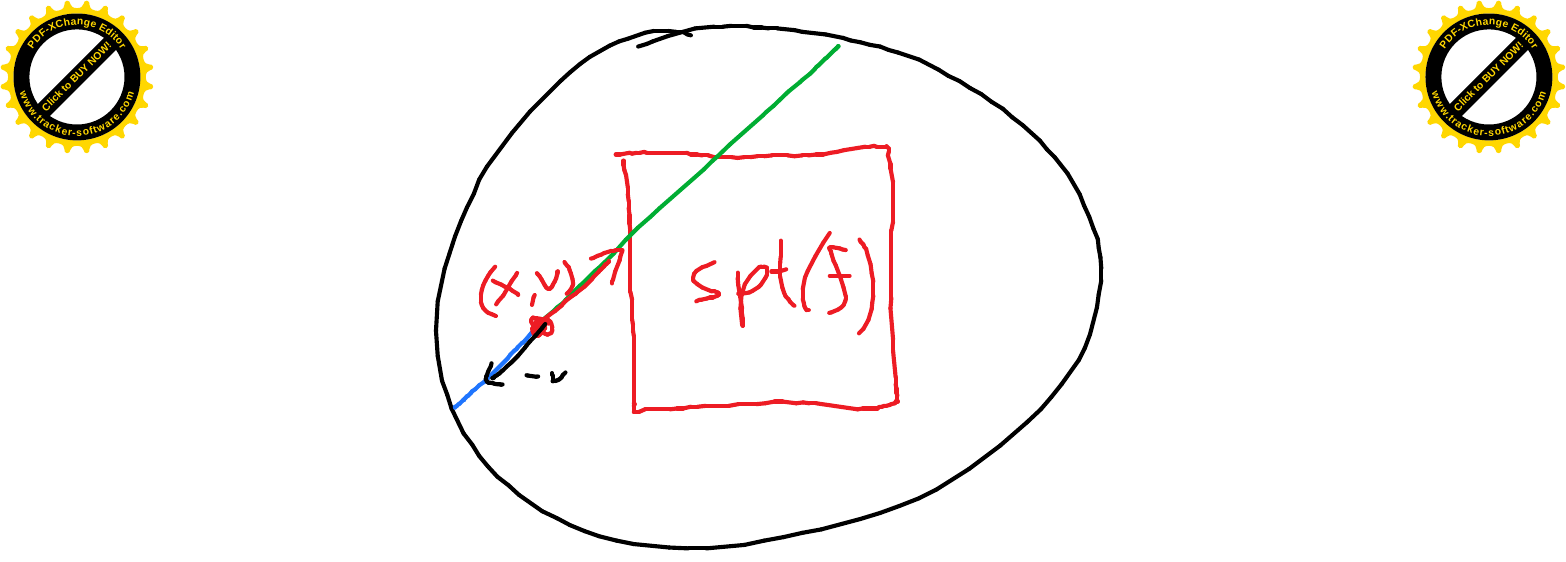}
    \caption{When the \puna{point~$x$} is close enough to the boundary, at most one of the two geodesics \viher{$\gamma(x,v)$ (in the direction~$v$} and \sini{$\gamma(x,-v)$ (in the direction $-v$} can meet \puna{the support of~$f$}. If the direction~$v$ is close enough to being tangent to the boundary, it may well happen that neither line segment hits the support.}
    \label{fig:no-2-hits}
\end{figure}

On the other hand, $u^f(x,v)+u^f(x,-v)=0$ for all $(x,v)\in S\Omega$ since $\xrt f=0$, so $u^f(x,-v)=0$ implies $u^f(x,v)=0$.
We have thus shown that $u^f(x,v)=0$ when $x\notin K$.
\end{proof}

\begin{ex}
\label{ex:compact-convex}
Suppose $\Omega\subset\R^n$ is a convex open set and $K\subset\Omega$ compact.
Show that the convex hull of~$K$ is compact and contained in~$\Omega$.
(Carath\'eodory's theorem can be useful.)
\end{ex}

\begin{ex}
\label{ex:convex-transversal}
Suppose $\Omega\subset\R^n$ is a bounded and convex~$C^1$ domain.
Suppose $(x,v)\in S\bar\Omega$ is such that $\phi_{\tau(x,v)}(x,v)$ is tangential to~$\partial\Omega$.
Show that $x\in\partial\Omega$ and $v\cdot\nu(x)=0$.
\end{ex}

We have chosen to work with compactly supported smooth functions to avoid technical difficulties.
The same method works for $f\in C^2(\bar\Omega)$ as well, with no assumptions on boundary values.
This would require more delicate analysis of boundary behaviour, since in general $u^f\notin C^2(S\bar\Omega)$ even if $f\in C^2(\bar\Omega)$.
In fact, without assuming $\xrt f=0$, one only has $u^f\in C^{1/2}(S\bar\Omega)$; see exercise~\ref{ex:u-characteristic}.

\begin{ex}
The integral function~$u^f$ was defined in~\eqref{eq:uf-def} for any function~$f$ on the sphere bundle.
Does lemma~\ref{lma:u-smooth} work also when $f\in C_c^\infty(S\Omega)$?
How does it help with smoothness that $f(x,v)$ is independent of $v$?
\end{ex}

We will next prove an integral identity.
The statement concerns second and first order derivatives, but the proof uses derivatives up to order four.
In cases like this the theorem can be shown to hold in~$C^2$ using the density of~$C^\infty$ in~$C^2$.

\subsection{The Pestov identity}

The key to proving uniqueness of~\eqref{eq:bvp} is an integral identity known as the Pestov identity.
It was introduced by Mukhometov, and it could well be called the Mukhometov--Pestov identity.

\begin{proposition}[Pestov identity]
\label{prop:pestov}
If $u\in C^\infty_c(S\Omega)$, then
\begin{equation}
\aabs{VXu}^2
=
\aabs{XVu}^2
+\aabs{Xu}^2.
\end{equation}
\end{proposition}

\begin{proof}
Using lemmas~\ref{lma:sm-ibp} and~\ref{lma:XV,VX} we find
\begin{equation}
\begin{split}
\aabs{VXu}^2
-\aabs{XVu}^2
&=
\ip{VXu}{VXu}
-\ip{XVu}{XVu}
\\&=
-\ip{VVXu}{Xu}
+\ip{XXVu}{Vu}
\\&=
\ip{XVVXu}{u}
-\ip{VXXVu}{u}
\\&=
\ip{(XVVX-VXXV)u}{u}
\\&=
\ip{[XV,VX]u}{u}
\\&=
\ip{-X^2u}{u}
\\&=
\ip{Xu}{Xu}
\end{split}
\end{equation}
as claimed.
\end{proof}

As mentioned earlier, the assumption of compact support is not necessary.
It is enough that $u|_{\partial(S\Omega)}=0$, but the proof would be somewhat more technical.
Smoothness is not necessary either, it is just convenient.

\begin{ex}
Use the Santal\'o formula to rewrite~$\aabs{Xw}^2_{L^2(S\Omega)}$ when $w\in C_c^\infty(S\Omega)$.
What is the integral you end up calculating over each geodesic?
Can these integrals be simplified?
Two terms of this kind appear on the right-hand side of the Pestov identity, with both $w=Vu$ and $w=u$.
(If you are familiar with the index form, you may want to compare this situation to it. If this kind of analysis of the Pestov identity is done on a manifold, the connection to index forms becomes more clear and important.)
\end{ex}

\begin{bex}
Stare at the Pestov identity, experience enlightenment, and explain what it means and why it should hold true.
What, in your opinion, is the meaning of the Pestov identity?
(This exercise is not about what it can be used for; that is exercise~\ref{ex:xrt-with-pestov}. This is about what the identity tells in itself.)
\end{bex}

The Pestov identity makes proving our injectivity result easy:

\begin{theorem}
\label{xrtthm:pestov}
Let $\Omega\subset\R^2$ be a bounded, smooth, and strictly convex domain.
If $f\in C_c^\infty(\Omega)$ satisfies $\xrt f=0$, then $f=0$.
\end{theorem}

\begin{ex}
\label{ex:xrt-with-pestov}
Prove theorem~\ref{xrtthm:pestov} by applying proposition~\ref{prop:pestov} to~$u^f$.
\end{ex}

This argument provides us with yet another uniqueness proof.
However, it does not give an inversion formula for the X-ray transform.
There are inversion formulas within this framework, but finding one requires considerably more work than proving uniqueness.

\subsection{Remarks about manifolds}
\label{sec:pestov-mfld}

This method can also be used to prove injectivity results on many Riemannian manifolds with boundary.
The sphere bundle and the derivatives on it are still well defined and useful.
The Santal\'o formula still holds true.
To be able to use the Pestov identity, the integral function~$u^f$ needs to be regular enough, and the right-hand side of the identity needs to be non-negative.

To obtain convenient regularity, the manifold is typically assumed to be compact and have strictly convex boundary.
Strict convexity is defined in terms of the curvature of the boundary.

The Pestov identity on a two-dimensional orientable Riemannian manifold~$M$ with boundary reads
\begin{equation}
\aabs{VXu}^2
=
\aabs{XVu}^2
+\aabs{Xu}^2
-\int_{SM}K\abs{Vu}^2\dd\Sigma
,
\end{equation}
where~$K$ is the Gaussian curvature of the surface.
If $K\leq0$, then the desired positivity result follows.
Indeed, the X-ray transform is injective on non-positively curves surfaces with strictly convex boundary.

In higher dimensions things are somewhat more complicated.
Vertical and horizontal derivatives are no longer given by vector fields, and gradient-like operators are needed instead.
In addition, curvature can no longer be adequately described with a scalar function.
However, there is a Pestov identity and it can be used to prove similar results.
Positivity now depends on the sectional curvature.

These results can be generalized in various ways.
Development and application of the relevant tools in differential geometry require a separate course.
See~\cite{I:geodesics,PSU-book}.

\qa

\section{X-ray tomography of vector fields}
\label{sec:vf}

\subsection{Definition of the X-ray transform}

So far we have only discussed X-ray tomography of scalar functions $f\colon\R^n\to\R$.
We can ask a similar question for other kinds of functions as well, and we will only explore one generalization in this course: X-ray tomography of vector fields.

A vector field in the Euclidean space is a function $f\colon\R^n\to\R^n$.
The integral of~$f$ over a line $\gamma\colon\R\to\R^n$ is the X-ray transform
\begin{equation}
\label{eq:vf-xrt}
\xrt f(\gamma)
=
\int_\R f(\gamma(t))\cdot\dot\gamma(t)\dd t
\end{equation}
whenever this integral exists.
(Some people call the transform something else in the case of vector fields. We do not. See bonus exercise~\ref{bex:doppler}.)
We will continue to use unit speed parametrization, although in this particular case it does not make a difference.
Those familiar with differential forms may identify a vector field with a one-form, and the integral of a $k$-form over an oriented $k$-dimensional manifold is parametrization invariant.

\begin{ex}
Prove that formula~\eqref{eq:vf-xrt} for the X-ray transform of a vector field is in fact invariant under any orientation-preserving reparametrization.

What happens if orientation is flipped?
Does the integral of a scalar function change if you reparametrize it or change orientation?
\end{ex}

We can now ask our main question in this new setting:
Does~$\xrt f$ determine~$f$ uniquely?
In other words, if $\xrt f=\xrt g$, do we then necessarily have $f=g$?
Is the X-ray transform~$\xrt$ injective when it operates on vector fields?

\subsection{An application}
\label{sec:doppler-application}

Generalizing mathematical questions is commonplace, and a mathematician may not need any further motivation for this variant of the problem.
While mathematical interest might be a sufficient reason for this detour, we will also present one physical application.
The applications of X-ray tomography of scalar and vector fields are not limited to what is mentioned in this course, and some applications call for further mathematical generalizations.

Consider a stationary flow of a liquid (or fluid), described by the flow field $u\colon\R^3\to\R^3$.
That is, at the point~$x$ the liquid flows with velocity~$u(x)$ at all times.
The speed of sound can be described by a scalar field $c\colon\R^3\to(0,\infty)$, but we assume that it is constant.
(Gravity causes position dependence to the speed of sound, and~$c$ can be coupled with~$u$ if the flow is compressible and large.)

If $\abs{u(x)}\ll c$ for all~$x$, then sound waves in the moving liquid travel at roughly straight lines, but their speeds along those lines are changed by~$u$.
(We can consider the flow to be a small perturbation to the completely still reference situation. Travel time has first order dependence on~$u$, 
but the change of trajectories only has a second order effect on it.
Therefore in the linearized problem geometry is unchanged but travel times change.
We will not attempt to make this linearization procedure precise.)

Consider a straight line $\gamma\colon[0,L]\to\R^3$ parametrized by arc length.
If $u=0$, then the time to travel from~$\gamma(0)$ to~$\gamma(L)$ is
\begin{equation}
\int_0^L \frac1{c}\dd s.
\end{equation}
The presence of~$u$ changes this to
\begin{equation}
\int_0^L \frac1{c+u(\gamma(s))\cdot\dot\gamma(s)}\dd s
\approx
\frac Lc
-
c^{-2}
\int_0^L u(\gamma(s))\cdot\dot\gamma(s)\dd s.
\end{equation}
The first term is the background travel time of the case $u=0$, the second term is the leading order deviation from the background.
Therefore the (linearized) travel time measurement determines~$\xrt u$.

Linearized travel time tomography often leads to X-ray tomography, but the unknown objects may or may not be scalar functions.
See theorem~\ref{thm:outlook-lin}.

The physical problem is then whether such time-of-flight measurements determine the flow field~$u$.
Does it help if the liquid is incompressible, which means $\nabla\cdot u=0$?

\begin{bex}
\label{bex:doppler}
The X-ray transform of vector fields is also known as the Doppler transform.
How is our physical example related to the Doppler effect?
\end{bex}

\subsection{Non-uniqueness and potentials}
\label{sec:non-u-pot}

It turns out that the answer to our main question is ``no'': a vector field~$f$ is not uniquely determined by~$\xrt f$.
There are vector fields~$f$ that are not identically zero but for which $\xrt f=0$.

The next best thing to ask for is a characterization of the kernel of the X-ray transform.
Can we characterize the set of those~$f$ for which $\xrt f=0$?

There is a special class of vector fields we study first: gradient fields.
If $h\colon\R^n\to\R$ is a smooth scalar function, then~$\nabla h$ is a smooth vector field.
Let us calculate the X-ray transform of such a vector field.

\begin{ex}
Let $h\colon\R^n\to\R$ be a smooth scalar function and $\gamma\colon[0,L]\to\R^n$ a line.
Show that
\begin{equation}
\int_0^L \nabla h(\gamma(t))\cdot\dot\gamma(t)\dd t
=
h(\gamma(L))-h(\gamma(0)).
\end{equation}
Explain why, if $h\in C^\infty_c(\R^n)$, then~$\xrt(\nabla h)$ is well defined and identically zero.
\end{ex}

This means that there is a freedom to change a vector field~$f$ to $f+\nabla h$ without changing~$\xrt f$ at all.
This is called a gauge freedom.

We now ask a refined question:
If a sufficiently nice vector field $f\colon\R^n\to\R^n$ satisfies $\xrt f=0$, then is there a scalar function~$h$ so that $f=\nabla h$?

The answer to this refined question is indeed positive, and we will prove it in one special case.
For simplicity, we will only prove the result in two dimensions.
In exercise~\ref{ex:2d-hd} we saw that for scalar functions the higher dimensional result follows from the one in dimension two.
The same argument works here, too:

\begin{ex}
Suppose this is known: If a compactly supported smooth vector field~$f$ on~$\R^2$ satisfies $\xrt f=0$, then there is a smooth compactly supported scalar function~$h$ on the plane so that $f=\nabla h$.

Show this: A smooth compactly supported vector field~$f$ on~$\R^3$ satisfies $\xrt f=0$ if and only if there is a smooth compactly supported scalar function~$h$ so that $f=\nabla h$.
(If you want, you can make use of the theorem that a compactly supported smooth vector field on~$\R^3$ is a gradient field of a compactly supported potential if and only if it has zero curl. Therefore it is enough to show that $\nabla\times f=0$. The same argument works in higher dimensions as well if one uses differential forms and the fact that the first de Rham cohomology group of the Euclidean space is zero.)
\end{ex}

In three dimensions one can write a vector field as a sum of a gradient field and a solenoidal (divergence-free) vector field in a unique way.
This is known as the Helmholtz decomposition.
There is an analogous decomposition in higher dimensions and also on manifolds, known as the Hodge decomposition.

The X-ray transform of the gradient component is always zero, but the rest is uniquely determined as we shall see.
This kind of result is known as solenoidal injectivity.
In particular, it follows that a solenoidal vector field is uniquely determined by its X-ray transform.
Our physical example problem is indeed uniquely solvable under the additional assumption that the flow is incompressible (solenoidal).

\subsection{Solenoidal injectivity}

We will now prove solenoidal injectivity in two dimensions by making use of the Pestov identity.

Let $\Omega\subset\R^2$ be a bounded, smooth, and strictly convex domain.
A vector field $f\colon\Omega\to\R^2$ can be regarded as a function~$\tilde f$ on~$S\R^2$ as $\tilde f(x,v)=f(x)\cdot v$.
We can define the integral function in two ways, by considering~$f$ as a function on~$S\Omega$ (see~\eqref{eq:uf-def} for a definition of~$u^{\tilde f}$ in terms of~$\tilde f$) or by using an integral formula like~\eqref{eq:vf-xrt}:
\begin{equation}
\label{eq:uf-def-vf}
u^f(x,v)
=
\int_0^{\tau(x,v)}f(x+tv)\cdot v\dd t.
\end{equation}
These two approaches lead to exactly the same function: $u^f=u^{\tilde f}$.
(These notes attempt to distinguish the vector field~$f$ on~$\Omega$ from the function~$\tilde f$ on~$S\Omega$ consistently, but be prepared for failures.)

One way to see the difference between~$f$ and~$\tilde f$ is relocation of complexity.
Our original~$f$ lives on the simpler domain~$\Omega$ but it is vector-valued.
The new function~$\tilde f$ lives on the more complicated domain~$S\Omega$ but is scalar-valued.
We can choose simplicity on either the domain or the target, but not both.

We assume that $\xrt f=0$.
An inspection of the proof of lemma~\ref{lma:u-smooth} shows that $u^f\in C_c^\infty(S\Omega)$ also in the case of vector fields.
The fundamental theorem of calculus of exercise~\ref{ex:ftc-SM} is still valid when~$f$ is seen as a function on~$S\Omega$.
The same proof gives that $Xu^f(x,v)=-\tilde f(x,v)=-f(x)\cdot v$.
However, now~$\tilde f$ does depend on direction, and so typically $V\tilde f\neq0$.
This causes a major change in our proof and result.

\begin{ex}
Let~$f$ be a smooth vector field on~$\R^2$, and define a function $\tilde f\colon S\R^2\to\R$ by $\tilde f(x,v)=f(x)\cdot v$.
Calculate~$V\tilde f(x,v)$ and interpret the result geometrically.
\end{ex}

\begin{ex}
\label{ex:vf-pestov-cancel}
Consider the function~$u^f$ on~$S\Omega$ defined in~\eqref{eq:uf-def-vf} for a smooth and compactly supported vector field~$f$ with $\xrt f=0$.
Show that $\aabs{VXu^f}=\aabs{Xu^f}$.
\end{ex}

\begin{theorem}
\label{thm:vf}
Let $\Omega\subset\R^2$ be a bounded, smooth and convex domain.
Let $f\colon\Omega\to\R^2$ be a compactly supported smooth vector field.
Then the following are equivalent:
\begin{enumerate}
\item $\xrt f=0$, i.e., the vector field integrates to zero over all lines.
\item There is $h\in C_c^\infty(\Omega)$ so that $f=\nabla h$.
\end{enumerate}
\end{theorem}

\begin{proof}
It was observed in section~\ref{sec:non-u-pot} that $\xrt(\nabla h)=0$, so it remains to prove the other direction.

As discussed above, the integral function~$u^f$ is compactly supported and smooth, so we may apply the Pestov identity:
\begin{equation}
\aabs{VXu^f}^2
=
\aabs{XVu^f}^2
+
\aabs{Xu^f}^2.
\end{equation}
By exercise~\ref{ex:vf-pestov-cancel}, this leads to
\begin{equation}
0
=
\aabs{XVu^f}^2.
\end{equation}
This implies that the function $XVu^f\in C_c^\infty(S\Omega)$ is identically zero.

Now it follows from exercise~\ref{ex:XV} that there is a scalar function $h\in C^\infty(\bar\Omega)$ so that $u^f=-\pi^*h$ or, in other words, $u^f(x,v)=-h(x)$.
The minus sign is just a matter of convenience so that another sign gets cancelled later.
As $u^f\in C_c^\infty(S\Omega)$, we must have $h\in C_c^\infty(\Omega)$ --- both vanish in a neighborhood of the boundary.

The last remaining step is to show that from $Xu^f=-\tilde f$ it follows that $f=\nabla h$.
The details are left as exercise~\ref{ex:Xu=dh}.
\end{proof}

\begin{ex}
\label{ex:Xu=dh}
Complete the proof above by showing that $f=\nabla h$.
\end{ex}

In the proof above we needed to produce a potential~$h$ for the vector field~$f$.
Finding the potential was easy, proving that it was a scalar was hard; this method of first relaxing one's definition and then restricting the conclusion at the end is not an uncommon method.
The potential turned out to be essentially the integral function~$u^f$.
This is a function on the sphere bundle~$S\Omega$, and the core of the proof was the conclusion that $u^f=-\pi^*h$ for some scalar function~$h$.

\begin{ex}
Let~$f_0$ be a scalar function and~$f_1$ a vector field.
Their sum is not a very reasonable object at first, and it can be considered just as a formal sum.
How can you consider $f=f_0+f_1$ as a function on~$S\R^n$?
How should we define $\xrt f$?
What does reversing orientation of~$\gamma$ do to~$\xrt f(\gamma)$?

Assume now that~$f_0$ and~$f_1$ are smooth and compactly supported.
Using previously obtained results, argue why $\xrt(f_0+f_1)=0$ implies that $f_0=0$ and $f_1=\nabla h$ for some $h\in C_c^\infty(\R^n)$.
(It is possible to use the Pestov identity to prove results like this, but here it is easier to study orientation reversals and apply theorems~\ref{xrtthm:pestov} and~\ref{thm:vf}.)
\end{ex}

Let us then see what this solenoidal injectivity result means for reconstructing a vector field from data.
If~$f$ and~$g$ are vector fields (or sums of scalars and vector fields) and $\xrt f=\xrt g$, then there is a scalar potential~$h$ vanishing at the boundary so that $f=g+\nabla h$ (and the scalar parts of~$f$ and~$g$ coincide).

We chose to use the Pestov identity, but it is not the only way to prove this statement in a Euclidean space.
We remark that the same proof works for non-positively curved Riemannian manifolds of dimension two with strictly convex boundary.

Let us see what happens in one dimension.
As we have seen before, the X-ray transform is not injective on one dimension scalar functions.
There are non-trivial continuous functions $f\colon[0,1]\to\R$ so that the X-ray transform vanishes.

In one dimension scalar functions and vector fields are essentially the same object, but the solenoidal injectivity requires less than full injectivity.
The conclusion might be surprising: we have no injectivity for scalars, but we do have solenoidal injectivity for vector fields.

\begin{ex}
Consider the one-dimensional set $(0,1)$ or its closure.
What does it mean if the X-ray transform is solenoidally injective on this space?
Prove this solenoidal injectivity.
\end{ex}

\begin{bex}
Use the tools of section~\ref{sec:torus} to prove solenoidal injectivity on the torus~$\T^n$, $n\geq2$, for smooth vector fields.
You can write a vector field on~$\T^n$ as a function $f\colon\T^n\times\R^n\to\C$ which is linear in the second variable.
You will need the lemma that if a linear function $\phi\colon\R^n\to\C$ vanishes in all directions orthogonal to $k\in\R^n$, then there is $a\in\C$ so that $\phi(v)=ak\cdot v$.
It may help (or confuse) to look at~\cite{I:torus} where this is done in more generality.

As before, solenoidal injectivity on a domain $\Omega\subset\R^2$ follows.
What changes in this transition from a result on the torus to a result on the plane in comparison to the scalar case?
\end{bex}

\subsection{A cohomological alternative}

There is an alternative proof of theorem~\ref{thm:vf} based on cohomology.
Namely, we need to know that the first de Rham coholomogy group of the Euclidean space is trivial, which is the content of the Poincar'e lemma.
We give a concise proof without dwelling on the new concepts mentioned.

\begin{proof}[Second proof of theorem~\ref{thm:vf}]
We think of the vector field~$f$ as a one-form in the whole plane, supported in~$\Omega$.\footnote{This is not necessary, but makes the language a bit lighter. Corresponding operations for vector fields are available.}

Take any line~$\gamma$ that meets~$\Omega$.
For $h>0$, let~$\gamma_h$ be a shifted and reversed copy of~$\gamma$ so that the distance between the two lines is~$h$.
Form a rectangle~$A^\gamma_h$ by taking these two lines and joining them with two length~$h$ segments outside~$\Omega$.
This construction is illustrated in figure~\ref{fig:cohomology-loop}.

\begin{figure}[t]
    \centering
    \includegraphics[scale=0.4,trim={3cm 0 3cm 0},clip,page=26]{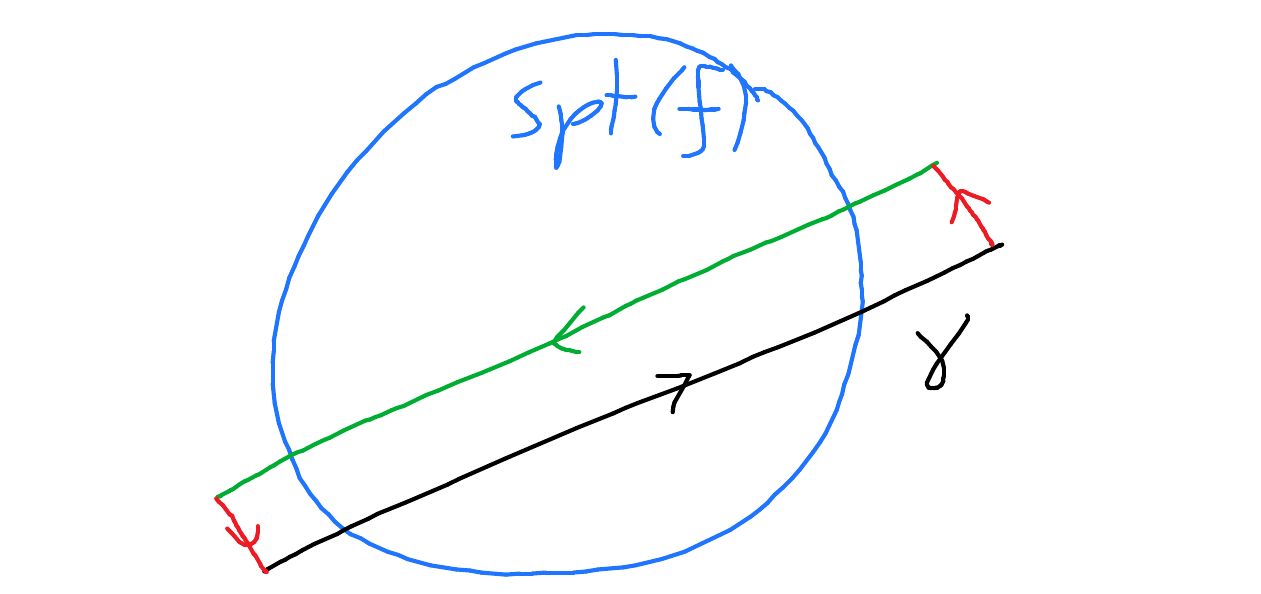}
    \caption{The original line $\gamma$ and \viher{its reversed copy~$\gamma_h$} are used to construct a rectangular loop in the plane. Instead of continuing the lines to infinity, they are joined together with \puna{short line segments}. This has no effect on the line integrals, as the manipulation happens outside \sini{the support of the vector field or one-form~$f$}.}
    \label{fig:cohomology-loop}
\end{figure}

The integral of~$f$ vanishes over~$\gamma$ and~$\gamma_h$ by assumption and these lines differ from the boundary of~$A_h^\gamma$ only outside the support of the vector field.
Therefore~$f$ integrates to zero over this boundary.
Using Stokes' theorem, we find that
\begin{equation}
\int_{A_h}
\der f
=
\int_{\partial A_h^\gamma} f
=
0.
\end{equation}
The exterior derivative~$\der f$ is a 2-form, and it may be identified with the scalar function~$\star\der f$ using the Hodge star~$\star$.

Letting $h\to0$ collapses the rectangle to~$\gamma$, so we find
\begin{equation}
0
=
\frac{1}{h}
\int_{A_h}
\star\der f
=
\xrt(\star\der f)(\gamma)
.
\end{equation}
This is true for any~$\gamma$, so $\xrt(\star\der f)=0$.
By any of the multiple theorems we have proven for scalar functions, it follows that $\star\der f=0$ and thus that $\der f=0$.
The differential form~$f$ is thus closed.
By the Poincar\'e lemma the closed differential form~$f$ is exact, meaning that $f=\der h$ for a scalar function.
Adding a constant to~$h$ so that it vanishes at some point outside~$\Omega$ we get it to vanish everywhere outside~$\Omega$.
Thus $f=\der h$ with~$h$ compactly supported in~$\Omega$, as claimed.
\end{proof}

\subsection{Higher order tensor fields}

We have studied X-ray tomography for symmetric covariant tensor fields of order~$0$ (scalar functions) and~$1$ ((co)vector fields).
One can study the same problem for tensor fields of any order $m\in\N$.

When $m=0$, the left-hand side of the Pestov identity vanishes.
When $m=1$, the term on the left exactly cancels a term on the right.
When $m\geq2$, the term on the left is typically larger than the corresponding one on the right, so our idea of proof no longer works as such.
An important new ingredient for $m\geq2$ is to write $u^f\in C^\infty(S\Omega) = C^\infty(\Omega\times \Sphere^1)$ as a Fourier series on $\Sphere^1=\T^1$.
We will not pursue this here, but some more details will be given in section~\ref{sec:outlook-tt}.

\qa

\section{The Fourier transform}
\label{sec:ft}

The purpose of this section is to build tools, and we will not be touching the X-ray transform at all.

\subsection{A general view to Fourier transforms}

Previously we studied the Fourier transform on a torus~$\T^n$.
The Fourier transform took a function on the torus~$\T^n$ to a function on the lattice~$\Z^n$, and the inverse Fourier transform did the opposite.
One could in fact define a the Fourier transform on the lattice, and that would turn out to be essentially the same as the inverse Fourier transform for the torus.

In this section we will study the Fourier transform on~$\R^n$.
It will take a function $\R^n\to\C$ to another function $\R^n\to\C$, and the inverse transform is very similar to the transform itself.
Before going any deeper into this, we will look at the Fourier transforms in more generality to see that the two Fourier transforms in this course are merely two special cases of a far more general structure.

Let~$G$ be a topological group.
It means that it is a topological space and a group so that the group operations $G\to G$, $x\mapsto x^{-1}$ and $G\times G\to G$, $(x,y)\mapsto xy$ are continuous.
We assume that~$G$ is abelian and locally compact; these assumptions substantially help give the theory some structure.

Let~$\hat G$ denote the set of all continuous homomorphisms $G\to \Sphere^1$.
Here $\Sphere^1\subset\C$ is considered as the multiplicative group of unit complex numbers.
Elements of~$\hat G$ are called characters of~$G$.

Let us compare the definitions of~$\hat G$ for a locally compact abelian group~$G$ and the dual~$E^*$ of a topological real vector space~$E$:
\begin{equation}
\begin{split}
\hat G
&=
\{
\phi\colon G\to\Sphere^1;
\phi
\text{ is a continuous homomorphism}
\},
\\
E^*
&=
\{
\phi\colon E\to\R;
\phi
\text{ is a continuous linear map}
\}.
\end{split}
\end{equation}
The two concepts are clearly analogous.
The map~$\phi$ has to be a morphism of the relevant category, respecting the topology and other structures.
The biggest difference is in the choice of the ``reference space''; one uses $\Sphere^1=U(1)=\T^1=\R/2\pi\Z$ and the other one uses~$\R$.

The set~$\hat G$ can be endowed with a group structure by pointwise multiplication of characters: $(\alpha\beta)(x)=\alpha(x)\beta(x)$.
The set~$\hat G$ is a set of functions, and it can be equipped with the topology of locally uniform convergence.
These structures make~$\hat G$ into a topological group, just like~$G$ itself.

What is important is that~$\hat G$ is also a locally compact abelian group and that~$\hat{\hat{G}}$ is naturally isomorphic to~$G$.
This result is known as the Pontryagin duality theorem and~$\hat G$ is called the dual group of~$G$.

The Fourier transform takes a function on~$G$ into a function on~$\hat G$, and the inverse Fourier transform reverses this.
More precisely, the Fourier transform of $f\colon G\to\C$ is the function $\ft f\colon\hat G\to\C$ defined by
\begin{equation}
\ft f(\alpha)
=
\int_G\alpha(x)f(x)\dd x
\end{equation}
when this integral exists, possibly with a normalization constant or complex conjugation of the character.
The integral is with respect to a Haar measure, a translation invariant Radon measure.
The Haar measure is unique up to a multiplicative constant.
In light of the duality theorem, it is not surprising that the inverse Fourier transform for~$G$ resembles the Fourier transform for~$\hat G$.
The Fourier transform requires a measure on the underlying space, and that is the Haar measure.
When~$G$ and~$\hat G$ are equipped with compatible Haar measures, the~$L^2$ theory (and much more) of Fourier transforms can be extended to any locally compact abelian groups.

Let us make the dual groups a little more concrete with examples.
The dual group of~$\T^n$ is~$\Z^n$ and vice versa, and this we encountered earlier with Fourier series.
The dual group of~$\R^n$ is~$\R^n$ itself, and this we will study now.
The measures on~$\T^n$ and~$\R^n$ are the Lebesgue measures, and the one on the lattice is the counting measure.

\begin{ex}
\label{ex:character1}
An element $k\in\Z^n$ can be identified with a character $\chi_k\colon\T^n\to \Sphere^1$ by $\chi_k(x)=e^{ik\cdot x}$.
(Or $\chi_k(x)=e^{-ik\cdot x}$ if you prefer so; flipping the sign makes the Fourier transform look more familiar. One can also choose to take a complex conjugate of the character. There is no way to make all the sings perfectly convenient.)

Show that for any $k\in\Z^n$ the corresponding character~$\chi_k$ is indeed a well defined and continuous homomorphism.
Recall that $\T^n=\R^n/2\pi\Z^n$.

How can you identify a point~$x$ on the torus~$\T^n$ with a character $\psi_x\in\widehat{\Z^n}$?
No need to prove anything; just give the formula.
\end{ex}

\begin{bex}
Let us prove that indeed
$\widehat{\T^n}\approx\Z^n$,
$\widehat{\Z^n}\approx\T^n$, and
$\widehat{\R^n}\approx\R^n$.
Coming up with characters identifiable with the desired space is simple enough, especially after exercise~\ref{ex:character1}.
In this problem it suffices to prove that all characters are of the desired form.

Show first that the general statement for $n\geq2$ follows once the results are known for $n=1$.
Then prove the results when $n=1$.
\end{bex}

Note that these characters were used in the formulas for the Fourier transform and its inverse on the torus.
This is how Fourier transforms work in general, by integrating a function against a character.

If~$G$ is not abelian, then~$\hat G$ should be replaced with the set of all (equivalence classes of) irreducible representations of~$G$.
This coincides with the dual group in the abelian case since irreducible complex representations of abelian groups are one-dimensional.
Moreover, one-dimensional representations coincide with their characters, so the characters introduced here are the same as the representation theoretic characters.
Fourier analysis on non-abelian groups is possible via representation theory.

Finally, we remark that there are several different conventions for the Fourier transform on a torus or a Euclidean space.
The differences concern the placement of factors of~$2\pi$.
It is impossible to get completely rid of the factors.

\subsection{The Fourier transform on a Euclidean space}

It is typical to call the Fourier transform on a torus the Fourier series and the one on a Euclidean space the Fourier transform.
Fourier analysis on other groups is much rarer.

The Fourier transform of a function $f\colon\R^n\to\C$ is $\ft f\colon\R^n\to\C$ defined by
\begin{equation}
\ft f(\xi)
=
\int_{\R^n}e^{-i\xi\cdot x}f(x)\dd x
\end{equation}
whenever this integral makes sense.
Again, we are purposely vague since the definition can be extended to various classes of functions or distributions.

\begin{theorem}
\label{thm:ft}
The Fourier transform is a bijection $\ft\colon L^2(\R^n)\to L^2(\R^n)$, given by
\begin{equation}
\ft f(\xi)
=
\int_{\R^n}e^{-i\xi\cdot x}f(x)\dd x
\end{equation}
for $f\in L^1(\R^n)\cap L^2(\R^n)$ and extended by continuity to the rest of~$L^2(\R^n)$.

The inverse Fourier transform $\ft^{-1}\colon L^2(\R^n)\to L^2(\R^n)$ is given by
\begin{equation}
(\ft^{-1}f)(x)
=
(2\pi)^{-n}
\int_{\R^n}e^{i\xi\cdot x}f(\xi)\dd \xi,
\end{equation}
interpreted in a suitable limiting sense when $f\notin L^1(\R^n)$.
The Fourier transform is unitary in the sense that
\begin{equation}
\int_{\R^n}\overline{g(x)}f(x)\dd x
=
(2\pi)^{-n}
\int_{\R^n}\overline{\ft g(\xi)}\ft f(\xi)\dd \xi.
\end{equation}
\end{theorem}

Again, the proof will be omitted.

\begin{ex}
What is the relation between~$\aabs{f}_{L^2}$ and~$\aabs{\ft f}_{L^2}$?
\end{ex}

To simplify matters, we will apply the Fourier transform to compactly supported continuous functions.
What we need to know is that $\ft f=0$ implies $f=0$.
Under additional assumptions very little information on~$\ft f$ is needed to conclude that $f=0$, and we will study this next.

\subsection{A Paley--Wiener theorem}

A general and important phenomenon in Fourier analysis is the correspondence between decay and regularity.
Fast decay of~$f(x)$ as $\abs{x}\to\infty$ corresponds to high regularity of~$\ft f$ and vice versa.
For a famous example, the Schwartz space contains by definition functions which have high regularity (infinitely differentiable) and fast decay (all derivatives vanish faster than $\abs{x}^{-N}$ for any $N\in\N$), and the Fourier transform of the Schwartz function space is precisely the space itself.

We will study the ultimate form of decay at infinity: compact support.
This should lead to very high regularity, and that turns out to be the case.
Our theorem in this subsection is a version of the Paley--Wiener theorem.

\begin{definition}
A function $f\colon\Omega\to\C$ defined on an open set $\Omega\subset\R^n$ is called real analytic if it is smooth and for every point $x\in\R^n$ there is $r>0$ so that the Taylor series of~$f$ around~$x$ converges pointwise to~$f$ in $B(x,r)$.
\end{definition}

The mode of convergence is not very important.
If we had required uniform convergence, we would have ended up with an equivalent definition.
But the feature that the Taylor series converges to the function itself instead of possibly something else is crucial; see exercise~\ref{ex:non-analytic}.

In complex analysis one can define analyticity in a similar fashion by demanding that a complex Taylor series converges to the function in a small neighborhood of any point.
This turns out to be equivalent with complex differentiability (the existence of the derivative as a limit of a difference quotient).
When working over the reals this is no longer the case; real analyticity is far stronger than real differentiability.

\begin{ex}
\label{ex:analytic-open}
Show that if a real analytic function $f\colon\R^n\to\C$ vanishes in a non-empty open set $U\subset\R^n$, then~$f$ is identically zero.
If you are unfamiliar with multidimensional Taylor series, feel free to take $n=1$; all the essentials are already present in one dimension.

See exercise~\ref{ex:taylor} for two versions of the Taylor polynomial.
\end{ex}

\begin{ex}
\label{ex:non-analytic}
Define the function $f\colon\R\to\R$ by
\begin{equation}
f(x)
=
\begin{cases}
0, & x\leq0 \\
\exp(-1/x), & x>0.
\end{cases}
\end{equation}
Consider it known that $f\in C^\infty(\R)$.
You may also consider it known that the composition of two analytic functions is analytic.

Explain and justify (or prove):
For any $x\in\R$ the Taylor series of~$f$ at~$x$ converges in some open neighborhood of~$x$.
However,~$f$ is not real analytic.
\end{ex}

The main result of this section is this:

\begin{theorem}[Paley--Wiener theorem]
\label{thm:pw}
The Fourier transform of a compactly supported function in $L^1(\R^n)$ is real analytic.
\end{theorem}

The converse implication is also true, and often considered a part of the Paley--Wiener theorem, but it is unimportant for our needs here.

In light of exercise~\ref{ex:non-analytic}, it is not enough to estimate the derivatives to establish a positive radius of convergence for the Taylor series.
We really need to show that the limit is correct.

Let us collect some tools before the proof.
First, recall a lemma from measure and integration theory:

\begin{lemma}
\label{lma:d-int}
Fix integers $1\leq j\leq n$.
Consider a function $g\colon\R^n\times\R^n\to\C$.
Suppose that for every $y\in\R^n$ we have $g(\dummy,y)\in L^1(\R^n)$,
that for every $x\in\R^n$ and $y\in\R^n$ the partial derivative~$\partial_{y_j}g(x,y)$ exists,
and that there is a function $h\in L^1(\R^n)$ so that $\abs{\partial_{y_j}g(x,y)}\leq h(x)$ for all $x\in\R^n$ and $y\in\R^n$.
Then the function
\begin{equation}
G(y)
=
\int_{\R^n}g(x,y)\dd x
\end{equation}
has the partial derivative~$\partial_{y_j}G(y)$ everywhere and
\begin{equation}
\partial_{y_j}G(y)
=
\int_{\R^n}\partial_{y_j}g(x,y)\dd x,
\end{equation}
where the last integral is a well-defined Lebesgue integral.
\end{lemma}

\begin{proof}[Proof sketch]
Fix any~$y$.
Let~$e_j$ be the $j$th basis vector of $\R^n$ and define $f_k(x)=k[g(x,y+e_j/k)-g(x,y)]$.
These are measurable functions which converge pointwise to $\partial_{y_j}g(x,y)$ as $k\to\infty$.
Now the difference quotient of~$G$ at~$y$ has the expected limit due to the intermediate value theorem and Lebesgue's dominated converge theorem.
\end{proof}

\begin{ex}
\label{ex:ft-der-poly}
Suppose $f\in L^1(\R^n)$ vanishes outside a compact set~$K$.
Denote by~$f_j$ the function $f_j(x)=x_jf(x)$.
Show that $\partial_{\xi_j}\ft f(\xi)=-i\ft f_j(\xi)$ and the partial derivative exists everywhere.
\end{ex}

Similarly, one can find that for any vector $v\in\R^n$ one has
\begin{equation}
\label{eq:ft-der}
v\cdot\nabla\ft f(\xi)
=
-i\ft f_v(\xi),
\end{equation}
where $f_v(x)=(v\cdot x) f(x)$.

\begin{ex}
\label{ex:ft-smooth}
Suppose $f\in L^1(\R)$ vanishes outside a compact set~$K$.
Argue that $\ft f\in C^\infty$.
(The same result holds in~$\R^n$ for any~$n$.)
\end{ex}

\begin{ex}
\label{ex:ft-translate}
Suppose $f\in L^1(\R^n)$ vanishes outside a compact set~$K$.
Show that $\ft f(\xi)=\ft g(\zeta)$, where $g(x)=e^{i(\zeta-\xi)\cdot x}f(x)$.
\end{ex}

\begin{proof}[Proof of theorem~\ref{thm:pw}]
Let $f\colon\R^n\to\C$ be compactly supported and integrable.
Repeated application of equation~\eqref{eq:ft-der} gives
\begin{equation}
\label{eq:vv11}
(v\cdot\nabla)^m\ft f(\xi)
=
(-i)^m\ft(\mu_v^mf)(\xi),
\end{equation}
where~$\mu_v^m$ is the multiplication operator defined by $(\mu_v^mf)(x)=(v\cdot x)^m f(x)$.
Notice that~$v\cdot\nabla$ is a derivative in the direction $v\in\R^n$, and it makes sense to take the~$m$th iterated derivative.
By exercise~\ref{ex:ft-smooth} all these derivatives exist.

Take any $\rho\in\C^n$.
We have $e^{\rho\cdot x}=\sum_{k\in\N}\frac1{k!}(\rho\cdot x)^k$.
It is clear that each partial sum is dominated by $\sum_{k\in\N}\frac1{k!}\abs{\rho\cdot x}^k=e^{\abs{\rho\cdot x}}$.
This majorant is uniformly bounded when $x\in K$ and~$\abs{\rho}$ is bounded.

Therefore, by the dominated convergence theorem and exercise~\ref{ex:ft-translate} with $\zeta=\xi_0$,
\begin{equation}
\begin{split}
\ft f(\xi)
&=
\ft g(\xi_0)
\\&=
\ft\left(
\sum_{k=0}^\infty\frac{i^k}{k!}\mu_{\xi-\xi_0}^kf
\right)
(\xi_0)
\\&=
\sum_{k=0}^\infty\frac{i^k}{k!}\ft(\mu_{\xi-\xi_0}^kf)(\xi_0).
\end{split}
\end{equation}
Applying~\eqref{eq:vv11} to each term gives
\begin{equation}
\ft f(\xi)
=
\sum_{k=0}^\infty\frac1{k!}((\xi-\xi_0)\cdot\nabla)^k\ft f(\xi_0).
\end{equation}
This is precisely the Taylor series of~$\ft f$ about the point~$\xi_0$ evaluated at~$\xi$; see also exercise~\ref{ex:taylor}.
We have shown that this series converges to~$\ft f(\xi)$ as desired.
\end{proof}

\begin{ex}
We proved above that the Taylor series of~$\ft f$ at any $\xi_0\in\R^n$ converges to~$\ft f$.
What can you deduce about the radius of convergence?
\end{ex}

\begin{ex}
\label{ex:taylor}
Let us compare two different representations of higher dimensional Taylor polynomials.

Suppose $f\in C^\infty(\R^n)$, $m\in\N$, and $v\in\R^n$.
Show that
\begin{equation}
\sum_{k=0}^m\frac1{k!}(v\cdot\nabla)^kf(0)
=
\sum_{\abs{\alpha}\leq m}\frac1{\alpha!} v^\alpha\partial^\alpha f(0).
\end{equation}
Here $\alpha\in\N^n$ is a multi-index.

In the case $n=3$ the notations mean $v^\alpha=v_1^{\alpha_1}v_2^{\alpha_2}v_3^{\alpha_3}$,
$\alpha!=\alpha_1!\alpha_2!\alpha_3!$,
$\partial^\alpha f=\partial_{x_1}^{\alpha_1}\partial_{x_2}^{\alpha_2}\partial_{x_3}^{\alpha_3}f$,
and
$\abs{\alpha}=\alpha_1+\alpha_2+\alpha_3$.
The case of general~$n$ can be inferred from these, and in the case $n=1$ the notations and the Taylor polynomial should look familiar.

(If~$f$ is real analytic and~$v$ is within the radius of convergence at the origin, both sides equal~$f(v)$ in the limit $m\to\infty$. It is worth noting that the Taylor series can be formally written as $f(v)=(e^{v\cdot\nabla}f)(0)$.)
\end{ex}

For the fun of it, let us see a couple of examples of real analytic functions.

\begin{ex}
Calculate the Fourier transform of the characteristic function of the cube $[-1,1]^3\subset\R^3$.
Make sure the function is defined everywhere.
\end{ex}

\begin{ex}
Calculate the Fourier transform of the characteristic function of the unit ball $B\subset\R^3$.
Make sure the function is defined everywhere.
\end{ex}

\qa

\section{The normal operator}
\label{sec:normal}

In this and the next section we will give our fifth injectivity proof based on the normal operator of~$\xrt$.

\subsection{Why care about a normal operator}

\begin{definition}
\label{def:normal}
The normal operator of a bounded linear operator $A\colon E\to F$ between complex or real Hilbert spaces is $A^*A\colon E\to E$, where $A^*\colon F\to E$ is the adjoint of~$A$.
\end{definition}

Let us discuss this definition and the concepts appearing in it in more detail.
We will work over~$\C$, but there is no significant difference to the real version.
The adjoint is defined to be the operator that satisfies
\begin{equation}
\ip{y}{Ax}_F
=
\ip{A^*y}{x}_E
\end{equation}
for all $x\in E$ and $y\in F$.

\begin{ex}
Using this definition, show that the adjoint~$A^*$ is unique if it exists.
\end{ex}

\begin{ex}
Using the definition, show carefully that $(A^*)^*=A$.
\end{ex}

Existence of the adjoint follows from the Riesz representation theorem which characterizes the dual of a Hilbert space.
Namely, for any~$y$, the mapping $x\mapsto \ip{y}{Ax}_F$ is in~$E^*$, and by the representation theorem there is $z\in E$ so that $\ip{y}{Ax}_F=\ip{z}{x}_E$.
It is easy to check that this~$z$ has to depend linearly on~$y$.
This gives rise to a linear operator~$A^*$ which maps~$y$ to~$z$.
It also turns out that~$A^*$ is bounded if~$A$ is.

\begin{ex}
Show that $\aabs{A^*}=\aabs{A}$ in the operator norm.
\end{ex}

It is very convenient to work with self-adjoint operators.
An operator~$A$ is called self-adjoint if~$A^*=A$.
For the operator~$A$ to be self-adjoint, we must have $E=F$, but this is not always the case.
Therefore it is convenient to replace~$A$ with its normal operator.
Self-adjointness in itself is convenient, but the normal operator tends to be nicer than the original operator in other ways as well.

\begin{ex}
Using the definition of an adjoint given above, show that the normal operator of any bounded linear operator between Hilbert spaces is self-adjoint.
\end{ex}

In our case~$A$ is the X-ray transform.
Now~$E$ is a function space over~$\R^n$ and~$F$ is a function space over the set of all lines in~$\R^n$.
There is no natural way to identify the two spaces, so we will study the normal operator of the X-ray transform.
(We did use such an identification in the plane, but it depended on the choice of origin.)
Our goal is then to show that~$\xrt^*\xrt$ is injective, from which it follows that~$\xrt$ is injective; see exercise~\ref{ex:left-inv}.

\begin{ex}
The injectivity of the X-ray transform is related to the surjectivity of the adjoint.
Let us explore this in a simplified setting.
Let $A\colon H\to H$ be a continuous linear operator in a Hilbert space~$H$.
Show that if~$A^*$ is surjective, then~$A$ is injective.
\end{ex}

\subsection{Measures on spheres and sets of lines}
\label{sec:measures}

The sphere~$\Sphere^{n-1}$ has a canonical measure.
We will not define it, but we will give some descriptions, some of which count as definitions for the reader with suitable knowledge of measure theory or differential geometry.
We will give a similar treatment to the set of lines soon.

In fact, a regular Borel measure is uniquely determined by the integrals of functions in~$C_c$, so our descriptions do secretly constitute a definition of a measure.

The sphere inherits a metric from~$\R^n$.
The metric allows us to define Hausdorff measures of any dimension, and the natural one has dimension $n-1$.

The sphere~$\Sphere^{n-1}$ is also a smooth manifold of dimension $n-1$.
It inherits a Riemannian metric from~$\R^n$, and the Riemannian metric induces a Riemannian volume form.
This leads to the same measure as the Hausdorff approach.
The Riemannian metric gives rise to a metric (distance along great circles in this case), and this gives rise to the same Hausdorff measure as the Euclidean (chordal) metric.

We will denote the measure on the sphere by~$S$.
The most important property for us is a spherical Fubini's theorem.
For $f\in C_c(\R^n)$, we have
\begin{equation}
\int_{\R^n}f(x)\dd x
=
\int_0^\infty\int_{\Sphere^{n-1}}f(r\omega)r^{n-1}\dd S(\omega)\dd r.
\end{equation}
This property could also be used as a definition of the measure~$S$ on~$\Sphere^{n-1}$.

Let us denote the set of all straight lines in~$\R^n$ by~$\Gamma$.
The lines themselves are easy to visualize, but the set of lines is a somewhat less intuitive geometrical object.
There is a natural structure of a Riemannian manifold on~$\Gamma$
,
and that gives rise to other structures as well: topology, metric, measure, smooth structure\dots
The manifold structure is a bit tricky and unnecessary for us, but we will need to understand the structure and measure of~$\Gamma$.
See figure~\ref{fig:line-bundle-fiber} and the surrounding discussion for how the space of oriented lines in~$\R^n$ is naturally identified with~$T\Sphere^{n-1}$, the tangent bundle of the unit sphere.

We will write lines as $x+v\R=\{x+vt;t\in\R\}\subset\R^n$ for $x\in\R^n$ and $v\in \Sphere^{n-1}$.
This parametrization is redundant --- each line is counted several times --- but in the set~$\Gamma$ every line is only included once.

\begin{ex}
Let $x_1,x_2\in\R^n$ and $v_1,v_2\in \Sphere^{n-1}$.
When is $x_1+v_1\R=x_2+v_2\R$?
\end{ex}

We will give some more details on the structure of~$\Gamma$ later in connection with sphere bundles.
For now we rely on intuition and acknowledge that the space~$C_c(\Gamma)$ of continuous and compactly supported functions $\Gamma\to\C$ is not rigorously defined.
We point out that although the lines themselves are not compact, there are non-trivial compact sets in the space of lines.

Let us then describe the measure~$\mu$ on~$\Gamma$.
For any $v\in \Sphere^{n-1}$, we denote\footnote{Previously~$v^\perp$ meant a rotated version of~$v$, but now it stands for the subspace orthogonal to~$v$. See exercise~\ref{ex:2v-perps}.} by $v^\perp\coloneqq\{x\in\R^n;x\cdot v=0\}$ the orthogonal complement of the space spanned by~$v$.
The space~$v^\perp$ can be identified with~$\R^{n-1}$, and we denote the measure there by~$\h^{n-1}$ (the $(n-1)$-dimensional Hausdorff measure).
The measure~$\mu$ is defined so that the integral of $g\in C_c(\Gamma)$ is
\begin{equation}
\int_\Gamma g(\gamma)\dd\mu(\gamma)
=
\int_{\Sphere^{n-1}}\int_{v^\perp}g(x+v\R)\dd\h^{n-1}(x)\dd S(v).
\end{equation}
We are here representing a line by a direction $v\in\Sphere^{n-1}$ and a point $x\in v^\perp$. 
In this representation the same line appears twice in the integral --- in both orientations.
The same formula can be used for the space of oriented lines as well.
The double counting could be removed by replacing~$\Sphere^{n-1}$ with its antipodal quotient (the real projective space of dimension $n-1$), but multiple counting of finite order is not an issue for our purposes.
More precisely, one can define an equivalence relation~$\sim$ on~$\Sphere^{n-1}$ identifying antipodal points; the antipodal quotient is $\Sphere^n/{\sim}$.

\begin{ex}
\label{ex:2v-perps}
In section~\ref{sec:2D-SM} we used~$v^\perp$ to denote something else.
How are these two~$v^\perp$s related in the two-dimensional setting?
\end{ex}

\begin{ex}
\label{ex:translation-invariance-mu}
For a vector $a\in\R^n$, define the translation operator $\phi_a\colon\Gamma\to\Gamma$ by $\phi_a(\gamma)=a+\gamma$.
Show that for $g\in C_c(\Gamma)$ we have $\int_\Gamma g\dd\mu=\int_\Gamma g\circ\phi_a\dd\mu$ for any~$a$.
\end{ex}

\begin{ex}
Exercise~\ref{ex:translation-invariance-mu} shows that the measure~$\mu$ is translation invariant.
It is also rotation invariant.
What does this property mean?
Write the statement in terms of the integral of an arbitrary function like above.
Then prove the statement.
\end{ex}

In the definition above we chose to take base points of lines in direction~$v$ in the hyperplane~$v^\perp$.
If the hyperplanes were chosen differently, the measure would still be translation invariant, but rotation invariance requires a good choice.
If we were to use a single fixed hyperplane, it would fail to parametrize all the required lines when~$v$ is contained in it.
However, this is a zero measure error.
If the fixed hyperplane is given as~$w^\perp$ for some fixed~$w$, one would need to multiply the Hausdorff measure with $\abs{w\cdot v}$.
Using~$v^\perp$ induces the inconvenience of changing the space of integration depending on~$v$, but the geometrical picture is far clearer.

\subsection{The formal adjoint of the X-ray transform}

Now we are ready to find the formal adjoint of the X-ray transform.
The adjoint~$\xrt^*$ is a convenient operator turning (by composition) the X-ray transform into an operator from a function space to itself.
To find a convenient operator, it suffices to find the formal normal operator.
Applying~$\xrt^*$ to the data~$\xrt f$ can be seen as post-processing the data.

Our Hilbert spaces are~$L^2(\R^n)$ and~$L^2(\Gamma)$.
In the definition of the adjoint, we will not use all~$L^2$ functions --- in fact, it is not important whether the X-ray transform is continuous or well defined $L^2(\R^n)\to L^2(\Gamma)$.
Instead, we will only use functions in~$C_c(\R^n)$ and~$C_c(\Gamma)$ to find the adjoint.
The whole point is to find an operator that ends up behaving nicely, and it does not matter how fishy the method to find the operator is.

We want to use the~$L^2$ inner product, but the adjoint as we defined it does not make sense since the X-ray transform is not continuous $L^2(\R^n)\to L^2(\Gamma)$.
This is why we need a formal adjoint, found by using only the nicer subset~$C_c$ of~$L^2$ over both~$\R^n$ and~$\Gamma$.

\begin{ex}
Let $p\in[1,\infty)$.
Let~$f$ be the characteristic function of the ball $B(0,R)\subset\R^n$.
Show that $\aabs{f}_{L^p(\R^n)}=aR^{n/p}$ and $\aabs{\xrt f}_{L^p(\Gamma)}=bR^{(n+p-1)/p}$ for some constants~$a$ and~$b$ depending on the exponent~$p$ and the dimension~$n$.

Therefore the X-ray transform is not continuous $L^p(\R^n)\to L^p(\Gamma)$ for any $p\in(1,\infty)$.
(It is quite easy to see that the X-ray transform is also discontinuous for $p=\infty$ but is continuous for $p=1$.)
\end{ex}

Recall that in section~\ref{sec:torus} we realized the X-ray transform as a family of self-adjoint operators.
Over the Euclidean space the X-ray transform is not self-adjoint but it is a single operator.
It seems to be impossible to realize the X-ray transform as a single self-adjoint operator in a natural way.

Let us finally turn to finding an explicit formula for the adjoint of the X-ray transform.
Let $f\in C_c(\R^n)$ and $g\in C_c(\Gamma)$.
Then
\begin{equation}
\label{eq:vv7}
\begin{split}
\ip{f}{\xrt^*g}_{L^2(\R^n)}
&=
\ip{\xrt f}{g}_{L^2(\Gamma)}
\\&=
\int_{\Sphere^{n-1}}\int_{v^\perp}
\overline{\xrt f(x+v\R)}
g(x+v\R)
\dd\h^{n-1}(x)
\dd S(v)
\\&=
\int_{\Sphere^{n-1}}\int_{v^\perp}
\int_\R
\overline{f(x+tv)}
g(x+v\R)
\dd t
\dd\h^{n-1}(x)
\dd S(v)
\\&\stackrel{\text{a}}{=}
\int_{\Sphere^{n-1}}
\int_{v^\perp}
\int_\R
\overline{f(x+tv)}
g(x+tv+v\R)
\dd t
\dd\h^{n-1}(x)
\dd S(v)
\\&\stackrel{\text{b}}{=}
\int_{\Sphere^{n-1}}
\int_{\R^n}
\overline{f(y)}
g(y+v\R)
\dd y
\dd S(v)
\\&=
\int_{\R^n}
\overline{f(y)}
\left(
\int_{\Sphere^{n-1}}
g(y+v\R)
\dd S(v)
\right)
\dd y.
\end{split}
\end{equation}

\begin{ex}
Explain the steps a and b in~\eqref{eq:vv7}.
\end{ex}

Here we used Fubini's theorem on the internal direct sum $v\R\oplus v^\perp=\R^n$, expressing the whole space essentially as a product of two subspaces.

This calculation indicates that the formal adjoint is
\begin{equation}
\xrt^*g(x)
=
\int_{\Sphere^{n-1}}
g(x+v\R)
\dd S(v).
\end{equation}
The formal adjoint of the X-ray transform is also known as the back projection operator.

There is a certain kind of duality between points and lines.
It might be more illuminating to describe the situation in words:
\begin{itemize}
\item For $f\in C_c(\R^n)$ and $\gamma\in\Gamma$, the X-ray transform~$\xrt f(\gamma)$ is the integral of~$f(x)$ over all~$x$ for which $x\in\gamma$.
\item For $g\in C_c(\Gamma)$ and $x\in\R^n$, the back projection $\xrt^* g(x)$ is the integral of~$g(\gamma)$ over all~$\gamma$ for which $x\in\gamma$.
\end{itemize}

\noindent
Now that we have found the adjoint, it remains to find the normal operator.

\begin{ex}
Show that for $f\in C_c(\R^n)$ we have
\begin{equation}
\label{eq:xrt-normal}
\xrt^*\xrt f(x)
=
2\int_{\R^n}f(x+y)\abs{y}^{1-n}\dd y.
\end{equation}
This is the (formal) normal operator that we have been looking for.
\end{ex}

\subsection{Convolutions and Riesz potentials}

Now, we ought to show that the normal operator $\xrt^*\xrt\colon C_c(\R^n)\to C(\R^n)$ defined by~\eqref{eq:xrt-normal} is injective.

\begin{ex}
\label{ex:left-inv}
Consider a function $F\colon X\to Y$ between any two sets.
Show that there is a left inverse $F^{-1}_L\colon Y\to X$ so that $F^{-1}_L\circ F=\id_X$ if and only if~$F$ is injective.
(Similarly, invertibility from the right is equivalent with surjectivity, but we do not need this side. In fact, this equivalence for right inverses is equivalent with the axiom of choice, but for left inverses it is not. One-sided inverse functions are typically not unique.)

Suppose we have a left inverse~$A$ for~$\xrt^*\xrt$.
What is a left inverse of~$\xrt$?
\end{ex}

The convolution of two functions $f,h\colon\R^n\to\C$ is the function $f*h\colon\R^n\to\C$ defined by
\begin{equation}
f*h(x)
=
\int_{\R^n}f(x-y)h(y)\dd y
\end{equation}
whenever this integral makes sense.

\begin{ex}
The normal operator is a convolution: $\xrt^*\xrt f=f*h$.
What is the function~$h$?
\end{ex}

\begin{definition}
For $\alpha\in(0,n)$, the Riesz potential~$I_\alpha$ is an integral operator defined by
\begin{equation}
I_\alpha f
=
f*h_\alpha,
\end{equation}
where
\begin{equation}
h_\alpha(x)
=
c_\alpha^{-1}
\abs{x}^{\alpha-n}
\end{equation}
and~$c_\alpha$ is a constant.
\end{definition}

The Riesz representation theorem and the Riesz potential are named after two different people. They were brothers.

To prove injectivity of the X-ray transform, we will show that the Riesz potentials are injective.

\begin{theorem}
\label{thm:riesz-potential}
The Riesz potential $I_\alpha\colon C_c(\R^n)\to C(\R^n)$ is an injection for every $\alpha\in(0,n)$.
\end{theorem}

The proof of a special case of this injectivity property of Riesz potentials is postponed to the next section.
We present the statement here because of its corollary:

\begin{theorem}
\label{xrtthm:riesz}
The X-ray transform is injective on~$C_c(\R^n)$.
\end{theorem}

\begin{ex}
Prove theorem~\ref{xrtthm:riesz} using the results and ideas obtained in this section.
\end{ex}

\subsection{Remarks}

The normal operator depends on the choice of the target space and the inner product on it.
If we parametrized lines redundantly with $\R^n\times \Sphere^{n-1}$, the adjoint of the X-ray transform would take $g\in C_c(\R^n\times \Sphere^{n-1})$ into the function
\begin{equation}
\xrt^*g(x)
=
\int_{\Sphere^{n-1}}\int_\R g(x+tv,v)\dd t\dd S(v).
\end{equation}
This is very similar to what we found before, but there is an additional integral over~$\R$.

If we now try to compute the normal operator, we find that
\begin{equation}
\xrt^*\xrt f(x)
=
\int_{\Sphere^{n-1}}\int_\R \int_R f(x+tv+sv)\dd s\dd t\dd S(v).
\end{equation}
This differs from our earlier normal operator by the factor~$\int_\R\dd t$ which is famously infinite.
This is why redundancy in the parametrization of geodesics is problematic.
We had two-fold redundancy, so we ended up with a factor~$2$ in our normal operator.
It is easier to divide by~$2$ than by~$\infty$.

Let us then see what the effect of changing inner products is.
Let $P\in\R^{n\times n}$ and $Q\in\R^{m\times m}$ be symmetric and positive definite.
Equip~$\R^n$ with the inner product
\begin{equation}
\ip{x}{y}_P
=
x^TPy
\end{equation}
and similarly~$\R^m$ with $\ip{\dummy}{\dummy}_Q$.

Let $A\colon\R^n\to\R^m$ be a linear operator (matrix).
Let $B\colon\R^m\to\R^n$ be the adjoint with respect to these inner products.
That is, suppose
\begin{equation}
\ip{x}{Ay}_Q
=
\ip{Bx}{y}_P
\end{equation}
for all $x\in\R^n$ and $y\in\R^m$.

\begin{ex}
Show that $B=P^{-1}A^TQ$.
Therefore the normal operator is $BA=P^{-1}A^TQA$.
\end{ex}

The conclusion is that changing the inner product can introduce an operator between~$\xrt^*$ and~$\xrt$, where~$\xrt^*$ is understood as the $L^2$-adjoint.
This is popular in X-ray tomography.
The corresponding inversion method is known as filtered back projection.
For this and other practical methods, see~\cite{N}.

\qa

\section{Riesz potentials}
\label{sec:riesz}

In this section we will study the injectivity of Riesz potentials.
We lack the prerequisite theory of distributions to give a precise proof of theorem~\ref{thm:riesz-potential}, but we will discuss two approaches to prove injectivity.
The theory of Riesz potentials is left incomplete, but the interested reader should find enough pointers here for further study towards completion.

\subsection{The Fourier approach}
\label{sec:fourier-riesz}

The first approach makes use of the Fourier transform.

The calculations in this section are heuristic.
It is possible to make rigorous sense of them and give a precise proof of theorem~\ref{thm:riesz-potential}, but we will avoid the technicalities.

Consider a function $f\colon\R^n\to\C$ which we want to reconstruct from $I_\alpha f$ for some $\alpha\in(0,n)$.
As in section~\ref{sec:normal}, denote $h_\alpha(x)=c_\alpha\abs{x}^{\alpha-n}$.
Now $I_\alpha f=f*h_\alpha$.
We will take a Fourier transform of this identity.

\begin{ex}
\label{ex:ft-prod-conv}
Show that if $f,g\in C_c(\R^n)$, then
\begin{equation}
\ft(f*g)(\xi)
=
\ft f(\xi)\ft g(\xi)
\end{equation}
for all $\xi\in\R^n$.
\end{ex}

Exercise~\ref{ex:ft-prod-conv} is a simple calculation using definitions.
The function~$h_\alpha$, however, is not in~$C_c(\R^n)$.
It is locally integrable, but not in any~$L^p$ space.
The same property of convolutions and Fourier transforms does hold in more generality, and we have
\begin{equation}
\label{eq:riesz-ft}
\ft(I_\alpha f)
=
\ft f\cdot\ft h_\alpha
\end{equation}
in the sense of distributions.
Both~$h_\alpha$ and~$I_\alpha f$ are distributions.

Then we want to compute the Fourier transform~$\ft h_\alpha$.
Our definition of the Fourier transform is not applicable, but the definition can be extended to distributions.
With such an extended definition one can calculate that
\begin{equation}
\label{eq:ft-h-alpha}
\ft h_\alpha(\xi)
=
b_\alpha\abs{\xi}^{-\alpha}
\end{equation}
for some constant $b_\alpha>0$.

Now if~$I_\alpha f$ vanishes, then by~\eqref{eq:riesz-ft} also~$\ft f\cdot\ft h_\alpha$ vanishes.
That is, $\abs{\xi}^{-\alpha}\ft f(\xi)=0$ for all~$\xi$.
This implies that~$\ft f$ vanishes, and so $f=0$.
This shows injectivity of~$I_\alpha$.
However, a number of steps were far from rigorous, including dividing by~$\abs{\xi}^{-\alpha}$ on the Fourier side.

This approach also gives an inversion formula for the Riesz potential and thus also the X-ray transform:
\begin{equation}
f
=
\ft^{-1}(\mu_\alpha\ft(I_\alpha f))
,
\end{equation}
where $\mu_\alpha(\xi)=b_\alpha^{-1}\abs{\xi}^\alpha$.

\subsection{The Laplace approach}

The second approach makes use of the Laplace operator.

We found in section~\ref{sec:normal} that the normal operator~$\xrt^*\xrt$ is, up to a multiplicative constant, the Riesz potential~$I_1$.
To show that~$I_1$ is injective, we show that $I_1\circ I_1$ is injective.
To make the argument rigorous, we assume $n\geq3$ and we will also assume more regularity in a moment.
But first, let us see what the operator~$\xrt^*\xrt\xrt^*\xrt$ or~$I_1I_1$ does.

\begin{lemma}
\label{lma:I1I1=I2}
If $f\in C_c(\R^n)$, $n\geq3$, there is a constant $c>0$ so that $I_1I_1f=cI_2f$.
\end{lemma}

In general, the Riesz potentials satisfy $I_\alpha I_\beta=I_{\alpha+\beta}$, but we will not try to prove this in full generality.

\begin{proof}[Proof of lemma~\ref{lma:I1I1=I2}]
First, a simple calculation gives
\begin{equation}
\begin{split}
c_1^2I_1I_1f(x)
&=
c_1
\int_{\R^n}
I_1f(x-y)\abs{y}^{1-n}
\dd y
\\&=
\int_{\R^n}
\left(
\int_{\R^n}
f(x-y-z)
\abs{z}^{1-n}
\dd z
\right)
\abs{y}^{1-n}
\dd y
\\&=
\int_{\R^n}
\left(
\int_{\R^n}
f(x-w)
\abs{w-y}^{1-n}
\dd w
\right)
\abs{y}^{1-n}
\dd y
\\&=
\int_{\R^n}
f(x-w)
\left(
\int_{\R^n}
\abs{w-y}^{1-n}
\abs{y}^{1-n}
\dd y
\right)
\dd w.
\end{split}
\end{equation}
By rotation invariance\footnote{Denote the integral by $\Phi(w)$. Pick any rotation $R\in O(n)$ and change the varialbe of integration in $\Phi(Rw)$ from $y$ to $Ry$ to see that $\Phi(Rw)=\Phi(w)$. Therefore $\Phi$ only depends on the norm of its argument.}, the inner integral is
\begin{equation}
\int_{\R^n}
\abs{w-y}^{1-n}
\abs{y}^{1-n}
\dd y
=
\phi(\abs{w})
\end{equation}
for some function~$\phi$.
When $r>0$, a simple scaling argument (exercise~\ref{ex:scaling-argument}) shows that $\phi(r)=r^{2-n}\phi(1)$.
Therefore
\begin{equation}
c_1^2I_1I_1f(x)
=
\phi(1)
\int_{\R^n}
f(x-w)
\abs{w}^{2-n}
\dd w
=
\phi(1)c_2I_2f(x)
.
\end{equation}
This is the desired conclusion.
\end{proof}

\begin{ex}
Explain why~$\phi(1)$ is a finite positive number.
\end{ex}

\begin{ex}
\label{ex:scaling-argument}
Make the simple scaling argument.
\end{ex}

Now we will turn to inverting~$I_2$.
The inverse operator is simply --- and perhaps surprisingly --- the Laplacian.
In fact, we will show that $-b\Delta I_2f=f$ for a suitable constant $b>0$.
For technical convenience, we assume $f\in C^2_c(\R^n)$ and $n\geq3$.

We have
\begin{equation}
I_2f(x)
=
c_2^{-1}
\int_{\R^n}
f(x-y)
\abs{y}^{2-n}
\dd y.
\end{equation}
Since $f\in C^2_c(\R^n)$ and $y\mapsto\abs{y}^{2-n}$ is locally integrable, 
lemma~\ref{lma:d-int} gives
\begin{equation}
c_2
\Delta I_2f(x)
=
\int_{\R^n}
(\Delta f)(x-y)
\abs{y}^{2-n}
\dd y.
\end{equation}
We split the integral in two parts, integrating separately near the singularity at $y=0$ and far from it.
For any $\eps>0$ (which will be let go to zero later) we have
\begin{equation}
\begin{split}
c_2
\Delta I_2f(x)
&=
\underbrace{
\int_{B(0,\eps)}
(\Delta f)(x-y)
\abs{y}^{2-n}
\dd y
}_{\eqqcolon P(x,\eps)}
\\&\qquad+
\underbrace{
\int_{\R^n\setminus B(0,\eps)}
(\Delta f)(x-y)
\abs{y}^{2-n}
\dd y
}_{\eqqcolon Q(x,\eps)}.
\end{split}
\end{equation}
By a direct computation
\begin{equation}
\abs{P(x,\eps)}
\leq
\max\abs{\Delta f}
\cdot
\int_{B(0,\eps)}
\abs{y}^{2-n}
\dd y
\to
0
\end{equation}
as $\eps\to0$.

\begin{ex}
Verify this limit, either by direct calculation in spherical coordinates or by appealing to local integrability of $y\mapsto\abs{y}^{2-n}$ and absolute continuity of the Lebesgue integral.
\end{ex}

The second integral contains no singularities, and we may integrate by parts.
Let us first recall a more general result:

\begin{ex}
\label{ex:green-laplace}
Let $\Omega\subset\R^n$ be an open set with smooth boundary and denote the exterior unit normal vector by~$\nu$.
Denote the surface measure on~$\partial\Omega$ by~$S$.
Suppose $u\in C^2(\R^n)$ and $v\in C^2_c(\R^n)$.
Show that
\begin{equation}
\begin{split}
\int_\Omega u(x)\Delta v(x) \dd x
&=
\int_\Omega v(x)\Delta u(x) \dd x
\\&\qquad+
\int_{\partial\Omega}
\left[
u(x)\nabla v(x)-v(x)\nabla u(x)
\right]
\cdot\nu(x)\dd S(x).
\end{split}
\end{equation}
Find or recall suitable integration by parts formulas.
\end{ex}

We will use exercise~\ref{ex:green-laplace} in our specific case.
Note that when $\Omega=\R^n\setminus\bar B(0,\eps)$, we may freely change the values of our functions near the origin to make them smooth.
We find
\begin{equation}
\begin{split}
Q(x,\eps)
&=
\int_{\R^n\setminus B(0,\eps)}
(\Delta_x f)(x-y)
\abs{y}^{2-n}
\dd y
\\&=
\int_{\R^n\setminus B(0,\eps)}
(\Delta_y f)(x-y)
\abs{y}^{2-n}
\dd y
\\&=
\int_{\R^n\setminus B(0,\eps)}
f(x-y)
\Delta_y\abs{y}^{2-n}
\dd y
\\&\quad+
\int_{\partial B(0,\eps)}\abs{y}^{2-n}\nabla f(x-y)\cdot y\abs{y}^{-1}\dd S(y)
\\&\quad+
\int_{\partial B(0,\eps)}f(x-y)\nabla\abs{y}^{2-n}\cdot y\abs{y}^{-1}\dd S(y)
.
\end{split}
\end{equation}
The exterior unit normal vector at $x\in\partial(\R^n\setminus\bar B(0,\eps))$ is $-x/\abs{x}$ and the gradient of $f(x-y)$ with respect to~$y$ is $-\nabla f$ evaluated at $x-y$.
This makes all the signs as they are.

This integral can be simplified significantly.

\begin{ex}
Show that $\nabla_y\abs{y}^{2-n}=(2-n)\abs{y}^{-n}y$.
\end{ex}

\begin{ex}
Show that $\Delta_y\abs{y}^{2-n}=0$ when $y\neq0$.
\end{ex}

\begin{ex}
Show that
\begin{equation}
\int_{\partial B(0,\eps)}\abs{y}^{2-n}\dd S(y)
=
a\eps
\end{equation}
for some constant $a>0$ depending on dimension.
Therefore the corresponding term vanishes as $\eps\to0$.
\end{ex}

The only term of~$Q(x,\eps)$ that does not vanish as $\eps\to0$ is
\begin{equation}
\begin{split}
&
\int_{\partial B(0,\eps)}f(x-y)\nabla\abs{y}^{2-n}\cdot y\abs{y}^{-1}\dd S(y)
\\&=
\int_{\partial B(0,\eps)}f(x-y)(2-n)\abs{y}^{1-n}\dd S(y)
\\&=
\int_{\partial B(0,\eps)}f(x-y)(2-n)\eps^{1-n}\dd S(y)
\\&=
\int_{\partial B(0,1)}f(x-\eps z)(2-n)\dd S(z).
\end{split}
\end{equation}
The last change of variables was made for the purpose of making the domain integrated over independent of the parameter~$\eps$.
As $\eps\to0$, we have $f(x-\eps z)\to f(x)$ uniformly for $z\in\bar B(0,1)$.

\begin{ex}
Collect the observations we have made and show that
\begin{equation}
c_2
\Delta I_2f(x)
=
(2-n)af(x).
\end{equation}
Therefore with a suitable choice of $b>0$ we have $-b\Delta I_2f=f$.
\end{ex}


We have proven a lemma for the inversion of~$I_2$:

\begin{lemma}
\label{lma:laplace-I2}
Let $n\geq3$.
There is a constant $b>0$ depending on~$n$ so that
\begin{equation}
-b\Delta I_2f=f
\end{equation}
for all $f\in C^2_c(\R^n)$.
\end{lemma}

This allows us to prove some injectivity results for Riesz potentials.

\begin{theorem}
If $n\geq3$, the Riesz potentials~$I_1$ and~$I_2$ are injective on the space~$C^2_c(\R^n)$.
\end{theorem}

For the purposes of memorising properties such as $I_\alpha I_\beta=I_{\alpha+\beta}$ and $-b\Delta I_2f=f$ it is useful to think of the Riesz potential a Fourier multiplier.
Every convolution operator is a Fourier multiplier due to the relationship between convolutions and Fourier transforms (exercise~\ref{ex:ft-prod-conv} gives a special case).
In the case of Riesz potentials the multiplier on the Fourier side is given in equation~\eqref{eq:ft-h-alpha}.
Multiplying by~$\abs{\xi}^\beta$ and then by~$\abs{\xi}^\alpha$ clearly amounts to multiplying simply by $\abs{\xi}^{\alpha+\beta}$.
The Laplace operator $-\Delta$ can also be seen as a Fourier multiplier~$\abs{\xi}^2$ (cf. exercise~\ref{ex:ft-der-poly}).
(Any constant coefficient differential operator corresponds to multiplication by a polynomial on the Fourier side.)
Therefore the powers cancel out in lemma~\ref{lma:laplace-I2}.
The reason we did not build all of this in section~\ref{sec:fourier-riesz} was that a more advanced approach (such as distribution theory) is needed to make sense of all this.
The resulting simplicity of things when seen from this point of view can hopefully serve as an incentive to get acquainted with the theory of distributions and their Fourier transforms.

\begin{ex}
Prove the theorem using lemmas~\ref{lma:I1I1=I2} and~\ref{lma:laplace-I2}.
\end{ex}

\qa

\section{Partial data}
\label{sec:partial}

In this section we will give our sixth injectivity proof based on Fourier analysis on~$\R^n$.
This method will also give a partial data result.

Some of the discussions in this section are vague as developing further would require substantial theory building.

\subsection{Various kinds of limitations}

In real life measurement situations there are various kinds of limitations to the measurements.
Sometimes one can fire X-rays through an object in any position and direction, but not always.
One might need to avoid hitting something sensitive, or the geometry of the measurement situation restricts the available directions.
In general, one might ask how large a set of lines is needed so that a function is uniquely determined by its integrals over them.

We have mostly studied X-ray tomography with full data so far.
We had one result with partial data so far, namely Helgason's support theorem (theorem~\ref{thm:helgason}), which concerns tomography around a convex obstacle.
In this section we will study a particular partial data scenario, where the set of directions of X-rays is not the whole sphere.
This is called limited angle tomography.

\subsection{Full data with the Fourier transform}

Before embarking on the study of partial data, let us first solve the simpler full data problem with these tools.
Consider a function $f\in C_c(\R^n)$ and define the X-ray transform as
\begin{equation}
\xrt f(x,v)
=
\int_\R f(x+tv)\dd t.
\end{equation}
We could also have used the notation~$\xrt_vf(x)$ as in section~\ref{sec:torus}.
We will restrict the parameters $(x,v)\in\R^n\times \Sphere^{n-1}$ so that $x\cdot v=0$.
In other words, $x\in v^\perp$, where~$v^\perp$ denotes the subspace orthogonal to~$v$.
Since $\xrt f(x+sv,v)=\xrt f(x,v)$ for any $s\in\R$, this restriction does not reduce our data.
In other words, $\xrt f(x,v)$ for all parameters $(x,v)\in\R^n\times \Sphere^{n-1}$ is uniquely determined by the restriction to $v\in \Sphere^{n-1}$ and $x\in v^\perp$.

Fix any $v\in \Sphere^{n-1}$ and consider the function $\xrt f(\dummy,v)$ on the $(n-1)$-dimensional space $v^\perp\subset\R^n$.
We can calculate the Fourier transform of $\xrt f(\dummy,v)$ on this space~$v^\perp$.
This function is continuous and compactly supported: $\xrt f(\dummy,v)\in C_c(v^\perp)$.
For $\xi\in v^\perp$, we denote
\begin{equation}
(\ft_{v^\perp}\xrt f(\dummy,v))(\xi)
=
\int_{v^\perp} e^{-i\xi\cdot x} \xrt f(x,v) \dd\h^{n-1}(x).
\end{equation}

\begin{ex}
\label{ex:og-fourier}
Fix any $v\in \Sphere^{n-1}$.
Suppose $\xi\in\R^n$ is orthogonal to~$v$.
Show that
\begin{equation}
(\ft_{v^\perp}\xrt f(\dummy,v))(\xi)
=
\ft f(\xi).
\end{equation}
This result is known as the Fourier slice theorem.
We saw a torus version of this in lemma~\ref{lma:ft-xrt-torus}.
\end{ex}

\begin{theorem}
\label{xrtthm:fourier}
If $f\in C_c(\R^n)$ integrates to zero over all lines, then $f=0$.
\end{theorem}

\begin{proof}
Take any $\xi\in\R^n$ and choose $v\in \Sphere^{n-1}$ so that $v\cdot\xi=0$.
Since $\xrt f(\dummy,v)=0$, exercise~\ref{ex:og-fourier} gives $\ft f(\xi)=0$.
Therefore $\ft f=0$, and by injectivity of the Fourier transform also $f=0$.
\end{proof}

In the plane both the unknown function~$f$ and its X-ray transform~$\xrt f$ can be written in polar coordinates as in section~\ref{sec:ang-fs}.
It is always a good idea to try a Fourier transform, and the function $\xrt f(r,\theta)$ has two variables it could be applied to.
The Fourier transform in the radius~$r$ leads to the Fourier slice theorem of exercise~\ref{ex:og-fourier}.
The Fourier transform in the angle~$\theta$ leads to the angular Fourier series we used in sections~\ref{sec:ang-fs} and~\ref{sec:abel}.
Different choices lead to different paths, but both lead to a proof.

\subsection{Limited angle tomography}

Now we turn to our partial data problem.
Let $D\subset \Sphere^{n-1}$ be the set of allowed directions.
The question is whether $f\in C_c(\R^n)$ is uniquely determined by~$\xrt f(x,v)$ for all $x\in\R^n$ and $v\in D$.

If $D=\Sphere^{n-1}$, then the result is stated in theorem~\ref{xrtthm:fourier} --- and our other injectivity theorems.
However, if~$D$ is finite, it turns out that there is an infinite dimensional subspace of functions $f\in C_c(\R^n)$ for which this data $\xrt f|_{\R^n\times D}=0$.
The simplest case is not hard to see.

\begin{ex}
Suppose $D=\{v\}$ is a singleton.
Show that there is a function $f\in C_c(\R^n)\setminus\{0\}$ for which $\xrt f(\dummy,v)=0$.
(Such functions can be constructed starting from any $g\in C_c^\infty(\R^n)$ which is not identically zero, in a couple of ways. Or you can choose to start in one dimension and then promote into higher dimensions.)
\end{ex}

The general case for any finite set follows from a convolution argument:

\begin{ex}
Fix any $v\in \Sphere^{n-1}$.
Show that
\begin{equation}
\xrt(f*g)(\dummy,v)=(\xrt f(\dummy,v))*g   
\end{equation}
for all $f,g\in C_c(\R^n)$.

It then follows that if $\xrt f_1(\dummy,v_1)=0$ and $\xrt f_2(\dummy,v_2)=0$, then
\begin{equation}
\xrt(f_1*f_2)(\dummy,v)=0
\end{equation}
for both $v\in\{v_1,v_2\}$.
The only thing left to worry about is that the convolution of two non-trivial compactly supported functions cannot vanish identically, but we shall not worry about it here.
(This is indeed true without any further assumptions.)
The same idea generalizes to any finite number of directions.
How would you construct a non-trivial function whose X-ray transform vanishes in all directions in $D=\{v_1,v_2,v_3,v_4\}$?
\end{ex}

Let us denote
\begin{equation}
D^\perp
=
\{\xi\in\R^n;\xi\cdot v=0\text{ for some }v\in D\}.
\end{equation}
Notice that~$D^\perp$ is not the orthogonal complement of the linear space spanned by~$D$.
Instead, it is the union of the orthogonal complements of the elements in~$D$.
In general~$D^\perp$ is not a vector space.\footnote{Usually for a set $D\subset E$ in a real inner product space,~$D^\perp$ denotes the intersection of all~$v^\perp$ for $v\in D$. This is a vector subspace. Now, however,~$D^\perp$ denotes the union rather than the intersection. This is not usually a vector subspace. The reader is advised to be alert with the symbol~$\perp$.}

Now we are ready to state and prove our theorem:

\begin{theorem}
\label{thm:limited-angle}
Let $f\in C_c(\R^n)$.
Suppose $D\subset \Sphere^{n-1}$ is such that $D^\perp\subset\R^n$ contains an interior point.
If $\xrt f(x,v)=0$ for all $x\in\R^n$ and $v\in D$, then $f=0$.
\end{theorem}

\begin{proof}
Take any $\xi\in D^\perp$.
Then there is $v\in D$ so that $v\cdot\xi=0$.
Since $\xrt f(\dummy,v)=0$, exercise~\ref{ex:og-fourier} gives $\ft f(\xi)=0$.
Therefore the Fourier transform~$\ft f$ vanishes in~$D^\perp$.

The continuous function~$f$ is compactly supported, so by theorem~\ref{thm:pw} the Fourier transform~$\ft f$ is real analytic.
By assumption~$D^\perp$ contains a non-empty open set, and~$\ft f$ vanishes in it.
Now exercise~\ref{ex:analytic-open} implies that~$\ft f$ has to vanish identically due to analyticity.
Since the Fourier transform is injective as mentioned in theorem~\ref{thm:ft}, we conclude that the function~$f$ vanishes identically.
\end{proof}

A new question arises:
How much is needed about the set~$D$ of admissible directions to ensure that~$D^\perp$ contains an interior point?
We will look at a couple of examples.

First, it is clear that~$D$ needs to be uncountable.
If~$D$ is countable, then~$D^\perp$ is a union of countably many hyperplanes, and such a union cannot have interior points.
By a simple approximation argument one can replace~$D$ with~$\bar D$, so we may in fact assume that~$D$ is closed if we want to.

If~$D$ contains an interior point, so does~$D^\perp$ (exercise~\ref{ex:perp-interior}).
In fact, much less is needed, as exercise~\ref{ex:light-cone-normals} shows.

\begin{ex}
\label{ex:light-cone-normals}
Consider the spacetime $\R^4=\R^3\times\R$ and imagine that measurements are only done along light rays.
In natural units ($c=1$) this means that our set~$D$ is
\begin{equation}
\{
(v_1,v_2)\in\R^3\times\R
;
\abs{v_1}^2+\abs{v_2}^2=1,\abs{v_1}=\abs{v_2}
\}
.
\end{equation}
The restriction to unit length is just a feature of our framework and unimportant for this case; the important restriction is $\abs{v_1}=\abs{v_2}$.
This condition defines the light cone (the set of lightlike directions).
Show that
\begin{equation}
D^\perp
=
\{(\eta_1,\eta_2)\in\R^3\times\R;\abs{\eta_1}\geq\abs{\eta_2}\}.
\end{equation}
This contains the spacelike and lightlike directions but not timelike.
\end{ex}

\begin{ex}
\label{ex:perp-interior}
Show that if $v\in \Sphere^{n-1}$ is an interior point of~$D$, then any non-zero $\xi\in\R^n$ orthogonal to~$v$ is an interior point of~$D^\perp$.
\end{ex}

\subsection{On stability and singularities}
\label{sec:stab-sing}

We found out above that~$D^\perp$ having an interior point is sufficient for injectivity.
However, it is insufficient for stability.
Stability (of Lipschitz type) would mean an estimate of the kind
\begin{equation}
\aabs{f}
\leq
C
\aabs{\xrt f}
\end{equation}
with suitable norms and~$\xrt f$ restricted to the set where data is available.
No matter which Sobolev norms one chooses for functions on~$\R^n$ and the relevant subset of~$\Gamma$, there is no continuous left inverse for the partial data X-ray transform.
Stable inversion is possible if $D=\Sphere^{n-1}$, and even if $D^\perp=\R^n$.

The reason for instability is that some kinds of singularities are undetected.
Using microlocal analysis one can have very fine control over singularities, and it is possible to ask whether a distribution is smooth or singular at a given point in a given direction.
This requires the introduction of a wave front set.
To be able to stably detect a singularity at a point $x\in\R^n$ in direction $v\in \Sphere^{n-1}$, the data must contain a line through~$x$ in a direction orthogonal to~$v$.
The smooth part of the data contains information about the singularities, sometimes enough for uniqueness, but any reconstruction is necessarily unstable.
A precise formulation of results of this kind of result is way beyond our reach here.

In our limited angle tomography situation the condition mentioned above amounts to $D^\perp=\R^n$.
In this case one can reconstruct the Fourier transform everywhere directly, and invert the Fourier transform to obtain the original function.
This is stable.
On the other hand, if~$D^\perp$ contains interior points but is not the whole space, then the data needs to be analytically continued (by virtue of exercise~\ref{ex:analytic-open}), and analytic continuation is unstable without strong a priori estimates.

\begin{ex}
Let $\Gamma'\subset\Gamma$ be a set of unoriented lines in~$\R^n$.
Suppose~$\Gamma'$ satisfies the stability condition mentioned above:
For every $x\in\R^n$ and $v\in \Sphere^{n-1}$ there is a line $\gamma\in\Gamma'$ going through~$x$ in a direction orthogonal to~$v$.
Show that if $n=2$, then $\Gamma'=\Gamma$, but if $n\geq3$, this is not necessarily the case.
\end{ex}

For more details on stable recovery and visible singularities, the reader is advised to see~\cite{KQ,Q}\footnote{There is also a book ``Microlocal Analysis and Integral Geometry'' by Stefanov and Uhlmann, currently under preparation.}.
\todo{Update when done!}

\subsection{Local reconstruction}
\label{sec:loc-reconstruction}

Another interesting question is whether a function can be reconstructed at a point from integrals of lines near that point.
More precisely, let $x\in\R^n$ and let $U\ni x$ be a neighborhood (a region of interest).
Does the knowledge of $\xrt f(\gamma)$ for all~$\gamma$ that meet~$U$ determine~$f|_U$ for some class of functions~$f$?

It turns out that this is not possible without prior knowledge of $f|_U$.
However, this data is enough to detect the singularities of~$f|_U$.
That is, one can locally reconstruct jumps and other singularities accurately, but not a smooth function.
Microlocal reconstruction is possible, local is not.
In many practical applications it is indeed important to find the singularities of the unknown to identify sharp features, and it is not a big issue if the smooth part remains beyond reach.

\begin{ex}
\label{ex:roi-normal}
Let $f\in C_c(\R^n)$ and let $U\subset\R^n$ be an open set.
Explain why the integrals of~$f$ over all lines that meet~$U$ determine~$(\xrt^*\xrt f)|_U$.
\end{ex}

However, local reconstruction is possible for the Radon transform in three dimensions.
Consider a point $x\in\R^3$ and a neighborhood $U\ni x$.
Then the integrals of $f\in C_c(\R^3)$ over all hyperplanes that meet~$U$ determine~$f|_U$.
As in exercise~\ref{ex:roi-normal}, this data determines~$R^*Rf|_U$, where~$R$ stands for the Radon transform.
It turns out that for some constant $c>0$ we have $-c\Delta R^*Rf=f$, where~$\Delta$ is the Laplace operator, so that~$f$ can be recovered from~$R^*Rf$ by differentiation.

In~$\R^n$, the normal operator of the X-ray transform can be inverted by the non-local operator~$(-\Delta)^{1/2}$ and that of the Radon transform by $(-\Delta)^{(n-1)/2}$.
Local reconstruction is possible for the Radon transform in odd dimensions starting at three when the exponent $(n-1)/2$ is an integer.
Non-integer powers of the Laplace operator can be defined via Fourier transform.

\begin{ex}
We saw above that local reconstruction for the Radon transform is possible in~$\R^3$.
On the other hand, we saw in exercise~\ref{ex:radon-xrt} that injectivity of the Radon transform implies injectivity for the X-ray transform.
In that exercise we essentially constructed an operator~$A$ so that $R=A\xrt$ and thus the X-ray transform can be inverted by $f=-c\Delta R^*A\xrt f$.
Why does this not lead to local reconstruction for the X-ray transform in~$\R^3$?
\end{ex}

\qa

\section{Outlook}
\label{sec:outlook}

To conclude the course we review some directions of further study on the subject.
The written descriptions are brief to avoid veering to far on a sidetrack; the interested reader is invited to consult the references or the instructor for further details.
The statements will not be made fully precise; the purpose is to give a flavor of what is known, not all details.
Some results have been weakened for technical convenience.
But before looking further, let us summarize the course briefly.

\subsection{Overview of the course}

It is now time to look back and see what, if anything, we have accomplished during the course.
The course started with the physical problem of X-ray tomography and its mathematical formulation.
We then proved a uniqueness result in six different ways.
These results are collected in theorems
\ref{xrtthm:torus},
\ref{xrtthm:cormack},
\ref{xrtthm:radon},
\ref{xrtthm:pestov},
\ref{xrtthm:riesz}, and
\ref{xrtthm:fourier}.
The inversion methods presented here, however, do not exhaust all the known ones.

Our methods had various different assumptions, but they all proved this:

\begin{theorem}
Suppose $f\in C^\infty_c(\R^n)$, $n\geq2$. If $\xrt f=0$, then $f=0$.
\end{theorem}

In addition, we proved a number of related results.
We proved an injectivity result for the X-ray transform on tori (theorem~\ref{thm:xrt-torus}) and solenoidal injectivity for vector field tomography (theorem~\ref{thm:vf}).
We also gave a proof of two partial data results, namely Helgason's support theorem (theorem~\ref{thm:helgason}) and a result in limited angle tomography (theorem~\ref{thm:limited-angle}).

\begin{ex}
We gave six uniqueness proofs for the X-ray transform.
Give a quick overview of the mathematical tools needed for each of the six proofs.
Give six lists, one for each proof.
Which proofs did you find most accessible and simple, and which ones hardest to follow?
\end{ex}

\subsection{Geodesic X-ray tomography}

So far we have studied functions in Euclidean domains and integrated them over straight lines.
But what if the domain is replaced by a Riemannian manifold with boundary and lines by geodesics?
Most of our methods will be inapplicable on manifolds, but not all.

Most of our tools (Fourier series, Fourier transform, convolutions, polar coordinates) fail on a general manifold.
If the manifold happens to be spherically symmetric, then our radial Fourier series approach works:

\begin{theorem}
Equip the closed unit ball $\bar B\subset\R^n$ with a rotation symmetric Riemannian metric so that every maximal geodesic meets the boundary and therefore has finite length.
Such manifolds are called non-trapping.
On a non-trapping spherically symmetric Riemannian manifold a function is uniquely determined by its integrals over all geodesics.
\end{theorem}

Spherical symmetry is a strong requirement.
Our proof idea with the sphere bundle and the Pestov identity works in more generality:

\begin{theorem}
A compact manifold with boundary is called simple if any two points can be joined with a unique geodesic and the geodesic depends smoothly on its endpoints.
On a simple Riemannian manifold a function is uniquely determined by its integrals over all geodesics.
\end{theorem}

The second theorem does not contain the first one; there are non-simple but non-trapping rotation symmetric manifolds.


\subsection{Tensor tomography}
\label{sec:outlook-tt}

A scalar function can be replaced with a tensor field of any rank.
So far we have studied only rank zero (scalar fields) and rank one (vector fields).
To go further, one must first understand what a tensor field is in the first place, and then figure out how to integrate them along lines (or geodesics).
As in the case of vector fields, there is non-uniqueness for any non-zero rank.
The goal is then to characterize this non-uniqueness.

A symmetric rank~$m$ tensor field on~$\R^n$ is a function $f\colon\R^n\times\R^{nm}\to\R$ so that $f(x;v_1,\dots,v_m)$ is smooth in~$x$, linear in each~$v_i$ and invariant under changes of any two~$v_i$ and~$v_j$.
The integral of such a tensor field over a line $\gamma\colon\R\to\R^n$ is
\begin{equation}
\int_\gamma f
=
\int_\R f(\gamma(t);\dot\gamma(t),\dots,\dot\gamma(t))\der t.
\end{equation}
That is, the velocity~$\dot\gamma$ is plugged into each of the~$m$ slots.

A scalar function is a tensor field of rank zero, and a vector field is a tensor field of rank one.

If $m=1$, there is only one slot for~$v$, and this is the integral of a vector field along a line as defined earlier.
If $m=0$, there are no slots at all and the function only depends on~$x$.
The resulting integral is the usual integral of a scalar function we have studied in this course.

\begin{theorem}
On a two-dimensional simple manifold a tensor field~$f$ of order~$m$ integrates to zero over all geodesics if and only if there is a tensor field~$h$ of order $m-1$ which vanishes at the boundary and satisfies $f=\der^s h$, where~$\der^s$ is a symmetrized covariant derivative.
\end{theorem}

\subsection{Linearization of travel time}

For one specific example of applications, we can consider travel time tomography in seismology.
The problem can be recast as a geometrical one, once the Earth is treated as a geometrical object.
The linearized problem has non-uniqueness, but it corresponds to the non-uniqueness inherent to the geometrical problem.
This has physical repercussions.

\begin{theorem}
\label{thm:outlook-lin}
Let~$M$ be a manifold and~$g_s$ a family of Riemannian metrics on it, depending on a parameter $s\in\R$.
Consider two points $x,y\in M$.
Let~$\gamma_s$ be a geodesic with respect to the metric~$g_s$ joining these two points.
Suppose~$\gamma_s$ depends smoothly on~$s$ and denote the length of~$\gamma_s$ by~$\ell_s(\gamma_s)$.
Denote by $f_s=\partial_sg_s$ the second order tensor field obtained by differentiating the metric with respect to the parameter.
Then
\begin{equation}
\frac{\der}{\der s}\ell_s(\gamma_s)
=
\frac12\int_{\gamma_s}f_s.
\end{equation}
\end{theorem}

The boundary distance function of a manifold~$M$ with boundary is the restriction of the distance function $d\colon M\times M\to\R$ to $\partial M\times\partial M$.
Then one can ask whether the Riemannian manifold $(M,g)$ is uniquely determined by its boundary distance function.
This problem is hard and non-linear, but the linearization is simpler.

Linearized travel time tomography is tensor tomography of rank two.
If one studies conformal variations of a metric, then each~$g_s$ is a conformal multiple of~$g_0$.
In that case~$f_s$ is also a conformal multiple of~$g_s$, and the tensor tomography problem reduces to a scalar tomography problem for the conformal factor.

Both the original problem and the linearized one have a gauge freedom --- an inherent invisibility or invariance.
In the original problem one can change coordinates on~$M$ by a diffeomorphism $\phi\colon M\to M$ and change the metric from~$g$ to~$\phi^*g$ and the data stays the same as long as $\phi|_{\partial M}=\id$.
In the linearized problem one can only hope to reconstruct the metric perturbation~$f_s$ up to tensor fields of the form~$\der^sh$, where~$h$ is a covector field (one-form) vanishing at the boundary.
It turns out that this gauge freedom in the linearized problem is the linearization of the gauge freedom of coordinate changes in the non-linear one.

We also discussed a linearized travel time tomography problem in section~\ref{sec:doppler-application}.

\subsection{Other classes of curves}

So far we have only integrated functions (and vector fields) over straight lines.
We also mentioned Riemannian geodesics, but there several other options as well.
Of course, this can be seen as the mathematical art of (over)generalization, but a great number of different geometrical situations turn out to be physically relevant.

As a broad term, inverse problems in integral geometry ask to recover an object from its integrals over a collection of subsets of the space.
Integral geometry questions are often asked in a differential geometric setting, but the apparent duality of the concepts is coincidental.

Some classes of curves to consider:
\begin{itemize}
\item Lines in~$\R^n$.
\item Circles in~$\R^2$.
\item Geodesics on a Riemannian manifold.
\item Magnetic geodesics on a manifold or~$\R^n$.
\item Geodesics on a Finsler manifold.
\item Integral curves of a dynamical system.
\item Curves which reflect or split in some way.
\end{itemize}

\subsection{Further reading}

The following selection of books, lecture notes and review articles is by no means complete.
The sources listed here are available online.

\begin{itemize}
\item Gunther Uhlmann, ``Inverse problems: seeing the unseen'' \cite{U:ip}: An overview of two important inverse problems, namely travel time tomography and Calder\'on's problem. Both are related to X-ray transforms of some kind.
\item Sigur\dh{}ur (Sigurdur) Helgason, ``The Radon Transform'' and ``Integral Geometry and Radon Transforms'' \cite{H:rt,H:igrt}: The books are freely available on the author's homepage. They give a very thorough treatment of the theory of X-ray transforms and particularly Radon transforms.
\item Vladimir Sharafutdinov, ``Ray Transform on Riemannian Manifolds. Eight Lectures on Integral Geometry'' \cite{S}: Lecture notes on X-ray tomography on Riemannian manifolds.
\item Gabriel Paternain, Mikko Salo, and Gunther Uhlmann, ``Tensor tomography: progress and challenges'' \cite{PSU} and Joonas Ilmavirta and Fran\c{c}ois Monard, ``Integral geometry on manifolds with boundary and applications'' \cite{IM}: Review articles on on X-ray tomography on manifolds.
\item Gabriel Paternain, Mikko Salo, and Gunther Uhlmann, ``Geometric inverse problems, with emphasis on two dimensions'' \cite{PSU-book}: A book covering tensor tomography on manifolds and related topics. Freely available as a digital copy.
\item Will Merry and Gabriel Paternain, ``Inverse Problems in Geometry and Dynamics'' \cite{MP}: Lecture notes on geodesic X-ray tomography with a dynamical focus.
\item Gunther Uhlmann and Hanming Zhou, ``Journey to the Center of the Earth'' \cite{UZ}: A review article of travel time tomography, which is closely related to tensor tomography.
\end{itemize}

\qa

\subsection{Feedback}

This course is somewhat unusual, compared to most courses in mathematics.
To make the course more suitable for students in the future, the last exercises concern the course itself.

\begin{ex}
This course introduced a physical problem and introduced a number of mathematical tools related to that application.
Should there have been more focus on physics --- more applications, more details, deeper explanations, or something else?
Was the balance between physics and mathematics good for your interests?
(Also, what is your major?)
\end{ex}

\begin{ex}
The course was designed to be broad but shallow.
We discussed a number of different mathematical tools and ideas related to X-ray tomography, but we did not go very deep into any of them.
This was done so as to give you a broad overview of the topic and an example of how various different tools in analysis can be used to tackle the same applied problem.
On the other hand, we could have developed a theory with optimal regularity (not restricting to continuous functions all the time), stability estimates, range characterizations, various function spaces and the like.
Should the course have been deeper as opposed to broad?
Would you have preferred to see one theory developed in detail rather than several independent ideas?
It would also be possible to give a broad introductory course like this one and then a deeper follow-up course.
\end{ex}

\begin{ex}
Answer the following questions:
\begin{enumerate}[(a)]
\item How much time did you spend on the course?
\item What were the worst things about the course?
\item Would you be interested in a course in the geometry of geodesics and geodesic X-ray tomography on Riemannian manifolds?
\item Do you have any feedback in mind that was not covered by other questions?
\end{enumerate}
Thank you for your feedback!
\end{ex}

Previous feedback has been of immense help in improving these notes.
Many thanks for all students who contributed!

\appendix

\section{Notation}

If you feel some notation is perplexing and should be either explained more or added here, let the author know.
Notation introduced in section~\ref{sec:outlook} has been excluded.

\subsection{Sets and spaces}

\begin{itemize}
\item $\R^n$: Our Euclidean spaces have dimension $n\geq2$ unless otherwise mentioned.
\item $\Sphere^n$: The unit sphere, the subset of $\R^{n+1}$ consisting of unit length vectors.
\item $\T^n=\R^n/2\pi\Z^n$: The $n$-dimensional torus.
\item $\Omega$: A subset of~$\R^n$, typically open.
\item $\gamma$: A straight line or a geodesic.
\item $\Gamma$: The set of all lines; see section~\ref{sec:measures} for integration.
\item $\pud$: The punctured unit disc; see~\eqref{eq:pud}.
\item $S\Omega$: The sphere bundle; see section~\ref{sec:geod-geom}.
\item $\inwb$: The inward pointing boundary of the sphere bundle; see section~\ref{sec:geod-geom}.
\item $v^\perp$: The subspace orthogonal to~$v$ or the rotation of~$v$.
\end{itemize}

\subsection{Derivatives}

\begin{itemize}
\item $\nabla f$: The gradient.
\item $\nabla\cdot f$: The divergence.
\item $\nabla\times f$: The curl.
\item $\dot\gamma(t)$: The derivative of a curve~$\gamma(t)$ with respect to~$t$.
\item $X$: The geodesic vector field; see section~\ref{sec:geod-geom}.
\item $V$: The vertical vector field; see section~\ref{sec:2D-SM}.
\item $X_\perp$: The horizontal vector field; see section~\ref{sec:2D-SM}.
\item $\partial^\alpha=\partial_1^{\alpha_1}\cdots\partial_n^{\alpha_n}$: The derivative of order $\alpha\in\N^n$. The factorial of a multi-index is $\alpha!=\alpha_1!\cdots\alpha_n!$.
\item $\Delta=\partial_1^2+\cdots+\partial_n^2$: The Laplace operator.
\end{itemize}

\subsection{Function spaces}

\begin{itemize}
\item $C(X;Y)$: The space of continuous functions $X\to Y$.
\item $C(X)$: The space $C(X;Y)$ with $Y=\R$ or $Y=\C$ depending on context. The norm on this space is $\aabs{f}=\sup_{x\in X}\abs{f(x)}$.
\item $C^k(X;Y)$: The space of~$k$ times continuously differentiable functions.
\item $C^k_c(X;Y)$: The space of~$k$ times continuously differentiable functions with compact support.
\item $C_B$: The space of continuous functions on~$\R^n$ supported in the closed unit ball~$\bar B$.
\item $C_b(X)$: The space of continuous and bounded functions on~$X$.
\item $L^p(X)$ (where usually $X\subset\R^n$):
\begin{itemize}
\item If $1\leq p<\infty$: The space of (Lebesgue-measurable) functions $f\colon X\to\C$ (or to~$\R$) for which $\int_X\abs{f(x)}^p\dd x<\infty$. The norm on this space is $\aabs{f}=\left(\int_X\abs{f(x)}^p\dd x\right)^{1/p}$.
\item If $p=\infty$: The space of bounded (Lebesgue-measurable) functions $f\colon X\to\C$ (or to~$\R$). The norm on this space is $\aabs{f}=\sup_{x\in X}\abs{f(x)}$.
\end{itemize}
\item $\ell^2(I)$: The space of all functions $a\colon I\to\C$ (or to~$\R$) that satisfy $\sum_{i\in I}\abs{a(i)}^2<\infty$. This is an inner product space, and the norm on this sequence space is $\aabs{a}=\left(\sum_{i\in I}\abs{a(i)}^2\right)^{1/2}$.
(Here~$I$ can be any set, thought of as a set of indices. The functions $I\to\C$ are thought of as functions and often denoted $a(i)=a_i$.)
\end{itemize}

\subsection{Integral operators}

\begin{itemize}
\item $\xrt$: The X-ray transform. The basic definition was given in~\ref{def:xrt}. There are various different incarnations throughout the course; see these lecture notes.
\item $\xrt_v$: The X-ray transform indexed by a direction~$v$; see section~\ref{sec:torus}.
\item $\ft$: The Fourier transform or series. (Not denoted by hats in this course.)
\item $R$: The Radon transform; see section~\ref{sec:radon-transform}.
\item $\A_k$: The generalized Abel transform; see section~\ref{sec:abel}.
\item $I_\alpha$: The Riesz potential; see sections~\ref{sec:normal} and~\ref{sec:riesz}.
\end{itemize}

\subsection{Miscellaneous}

\begin{itemize}
\item $\spt f$: The support of a function, the closure of $f^{-1}(\R\setminus\{0\})$.
\item $\h^d$: The Hausdorff measure of dimension~$d$.
\item $q$: The quotient map $\R^n\to\T^n$; see section~\ref{sec:torus}.
\item $\rot_\phi$: The rotation by angle~$\phi$; see section~\ref{sec:ang-fs}.
\item $\phi_t$: The geodesic flow; see section~\ref{sec:geod-geom}.
\item $\pi$: The projection $S\Omega\to\Omega$; see section~\ref{sec:geod-geom}.
\item $\pi^*$: The pullback. In this course simply $\pi^*f=f\circ\pi$.
\item $A^*$: The (formal) adjoint of a linear operator~$A$; see section~\ref{sec:normal}.
\item $[A,B]=AB-BA$: The commutator of linear operators~$A$ and~$B$; see section~\ref{sec:commutators}.
\end{itemize}

\end{document}